%% file: main.tex
\documentclass{amsart}
\input{matroidpreamble.tex}
\input{leanpreamble.tex}
\newcommand{\gutsproj}[3]{\mathsf{g}_{#2}^{#3}(#1)}
\newcommand{\gutscut}[2]{{\mathcal{G}}_{#2}(#1)}
\newcommand{\relrank}[3]{\rho_{#1}({#2};{#3})}

\usepackage[utf8]{inputenc}

\title{Modularity, Extensions, and Connectivity in Infinite Matroids}
\author[Ehatamm]{Mattias Ehatamm}
\author[Nelson]{Peter Nelson}
\author[Rivera Oma\~na]{Fernanda Rivera Oma\~na}

\begin{document}
\thanks{This research was partially supported by an NSERC Discovery Grant}
\sloppy
\begin{abstract}
We generalize the well-studied notion of a \emph{modular pair} of a finite matroid to arbitrary families of sets in 
infinite matroids, and use it to develop the theory of infinite matroids in several 
as-yet-unexplored areas. Our results include a complete theory of single-element extensions, 
a description of the relationship between quotients and projections, 
a proof that matroids for which every flat is modular must be finitary, 
and two new perspectives on the infinite matroid connectivity parameter $\lambda$. 
In most cases, existing theory for finite matroids either fails completely
or does not extend in obvious ways, and as a result we develop multiple new techniques for 
reasoning about infinite matroids, including establishing well-behaved
infinite analogues of nullity, local connectivity and skewness. 

We also point to an online repository containing formalized proofs of 
all our results using the \texttt{lean4} proof assistant. 
\end{abstract}
\maketitle

\section{Introduction}

A defining property of finite matroids is submodularity: every pair of sets $X, Y$ in a finite matroid $M$
satisfies the inequality $r_M(X) + r_M(Y) \ge r_M(X \cup Y) + r_M(X \cap Y)$. 
If this holds with equality, then $(X,Y)$ is a \emph{modular pair} in $M$. 
A flat of $M$ is \emph{modular} if it forms a modular pair with every other flat of $M$, 
and a matroid itself is \emph{modular} if all of its flats are modular. 

Modularity is important in many parts of matroid theory; 
for example, a single modular flat can force large amounts of structure (such as representability)
in an entire matroid with only mild connectivity assumptions [\ref{gk13},\ref{k14}]. 
More generally, modularity links matroid theory to lattice theory: 
a matroid is modular if and only if its lattice of flats is a modular lattice, 
which is in turn essentially equivalent to $M$ being a finite projective geometry 
(see Theorem~\ref{modularpg}). 
Modularity is also closely related to the embedding of matroids in larger matroids:
specifically, modular pairs are fundamental to the `modular cuts' underlying 
Crapo's theory of single-element matroid extensions [\ref{crapo65}].

Our goal is to study the interplay between modularity, 
extensions and connectivity in the setting of infinite matroids. 
Most of the existing literature on these subjects for finite matroids 
makes heavy use of the rank function, 
and (more or less) the techniques and definitions fail to generalize correctly to the infinite setting. 
Because of this,  
we need to do nontrivial work to recover definitions and prove basic 
properties that are almost immediate for finite-rank matroids. 

Despite these difficulties, we develop versions of modularity (in all senses of the word), 
the theory of matroid extensions, skewness, nullity, local connectivity,
quotients and projections 
that both behave well for infinite matroids and specialize correctly to finite matroids. 
Roughly the first half of this paper is devoted to defining these notions
and proving the basic properties needed to reason with them.

The later sections then apply this theory to derive new results on matroid extensions, 
quotients, modularity and connectivity, including some theorems 
that are new even for finite matroids. 
Going forward, we use the term `matroid' to mean a possibly infinite matroid 
in the sense of [\ref{bdkpw}], explicitly qualifying matroids as `finite' where 
this is occasionally needed. 

\subsection*{Formalization}

Though this is not our focus, 
the material in this article all arose from an effort to produce 
a body of formalized matroid theory in the \smalltt{lean4} theorem prover [\ref{lean4}]: 
that is, proofs written in a programming language whose correctness can be checked 
by computer.
In the name of generality, we allowed matroids in these proofs to be infinite,
so a lot of `theory-building' code was required to go from first principles
to a working coverage of matroid theory.  
Often this amounted to producing formal versions of existing theories, 
but not always; this paper showcases some of the the places where the formalization
process yielded new mathematics.

At the time of writing, formalized proofs of the results in this paper 
are a small fraction of the second author's \href{https://github.com/leanprover-community/mathlib4}{\smalltt{Matroid}}\footnote{\url{https://github.com/apnelson1/Matroid}} repository [\ref{matroidrepo}].
In turn, this repository implicitly relies on material
in the much larger community-maintained \href{https://github.com/leanprover-community/mathlib4}{\smalltt {mathlib4}} 
~\footnote{\url{https://github.com/leanprover-community/mathlib4}} 
library of formalized mathematics [\ref{mathlibpaper},\ref{mathlib}], 
including both work in matroid theory that originally came from [\ref{matroidrepo}], 
and general mathematical infrastructure written by the larger community. 

As a consequence, an interested reader can verify all our theorems' correctness by 
simply downloading the relevant software, checking that the formal definitions
and theorem statements correspond to their counterparts as described in the writing, 
and expending minimal computing power to ensure that all formal proofs
are accepted by the proof assistant as valid. 

We discuss all this is much more detail in the appendices,
but until then, we do not mention it further; 
all our mathematical prose, including the proofs, 
is written with the usual style and rigour of a mathematics article.

\subsection*{Modular Pairs}

The definition of a modular pair in a finite matroid $M$ 
(that is, a pair of sets $X,Y$ for which $r_M(X) + r_M(Y) = r_M(X \cup Y) + r_M(X \cap Y)$)
is not fit for purpose in matroids of infinite rank.
Although the rank function of an infinite matroid, 
which takes values in $\enn = \bN \cup \{\infty\}$, is well-defined and even submodular, 
it is too coarse to capture all the properties it does for finite matroids;
for instance, the above definition would give the undesirable statement that $(X,Y)$ is modular whenever $X$ or $Y$ has infinite rank. 

A definition that works better for infinite matroids comes from the following routine observation. (See Lemma~\ref{modpairiffbasis} for a proof.)

\begin{lemma}\label{modpairiffbasisintro}
  A pair $(X,Y)$ is modular in a finite matroid $M$ if 
  and only if $M$ has an independent set that contains bases for both $X$ and $Y$. 
\end{lemma}

The existence of such an independent set is a well-behaved property for infinite matroids, 
and will be important enough that we give it a name. 

\begin{definition}
  If $\cX$ is a family of sets in a matroid $M$, then an independent set $B$ of $M$ is a \emph{mutual basis}
  for $\cX$ in $M$ if $B$ contains bases for every $X \in \cX$. 
\end{definition}

We can use this notion to define a number of senses of the word `modular', 
all of which will behave well for infinite matroids, 
and specialize to the existing well-established notions of the same names for finite matroids. 

\begin{definition}\label{modpairdef}
  Let $M$ be a matroid. 
  \begin{itemize}
    \item A family $\cX$ of sets in $M$ is \emph{modular} if $M$ has a mutual basis for $\cX$;
    \item a pair $(X,Y)$ of sets is \emph{modular} if $\{X,Y\}$ is a modular family;
    \item a flat $F$ of $M$ is \emph{modular} if the pair $(F,G)$ is modular for every flat $G$;
    \item the matroid $M$ is \emph{modular} if every flat of $M$ is modular. 
  \end{itemize} 
\end{definition}

With these definitions stated, we can summarize our main results. 

\subsection*{Extensions}
Our first application of modularity is in the description of single-element extensions. 
If $M$ is a matroid, and $e \notin E(M)$, then an \emph{extension of $M$ by $e$}
is a matroid $M'$ such that $M = M' \del e$. 
Crapo's description of single-element extensions of matroids [\ref{crapo65}] relies on modular pairs 
via the following seminal definition.

\begin{definition}\label{finmodcut}
A \emph{modular cut} in a finite matroid $M$ is a collection $\cF$ of flats of $M$
such that
\begin{enumerate}[(i)]
  \item\label{finsuperflat} 
    For each $F \in \cF$, every flat $F'$ of $M$ with $F \subseteq F'$ satisfies $F' \in \cF$. 
  \item\label{forallfinmodpair} 
    For very modular pair $(F,F')$ of $M$ with $F, F' \in \cF$, we have $F \cap F' \in \cF$. 
\end{enumerate}
\end{definition}

Given an extension $M'$ of $M$ by $e$, 
write $\cF_M^{M'}$ for the collection of flats of $M$ whose closure in $M'$ contains $e$. 
The following two results from [\ref{crapo65}] 
show that there is a bijective correspondence between modular cuts and 
single-element extensions. 

\begin{theorem}\label{finmodcutdeletion}
  If $M'$ is an extension of a finite matroid $M$ by $e$, then $\cF_M^{M'}$ is a modular cut of $M$. 
\end{theorem}
 
\begin{theorem}\label{finmodcutextension}
  Let $\cF$ be a modular cut in a finite matroid $M$, and suppose that $e \notin E(M)$. 
  Then $M$ has a unique extension $M'$ by $e$ such that $\cF = \cF_M^{M'}$. 
\end{theorem}



We show that in the infinite setting, Definition~\ref{finmodcut} is not sufficient for 
this theory to work, even with our amended definition of `modular pair'. 
Specifically, every infinite-rank matroid has a collection of flats
that is a modular cut in the finite sense, but does not correspond to any extension. 

\begin{theorem}
  Every matroid $M$ of infinite rank has a collection $\cF$ of flats
  satisfying (i) and (ii) of Definition~\ref{finmodcut}, 
  for which the statement of Theorem~\ref{finmodcutextension} fails. 
\end{theorem}

Happily, we can rectify this with a modified definition of modular cuts for infinite matroids,
where the objects satisfying the new definition do precisely correspond to single-element extensions. 
The modification we need is to insist that $\cF$ is closed under taking intersections of `modular chains'
as well as modular pairs. 
(A \emph{chain} is a collection $\cC$ of sets such that all $X,Y \in \cC$ satisfy $X \ss Y$ or $Y \ss X$.)

\begin{definition}\label{infmodcut}
  A collection $\cF$ of flats in a matroid $M$ is a \emph{modular cut} if
  \begin{enumerate}[(i)]
  \item\label{superflat} 
    For each $F \in \cF$, every flat $F'$ of $M$ with $F \subseteq F'$ satisfies $F' \in \cF$. 
  \item\label{forallmodpair}
    For every modular pair $(F,F')$ of $M$, if $F, F' \in \cF$, then $F \cap F' \in \cF$.
  \item\label{forallmodchain} 
    Each infinite chain $\cC$ of flats in $M$ that is a modular family satisfies
    $\bigcap \cC \in \cF$. 
  \end{enumerate}
\end{definition}

Since finite matroids (and indeed, finite-rank matroids) have no infinite chains of flats,
this definition readily specializes to Definition~\ref{finmodcut}
in the finite case. We show that these are the right objects to characterize extensions. 
\begin{theorem}
  If `modular cut' is defined as in Definition~\ref{infmodcut}, 
  then Theorems~\ref{finmodcutdeletion} and Theorems~\ref{finmodcutextension} 
  hold for infinite matroids. 
\end{theorem}

See Theorems~\ref{finmodcutdeletioninf} and~\ref{modcutextension} for the proof of the above. 

\subsection*{Quotients and Projections}

We now give two definitions that are well-known to be equivalent for finite matroids. 

\begin{definition}
A matroid $N$ is a \emph{projection} of a matroid $M$
if there is some matroid $P$ and a set $X$ for which $P \con X = N$ and $P \del X = M$. 
\end{definition}

\begin{definition}
  A matroid $N$ is a \emph{quotient} of a matroid $M$, and we write $N \preceq M$,
  if $E(M) = E(N)$, and $\cl_M(X) \ss \cl_N(X)$ for all $X \ss E(M)$. 
\end{definition}

Since these definitions are equivalent,
there is some disagreement on terminology in the literature on finite matroids,
where authors pick one term or the other, and one definition or the other.
For instance, [\ref{oxley}] uses `quotient' to mean what have defined as a `projection'. 
As well as being equivalent to each other, quotients and projections of finite matroids have
have nice characterizations in terms of circuits, flats and relative rank. 
The following appears in ([\ref{oxley}], Propositions 7.3.1 and 7.3.6).
The statement uses the \emph{relative rank} of two sets $X \ss Y$ in $M$, 
which is defined by $\relrank{M}{Y}{X} = r_{M \con X}(Y-X)$; this is equal to $r_M(Y) - r_M(X)$ in the finite case.

\begin{theorem}\label{finquotequiv}
  Let $M$ and $N$ be finite matroids. The following are equivalent: 
  \begin{enumerate}[(1)]
    \item\label{finquot} $N$ is a quotient of $M$,
    \item\label{finprojdual} $M^*$ is a quotient of $N^*$,
    \item\label{finflat} every flat of $N$ is a flat of $M$,
    \item\label{fincircuit} every circuit of $M$ is a union of circuits of $N$,
    \item\label{finrank} $\relrank{N}{Y}{X} \le \relrank{M}{Y}{X}$ for all $X \ss Y \ss E(M)$. 
    \item\label{finquotproj} $N$ is a projection of $M$,
  \end{enumerate}
\end{theorem}

For infinite matroids, we show that the notions of a quotient and a projection diverge, 
and that the other properties on the above list are all equivalent to the former. 

\begin{theorem}\label{quotprojectequiv}
  For general infinite matroids, conditions (\ref{finquot}) - (\ref{finrank}) of Theorem~\ref{finquotequiv} are equivalent
  to each other, and are implied by (\ref{finquotproj}), but do not imply~(\ref{finquotproj}). 
\end{theorem}

In fact, we show that without additional assumptions, quotients and projections can be quite different. 
It is not difficult to show that, if $N$ is a projection of $M$, and there is a set $B$ that is a basis
for both $N$ and $M$, then $N = M$. We give examples where this fact dramatically fails if $N$ is instead
known to be a quotient of $M$. These examples are well-known constructions coming from infinite graphs; 
see Theorem~\ref{badgraphic}. More abundant and exotic examples are constructed in [\ref{gj25}]. 

On the other hand, we can recover a range of cases where being a quotient does imply being a projection. 
We require an assumption that either $N$ or $M^*$ is finitary (has no infinite circuits), 
as well as an assumption that $N$ and $M$ have `finite rank difference' in a certain sense. 
The proof of this theorem requires the theory of infinite modular cuts just discussed; 
see Theorem~\ref{quotientisproject}.

\begin{theorem}\label{quotientisprojectintro}
  Let $N$ be a quotient of a matroid $M$. If $N$ or $M^*$ is finitary, 
  and $N$ has a basis $B$ for which $M \con B$ has finite rank, then $N$ is a projection of $M$. 
\end{theorem}

We do not know any counterexamples if the finite-rank condition is dropped, 
but the above statement is false without the finitary hypotheses, 
even if it is known that $r(M \con B) = 0$ and that $N^*$ and $M$ are both finitary; 
again, the bad examples come from Theorem~\ref{badgraphic}. 

A \emph{single-element projection} of a matroid $M$ is a matroid $N$ such that
$M = P \del e$ and $N = P \con e$ for some matroid $P$. Our next theorem shows 
that a sequence of single-element projections can be realized as a single $k$-element 
projection. See Theorem~\ref{projectseq} for the proof. 

\begin{theorem}\label{bigprojofseq}
  If $M_0, \dotsc, M_k$ are matroids such that $M_{i+1}$ is a single-element 
  projection of $M_i$ for each $0 \le i < k$, then there is a matroid $P$
  and a $k$-element set $X$ such that $P \del X = M_0$ and $P \con X = M_k$. 
\end{theorem}

A corollary is the following. The matroid $P$ in the following lemma is a \emph{splice}
of $M$ and $N$, and in fact the proof constructs a `free splice'. 
These were defined and studied in the finite case by Bonin and Schmitt [\ref{bs09}]. 

\begin{theorem}\label{splice}
  Let $M$ and $N$ be matroids, and $C$ and $D$ be disjoint sets, not both infinite. If $M \con C = N \del D$,
  then there is a matroid $P$ with $P \del D = M$ and $P \con C = N$. 
\end{theorem}

Our techniques for this result rely on induction, and fail if $C$ and $D$
are both infinite, as do the techniques in [\ref{bs09}] for constructing splices
directly. It seems that constructing infinite splices is related to the problem
of constructing uniform matroids of infinite rank and corank. 
As with [\ref{bg}], it is possible that they require additional axioms
to construct. 

\subsection*{Skew Families}

For finite matroids, a set family $\cX$ is called \emph{skew} (or sometimes \emph{mutually skew}) if 
$r_M\br{\cup \cX} = \sum_{X \in \cX}r_M(X)$. 
Much like modularity, this definition does not work well for infinite matroids, 
but we can use the notion of a modular family to find a definition that does.

\begin{definition}\label{skewdef}
  A collection $\ab{X_a : a \in A}$ of sets \emph{skew} in a matroid $M$ if $\cX$ is modular in $M$ and the intersection of any two sets in $\cX$ has rank zero.  
\end{definition}

We show in Section~\ref{skewsec} that
skewness defined this way agrees with the usual definition in the finite case,
and more generally has the familiar properties we expect. We summarize a few such properties here, 
with a simplified statement of Theorem~\ref{skewequiv}.

\begin{theorem}\label{skewequivintro}
  Let $\cX$ be a pairwise disjoint set family in a matroid $M$. 
  The following are equivalent: 
  \begin{enumerate}[(1)]
    \item $\cX$ is skew in $M$, 
    \item $M | \cup \cX = \oplus_{X \in \cX} (M | X)$, 
    \item $(M \con \cup \cY) | \cup \cZ = M | \cup \cZ$ for each partition $(\cY, \cZ)$ of $\cX$, 
    \item each circuit of $M | \cup \cX$ is contained in some $X \in \cX$. 
  \end{enumerate}
\end{theorem}

\subsection*{Nullity}

For each set $X$ in a finite matroid $M$, the `nullity' of $M$ in $M$, 
namely the difference $|X| - r_M(X) = r^*(M | X)$, is always 
nonnegative. Since it is so easily expressible as a function of $|X|$ and $r_M(X)$,
there is generally no need to give a specific name to this quantity, 
so the term is used quite rarely. 

For infinite matroids, the quantity $r^*(M|X)$ will be both meaningful 
and useful to consider, though it is no longer definable as a subtraction. 

\begin{definition}
  For each set $X$ in a matroid $M$, the \emph{nullity} of $X$ in $M$ is defined by $n_M(X) = r^*(M|X)$. 
\end{definition}

For each basis $I$ of $X$, we have $n_M(X) = |X - I|$, 
and so $r_M(X) + n_M(X) = |X|$ for all $X$;
therefore $n_M(X) = |X| - r_M(X)$ in the finite case, while 
in the infinite case, if $r_M(X) = |X| = \infty$, the nullity can take any value, 
and hence the nullity of a set is not determined by its rank and cardinality. 
(This is seen explicitly by considering the nullity of the ground set of an infinite, 
finite-corank uniform matroid.) 

While the nullity function $n_M$ may appear
to play a symmetrical role to that of the rank function $r_M$, 
it contains strictly more information in the sense that it determines the matroid;
the independent sets are the sets of nullity zero.
The current version of the rank axioms for infinite matroids uses the relative 
rank, which requires assigning a number to every \emph{pair} of sets. 
Could it be that there is a natural characterization of the functions $n : 2^E \to \enn$
that are the nullity function of a matroid on $E$? 

\begin{problem}
  Find a natural axiomatic description of infinite matroids in terms of the nullity function. 
\end{problem}

In any case,
nullity will be quite useful when reasoning about connectivity in infinite matroids, 
and Section~\ref{nullitysec} establishes a number of such properties. 
Nearly all of these are properties are easy for finite matroids via 
rank calculations, but they require separate proofs in the infinite case. 
One such property is supermodularity: we have $n_M(X) + n_M(Y) \le n_M(X \cup Y) + n_M(X \cap Y)$, which is shown in Theorem~\ref{nullityprop}.

\subsection*{Connectivity}

The \emph{connectivity} of a set $X$ in a finite matroid $M$ is defined by $\lambda_M(X) = r_M(X) + r(M \del X) - r(M)$. 
This definition is meaningless for infinite matroids, but Bruhn and Wollan [\ref{bw12}] provide 
a good alternative, as follows.

\begin{definition}\label{bwdef}
  The \emph{connectivity} of a set $X$ in a matroid $M$ is the
  number in $\enn$ given by 
  $\lambda_M(X) = r^*(M | (B \cup B'))$, 
  where $B$ and $B'$ are arbitrarily chosen bases for $X$ and $E(M) - X$ respectively. 
\end{definition}
We have stated this slightly differently from their definition, which phrases the co-rank evaluation instead 
as a certain minimum that is readily seen to be equivalent. 
(Less importantly, they use the  symbol $\kappa$ rather than $\lambda$, 
but we will stick with $\lambda$ in our discussion to align with the usual convention in the finite case.)

It is not clear that this value does not depend on the choice of $B$ and $B'$, 
but they show that it does not ([\ref{bw12}], Lemma 14(iii)), and hence that the definition makes sense. 
Although their definition is a little technical to work with, they show that it behaves well; 
it agrees with the finite version of $\lambda$ in the finite case, is symmetric, self-dual and submodular for all matroids, 
and satisfies a natural version of Tutte's Linking theorem in the finitary and cofinitary cases. 
They also show that $\lambda_M(X) = 0$ if and only if $(X, E(M)-X)$ is a separation of $M$. 

We show that the definition of $\lambda$ can be simplified, only requiring us to choose a basis for $X$, 
but not its complement. Again, we use relative rank notation. 

\begin{theorem}\label{shortconndef}
  If $X$ is a set in a matroid $M$, then $\lambda_M(X) = \relrank{M^*}{X}{X-I}$ 
  for every basis $I$ of $X$.
\end{theorem}

We use this fact to give a new, short proof that the parameter $\lambda$ is self-dual; 
see Theorem~\ref{connselfdual}. 

Instead of writing $\lambda_M(X)$ for some $X \ss E(M)$, it is not uncommon to instead write $\lambda_M(X,Y)$
for a partition $X,Y$ of $E(M)$ to emphasise the symmetry of $\lambda$. 
We define a generalization of $\lambda$ to arbitrary partitions of the ground set. 
\begin{definition}\label{gencon}
  For an indexed partition $\cX = \ab{X_a : a \in A}$ of $E(M)$, 
  define the \emph{connectivity} $\lambda_M(\cX) = r^*(M | \cup_a I_a)$, where $I_a$ is a basis for $X_a$ for each $a \in A$. 
\end{definition}

We defer the proof that this value is independent of the choice of bases 
until later (see Corollary~\ref{lcwelldef}). 
For now, it is clear that this generalizes Definition~\ref{bwdef}, 
and one can prove a number of reasonable properties. 
For instance, it will be easy to prove that $\lambda_{M}(\cX) = 0$ 
if and only if $M$ is the direct sum of
its restrictions to the sets in $\cX$, 
and it is immediate from the above definition that 
$\lambda_M(\cX) = \sum_{X \in \cX} r_M(X) - r(M)$ for finite matroids.
The parameter is not self-dual for partitions of size more than two
(for instance, consider the partition of the matroid $M(K_{2,3})$ into three two-edge induced paths), 
but its dual has a geometric interpretation, which we discuss soon.

\subsection*{Local Connectivity}

A closely related parameter to $\lambda$ is the \emph{local connectivity} between two sets $X$ and $Y$,
defined for finite matroids by 
\[\sqcap_M(X,Y) = r_M(X) + r_M(Y) - r_M(X \cup Y).\] 
This parameter interacts nicely with skewness and modularity; 
indeed, a pair $(X,Y)$ is skew if and only if $\sqcap_M(X,Y) = 0$, 
and is modular if and only if $\sqcap_M(X,Y) = r_M(X \cap Y)$. 
Note that if $X$ and $Y$ are disjoint, then $\sqcap_M(X,Y) = \lambda_{M | (X \cup Y)}(X)$, 
and if $X$ and $Y$ are not disjoint, then $\sqcap_M(X,Y) = \sqcap_{\hat{M}}(X,\hat{Y})$, 
where $\hat{M}$ is obtained from $M$ by adding a parallel disjoint copy $\hat{Y}$ of $Y$. 
Thus, in principle, $\sqcap$ can be defined in terms of $\lambda$. 
We do something like this later (see Definition~\ref{mlcdef}), 
but for now we define and discuss the two-set version using an alternative method. 

\begin{definition}\label{lcdef}
  The \emph{local connectivity} between sets $X,Y$ in a matroid $M$ is the number in $\enn$ given by 
  $\sqcap_M(X,Y) = |I \cap J| + r^* (M | (I \cup J))$, 
  where $I$ and $J$ are arbitrary bases for $X$ and $Y$ respectively.
\end{definition}

To show that this is well-defined, 
we prove that this value is independent of the choices of $I$ and $J$; 
see Theorem~\ref{lcwelldef}.
It is immediate from the definitions that $\lambda_M(X) = \sqcap_M(X, E(M) - X)$, 
so this interacts correctly with Definition~\ref{bwdef}. We prove in Theorem~\ref{lcprop}
that this definition has most of the good properties one would expect, and 
that there is a version of the definition that requires picking a basis for only one of the 
sets. We also show that local connectivity interacts correctly with modularity and skewness;
see Theorem~\ref{lcmod}. 

\subsection*{Guts Extensions}

Despite our simplified expression for $\lambda_M(X)$ in Lemma~\ref{shortconndef}, 
it is still perhaps a little unsatisfying 
that it still requires a choice of basis and an invariance proof, and that the generalized 
Definition~\ref{gencon} requires choosing a basis for every set in the family. 
Using our machinery of extensions of infinite matroids, 
we provide an alternative geometric interpretation of $\lambda$ that needs no choice of basis, 
and lends rigour to the intuition that $\lambda$ is measuring the `interaction' between the cells 
of a partition of the matroid. 

Geelen, Gerards and Whittle showed ([\ref{ggw06}], Theorem 6.1) that, 
given a partition $(X,Y)$ of the ground set of a finite matroid $M$, 
there is a unique extension $M'$ of $M$ by an element $e$ in which $e$ is spanned by each set $Z \ss E(M)$ 
if and only if $\lambda_{M \con Z}(X-Z) = 0$. 
Here, $e$ can be thought of as being added 
into the common closure of $X$ and $Y$ as freely as possible; 
in $M'$, the element $e$ lies in the `guts' of the separation $(X \cup \{e\}, Y)$. 
It stands to reason (and is easy to prove), 
that $e$ is a loop of $M'$ if and only if $\lambda_M(X) = 0$, 
and that if $\lambda_M(X) \ne 0$ then $\lambda_{M' \con e}(X) = \lambda_M(X) - 1$. 
It follows that $\lambda_M(X)$ is equal to the minimum $k$ for which iterating such an extension-contraction
operation $k$ times gives a matroid $N$ for which $\lambda_N(X) = 0$. 

We prove that this theory carries through to infinite matroids and arbitrary partitions. 
As far as we know, the generalization to arbitrary partitions is new even for finite matroids. 

\begin{theorem}\label{gutscutintro}
  Let $\ab{X_a : a \in A}$ be a partition of the ground set of a matroid $M$, and let $e \notin E(M)$. 
  There is a unique extension $M'$ of $M$ by $e$ such that each flat $F$ of $M$
  satisfies $e \in \cl_{M'}(F)$ if and only if $\ab{X_a - F : a \in A}$ is skew in $M \con F$. 
\end{theorem}

Call the matroid $M'$ the \emph{guts extension} of $M$ by $e$, and 
write $\gutsproj{M}{\cX}{}$ for the matroid $M' \con e$, 
which we refer to as the \emph{guts projection} of $\cX$ in $M$.
Note that this is also a `projection' in the sense discussed earlier. 
For $k \in \bN$, write $\gutsproj{M}{\cX}{k}$ for the matroid $M_0$ obtained from $M$ by iterating 
the operation $\gutsproj{\cdot}{\cX}{}$ a total of $k$ times. 

Our next main result, which can be seen as either a geometric interpretation or 
an alternative definition of $\lambda$, 
shows that $\lambda_M^*(\cX)$ (defined as $\lambda_{M^*}(\cX)$) 
is equal to the minimum number of guts projections required
to reduce the connectivity of $\cX$ to zero. 

\begin{theorem}\label{lambdadualeq}
  If $\cX$ is a partition of the ground set of a matroid $M$, then $\lambda_M^*(\cX)$ is equal to the minimum $k \in \bN$ 
  such that the matroid $M_0 = \gutsproj{M}{\cX}{k}$ satisfies $\lambda_{M_0}(\cX) = 0$, 
  or is $\infty$ if no such $k$ exists. 
\end{theorem}

If $|\cX| = 2$, then $\lambda_M(\cX) = \lambda_M^*(\cX)$, and the above gives the same interpretation of $\lambda_M$
discussed earlier. If $|\cX| > 2$, then the above is a description of $\lambda_M^*$ but not $\lambda_M$. 
One can dualize the statement to see that $\lambda_M(\cX)$ is the minimum number of `coguts lifts' needed 
to reduce the connectivity of $\cX$ to zero, where a coguts lift is the dual operation of a guts projection. 

In fact, we can view the minimum in Theorem~\ref{lambdadualeq} differently 
with the following statement, which is as simple in the dual version as in the primal. 
Roughly, it states that $\lambda^*_M(\cX)$ is the minimum size of a set whose
projection renders $\cX$ skew. 

\begin{theorem}\label{lambdaminproj}
  Let $\cX$ be a partition of the ground set of a matroid $M$. Then 
  \begin{enumerate}[(i)]
    \item $\lambda_M(\cX)$ is the minimum size of a set $K$ such that 
      there is a matroid $P$ with $P \con K = M$ and $\cX$ skew in $P \del K$; 
    \item $\lambda^*_M(\cX)$ is the minimum size of a set $K$ such that 
      there is a matroid $P$ with $P \del K = M$ and $\cX$ skew in $P \con K$.  
  \end{enumerate}
\end{theorem}

\subsection*{Modular Flats}

Returning to modularity, 
we prove that the modularity of a flat $F$ in a matroid has many equivalent characterizations. 
The first are in terms of skew pairs, the next is stated in terms of 
contracting points onto $F$ in a minor of $M$, and the last two have a lattice-theoretic flavour. 
(Two flats are `complementary' if they have trivial intersection and spanning union.)

\begin{theorem}
  Let $F$ be a flat of a matroid $M$. The following are equivalent: 
  \begin{enumerate}[(1)]
    \item $F$ is a modular flat of $M$;
    \item $F-G$ and $G-F$ are skew in $M \con (F \cap G)$ for every flat $G$ of $M$;
    \item every complementary flat of $F$ in $M$ is skew to $F$;
    \item
    for every $C \subseteq E(M)$ and every nonloop $e$ of $M \con C$ with $e \in \cl_{M \con C}(F-C)$;
    there exists $f \in F$ that is parallel to $e$ in $M \con C$;
    \item 
    for every pair of flats $F_0, G$ of $M$ with $F_0 \ss F$, we have $F \cap \cl_M(G \cup F_0) = \cl_M((F \cap G) \cup F_0)$;
    \item 
    for every pair $G_1,G_2$
    of flats of $M$ with $G_1 \ss G_2$, we have $\cl_M(G_1 \cup F) \cap G_2 = \cl_M(G_1 \cup (F \cap G_2))$. 
  \end{enumerate}
\end{theorem}

\subsection*{Modular Matroids}

Finite matroids where every flat is modular are very rare. 
Aside from trivialities coming from direct sums and simplification, 
and non-Desarguesian projective planes in rank three, all examples are projective geometries over finite fields. 
The following statement combines ([\ref{oxley}], Proposition 6.9.1 and Theorem 6.1.2). 

\begin{theorem}\label{modularpg}
  Every simple, connected, finite modular matroid of rank at least $4$ is isomorphic to $\PG(n-1,\GF(q))$ for some prime power $q$. 
\end{theorem}

Modularity is therefore the property that characterizes the `complete' objects 
for the class of matroids representable over finite fields.

For infinite matroids, more modular matroids exist. 
As well as replacing the field $\GF(q)$ with an infinite field in the above example,
one can define a matroid $PG(V,\bK)$ for an infinite-dimensional module $V$ over a skew field $\bK$ by taking 
$E(M)$ to comprise an arbitrary transversal (not containing zero) of the one-dimensional subspaces of $V$, 
and taking a set $I$ to be independent in $M$ if it is linearly independent over $\bK$, 
by which we mean that no nontrivial finite linear combination of the elements of $I$ is zero, 
with respect to right-multiplication in $\bK$.  

(In considering both infinite modules and skew fields, 
this construction combines two slightly exotic notions of matroid representations, 
but their combination is straightforward: 
See [\ref{hk96}] for a treatment of matroid representations over skew fields 
and [\ref{bc18}] for a discussion of finitary representations of infinite matroids.)

The matroid $\PG(V,\bK)$ just described has only finite circuits, 
as will any matroid using the conventional notion of linear independence for infinite sets. 
However, there are more permissive notions of infinite matroid representations 
that allow for infinite circuits, notably `thin sums representability' [\ref{bc18}].
It is natural to ask whether there might be a class $\cM$ of modular matroids, 
playing a role similar to projective geometries, such that 
every simple thin-sums-representable matroid is a restriction of a matroid in $\cM$. 
We prove that the answer is no, since modular matroids are all finitary. 

\begin{theorem}\label{modularfinitary}
Every modular matroid is finitary. 
\end{theorem}

This possibly puts within reach the problem of characterizing all modular matroids, finite or infinite. 
We remark that that there are interesting finite-rank examples that do not occur
in the finite case, such as the projective plane over the octonions. 

\begin{problem}
  Characterize all modular matroids. 
\end{problem}

Finally, we prove a necessary and sufficient condition for modularity
that only considers lines and hyperplanes. This is well-known for finite matroids
(see [\ref{oxley}], Proposition 6.9.1), but seems to be new for infinite 
matroids.

\begin{theorem}\label{modularlinehyperplane}
  Let $M$ be a loopless matroid. Then $M$ is modular if and only if every line of $M$
  intersects every hyperplane of $M$. 
\end{theorem}

\section{Preliminaries}

\subsection*{Set Families}
So far, we have mostly stated our result about set families in $M$ in terms of 
subsets of the power set of $E(M)$. 
When working with disjointness and skewness, 
this perspective can give rise to bothersome and distracting technicalities; 
for example, the collection $\cX = \{\{1,2\},\{2,3\}\}$ is not pairwise disjoint,
but the (one-element) collection $\{X \cup \{1,3\} : X \in \cX\}$ is pairwise disjoint. 

For this reason, we prefer working with indexed collections rather than 
sets of sets. We write $\ab{X_a : a \in A}$ for an indexed collection of 
subsets of some ambient set $S$, usually the ground set of a matroid
(formally, this is just a function from $A$ to 
the power set of $S$). We seldom vary our indexing set; usually after 
mentioning it once, we write $\ab{X_a}$ instead of $\ab{X_a : a \in A}$ 
to be concise. 

In certain proofs, to limit unwieldy notation, we allow sets as subscripts 
with indexed collections to denote unions. That is, for an indexed 
collection $\ab{X_a : a \in A}$ and a set $B \ss A$, we write $X_B$ for $\cup_{b \in B}X_b$. 

\subsection*{Infinite Matroids}

The definition we use for infinite matroids was given by Bruhn et al. [\ref{bdkpw}],
and earlier (in the slightly different form of a `$B$-matroid') by Higgs [\ref{higgs}].

\begin{definition}
  A \emph{matroid} is a pair $M = (E, \cI)$, where $E$ is a set and $\cI$ is a
  family of subsets of $E$ for which
  \begin{enumerate}[(1)]
    \item\label{I1} $\varnothing \in \cI$,
    \item\label{I2} for all $I \subseteq J$, if $J \in \cI$ then $I \in \cI$, 
    \item\label{I3} for all $I, B \in \cI$ such that $B$ is a maximal member of
      $\cI$ and $I$ is not, there exists $e \in B - I$ for which $I \cup \{e\} \in \cI$, and
    \item\label{IM} for all $I \subseteq X \subseteq E$ with $I \in \cI$, the collection
      $\setof{J \in \cI}{I \subseteq J \subseteq X}$ has a maximal element.
  \end{enumerate}
\end{definition}

It was shown in [\ref{bdkpw}] that infinite matroids are well-behaved: 
they specialize to finite matroids in the case where $E$ is finite, 
have good notions of minors and duality, and admit equivalent axiomatizations 
in terms of bases, circuits, the closure function and the relative rank function. 
We mostly use the basic material and terminology from [\ref{bdkpw}] without explicit reference
to named lemmas. 

We also use some basic notions related to connectivity of infinite matroids from [\ref{bw12}], 
such as the parameter $\lambda$ of Definition~\ref{bw12}, and the notion of the 
direct sum $\oplus_{i}M_i$ of a family $(M_i : i \in I)$ of matroids. 
We now review some concepts of particular relevance to what lies ahead.  

\subsection*{Rank and cardinality}

We write $\bN = \{0, 1, \dotsc\}$ for the set of nonnegative integers, 
and $\enn$ for the set $\bN \cup \{\infty\}$. This set has a well-defined linear 
order and notion of addition (with $\infty + x = \infty$ for all $x$), 
but the subtraction $\infty - \infty$ is undefined, 
and consequently we cannot conclude from an equality $a + x = a + y$ that $x = y$, 
unless we know that $a < \infty$. 
This issue will be particularly relevant for us, 
since it is often what stops finite definitions from generalizing correctly, 
such as with the parameter $\lambda$. 

If $X$ is any set, we will write $|X|$ for the number of elements of $X$, 
viewed as an element of $\enn$, and for each infinite matroid we will view the rank 
$r(M)$, or the number of elements of a basis of $M$, as an element of $\enn$. 
(It is routine to show that any two bases of a matroid 
have the same number of elements, when the number is viewed as an element of $\enn$.)
In the context of set theory, it may seem rather coarse not to distinguish 
infinite sets of different cardinalities. But this seems to be the largest codomain 
in which the `rank' of an infinite matroid is well-defined without additional set-theoretic assumptions. 
 While Higgs [\ref{higgs69}] showed that any two bases of a matroid
have the same cardinality subject to the generalized continuum hypothesis, 
Bowler and Geschke [\ref{bg}] showed that it is consistent with ZFC that 
there is a matroid having both a countable and an uncountable basis. 
So (without additional exotic axioms) the rank of a matroid as a cardinal is not 
even well-defined. 

The fact that the $\enn$-rank $r(M)$ of a matroid $M$ is well-defined allows us to 
meaningfully measure the rank $r_M(X) = r(M | X)$ of each set $X$ in a matroid $M$, 
the rank $r^*(M)$ of the dual matroid $M^*$, and the rank $r_{M \con X}(Y)$ 
of a set $Y$ in the contraction $M \con X$. This last fact is useful, since 
it allows for a more fine-grained version of the rank function. 
For sets $X,Y \ss E(M)$, the \emph{relative rank} of $Y$ to $X$ in $M$ 
is defined by $\relrank{M}{Y}{X} = r_{M \con X}(Y-X)$. 
(In [\ref{bdkpw}], the relative rank is only defined when $X \ss Y$;
this makes sense from the perspective of a minimal axiomatization, 
but otherwise there is no particular reason not to define it more generally.)
If $M$ is a finite-rank matroid,
this expression can be defined using subtraction as $\relrank{M}{Y}{X} = r_M(X \cup Y) - r_M(X)$. 
A matroid is completely determined by its ground set and relative 
rank function [\ref{bdkpw}], but not just its rank function. 
For instance, if $M_0$ is a free matroid on $\bN$ and $M_1$ is 
a circuit on $\bN$, then $M_0 \ne M_1$, but $r_{M_0}(X) = r_{M_1}(X)$ for all $X$. 

\subsection*{Projecting elements}
When contracting a set $X$ in a matroid $M$, the elements of $X$ are removed from $M$, 
but sometimes, to avoid set-theoretic bookkeeping, it is more convenient to keep them.
The following definition appears in [\ref{gn}].
\begin{definition}
  For a set $X$ of elements of a matroid $M$, the \emph{projection of $M$ by $X$} is the matroid $M \dcon X = M \con X \oplus O_X$,
where $O_X$ is the rank-zero matroid with ground set $X$.   
\end{definition}

Equivalently, the projection is the matroid with ground set $E(M)$ 
and the same independent sets as $M \con X$, and is also 
the matroid obtained from $M$ by extending each 
element of $X$ in parallel, and contracting all the new elements. 
Hence, the projection of $M$ by $X$ is also a `projection' of $M$ in the sense defined earlier. 
The following facts, which we use freely, are all straightforward to prove. 
\begin{lemma}
  If $X$ and $Y$ are sets in a matroid $M$, then 
  \begin{enumerate}[(i)]
    \item $E(M \dcon X) = E(M)$;
    \item $\cI(M \dcon X) = \cI(M \con X)$;
    \item $M \dcon X \del X = M \con X$;
    \item $(M \dcon X) \dcon Y = M \dcon (X \cup Y)$;
    \item $\cl_{M \dcon X}(Y) = \cl_M(X \cup Y)$;
    \item if $\cl_M(X) = \cl_M(Y)$, then $M \dcon X = M \dcon Y$;
    \item if $Y$ is independent in $M$, then $X$ is independent 
      in $M \dcon Y$ if and only if $X$ and $Y$ are disjoint and $X \cup Y$ is independent in $M$. 
  \end{enumerate}

\end{lemma}

\section{Modularity}

In this section, we prove some facts about modularity and mutual bases. 
We start with a useful lemma about closure. 
 
\begin{lemma}\label{closureinter}
  Let $\cI$ be a collection of sets in a matroid $M$. If $\bigcup \cI$ is independent in $M$,
  then $\cl_M(\bigcap \cI) = \bigcap_{I \in \cI} \cl_M(I)$. 
\end{lemma}
\begin{proof}
  Since $\cl_M(\bigcap \cI) \ss \cl_M(I)$ for each $I$, we have $\cl_M(\bigcap \cI) \ss \bigcap_{I \in \cI} \cl_M(I)$. 
  Conversely, suppose that $e \in \cl_M(I)$ for all $I \in \cI$. Let $I_0 = \bigcap \cI$ and $I_1 = \bigcup \cI$, 
  and suppose that $e \notin \cl_M(I_0)$, so $I_0 \cup \{e\}$ is independent. 
  
  For each $I \in \cI$, let $J_I$ be a basis for $I \cup \{e\}$ containing $I_0 \cup \{e\}$. Let $J = \bigcup_I J_I$. 
  Since $\cl_M(J_I) = \cl_M(I \cup \{e\}) = \cl_M(I)$ for each $I \in \cI$, we have $\cl_M\br{I_1 \cup \{e\}} = \cl_M(J)$.
  Let $J_1$ be a basis for $I_1 \cup \{e\}$ for which $I_0 \cup \{e\} \ss J_1 \ss J$. 
  
  By construction, both $J_1$ and $I_1$ are bases for $I_1 \cup \{e\}$, so there exists $f \in I_1$ 
  such that $J_1 = I_1 \cup \{e\} - \{f\}$.
  Since $I_0 \ss J_1$, we have $f \notin I_0$, so 
  there is some $I \in \cI$ for which $f \notin I$. Therefore $I \ss I_1 \cup \{e\} - \{f\} = J_1$, 
  so the independent set $J_0$ contains the set $I$, as well as the element $e \in \cl_M(I)$, a contradiction. 
\end{proof}

We now restate and prove Lemma~\ref{modpairiffbasisintro}, which shows that our definition of a modular pair 
aligns with the standard definition for finite matroids. 

\begin{lemma}\label{modpairiffbasis}
  If $X$ and $Y$ are sets in a finite matroid $M$, then $r_M(X) + r_M(Y) = r_M(X \cup Y) + r_M(X \cap Y)$
  if and only if $\{X,Y\}$ has a mutual basis. 
\end{lemma}
\begin{proof}
  If a mutual basis $B$ exists for $\{X,Y\}$, then
  \begin{align*}
    r_M(X \cap Y) &\ge |B \cap X \cap Y| \\ 
    &=  |B \cap X| + |B \cap Y| - |(B \cap X) \cup (B \cap Y)| \\ 
    &= r_M(X) + r_M(Y) - r_M(X \cup Y),
  \end{align*}
  and equality holds by submodularity. 

  Conversely, suppose that $r_M(X) + r_M(Y) = r_M(X \cup Y) + r_M(X \cap Y)$, and let $I$ be a basis for $X \cap Y$. 
  and let $I_X$ and $I_Y$ be bases for $X$ and $Y$ respectively containing $I$, noting that $I_X \cap I_Y = I$
  by the maximality of $I$. Now 
  \begin{align*}
    r_M(I_X \cup I_Y) &= r_M(X \cup Y) \\ 
    &= r_M(X) + r_M(Y) - r_M(X \cap Y) \\
    &= |I_X| + |I_Y| - |I| \\ 
    &= |I_X \cup I_Y|,
  \end{align*}
  so $I_X \cup I_Y$ is independent. Let $B$ be a basis for $M$ containing $I_X \cup I_Y$. 
  Since $I_X \ss B \cap X$ and $I_Y \ss B \cap Y$, this is the required mutual basis $B$. 
\end{proof}

To handle the concept of a mutual basis more smoothly, we introduce some notation.  
For each set $I$ in a matroid $M$, let $\cL_M(I) = \set{X \ss E(M) : X \ss \cl_M(X \cap I)}$. 
The following shows that $\cL_M(I)$ is a complete sublattice of the boolean lattice on $E(M)$.
(In the statement, the empty intersection is taken to have value $E(M)$.)

\begin{lemma}\label{blattice}
  Let $I$ be an independent set of $M$. Then $\cL_M(I)$ is closed under arbitrary unions and intersections, 
  and every $X \in \cL_M(I)$ satisfies $\cl_M(X) \in \cL_M(I)$. 
\end{lemma}
\begin{proof}
  Let $\cX \ss \cL_M(I)$. Then 
  \[ \cup \cX \ss \bigcup_{X \in \cX} \cl_M(X \cap I) \ss \cl_M\br{\bigcup_{X \in \cX} (X \cap I)} = \cl_M\br{\br{\cup\cX} \cap I },\]
  so $\bigcup \cX \in \cL_M(I)$. 
  Similarly, applying Lemma~\ref{closureinter} with $\cI = \{X \cap I : X \in \cX\}$ gives
  \[ \cap \cX \ss \bigcap_{X \in \cX} \cl_M(X \cap I) = \cl_M\br{\bigcap_{X \in \cX} (X \cap I)} = \cl_M\br{(\cap \cX) \cap I},\]
  so $\bigcap \cX \in \cL_M(I)$. (Note that this is still valid for $\cX$ empty, if empty intersections are taken to have value $E(M)$.)

  Finally, it is clear that if $X \in \cL_M(I)$, then $\cl_M(X) \ss \cl_M(\cl_M(X \cap I)) \ss \cl_M(\cl_M(X) \cap I)$, 
  so $\cl_M(X) \in \cL_M(I)$. 
\end{proof}

Our next lemma shows that $\cL_M(I)$ is the largest collection of sets for which $I$ is a mutual basis.
The proof is trivial, since a basis for a set $X$ is an independent set that spans $X$. 

\begin{lemma}\label{mutbasisiffsubset}
  If $I$ is an independent set of a matroid $M$, then $X \in \cL_M(I)$ if and only if $X \cap I$ is a basis for $X$ in $M$. 
  In particular, $I$ is a mutual basis for a set family $\cX$ if and only if $\cX \ss \cL_M(I)$. 
\end{lemma}

Although we do not need this, 
we remark that [\ref{bdkpw}] essentially uses the set $\cL_M(B)$ in the process of proving the dual matroid is a matroid, 
in fact proving that for a basis $B$ of $M$, the sets in $\cL_{M^*}(E-B)$ are the complements of the sets in $\cL_M(B)$. 
The following is a direct consequence of Lemma~\ref{mutbasisiffsubset} and ([\ref{bdkpw}], Lemma 3.2).

\begin{lemma}\label{blatticedual}
  Let $M$ be a matroid on $E$. If $B$ is a basis of $M$ and $X \ss E$, 
  then $X \in \cL_M(B)$ if and only if $E - X \in \cL_{M^*}(E-B)$. 
\end{lemma}



The next lemma allows us to enlarge a modular family by adding a set that is either above or below
all the existing elements of the family. 

\begin{lemma}\label{modularinsert}
  Let $\cX$ be a modular family in a matroid $M$. If $Y$ is a set such that each $X \in \cX$ satisfies $X \ss Y$ or $Y \ss X$, 
  then $\cX \cup \{Y\}$ is a modular family. 
\end{lemma}
\begin{proof}
  Let $B$ be a mutual basis for $\cX$. 
  Let $\cX_0 = \setof{X \in \cX}{X \ss Y}$, and let $\cX_1 = \cX - \cX_0$, noting that $Y \ss X$ for all $Y \in \cX_1$.
  We have $\cup \cX_0 \ss Y \ss \cap \cX_1$. 
  Let $I$ be a basis for $Y$ containing $B \cap (\cup \cX_0)$, and let $I_1$ be a basis for $\cap \cX_1$ containing $I$. 

  Let $B_1 = B \cap (\cap \cX_1)$; we argue that $B' = B - B_1 \cup I_1$ is a mutual basis for $\cX \cup \{Y\}$. 
  By Lemma~\ref{blattice}, we have $\cl_M(B_1) = \cl_M(\cap X_1) = \cl_M(I_1)$, so Lemma~\ref{indepclosureswitch}
  applied with the collections $\{B_1\}$ and $\{I_1\}$ gives that $B'$ is independent in $M$. 

  By construction, the set $B'$ contains the basis $I'$ for $Y$, and for each $X \in \cX_0$, 
  we have 
  \[X \ss \cl_M(B \cap X) \ss \cl_M((B \cap (\cup \cX_0)) \cap X) \ss \cl_M(I \cap X) \ss \cl_M(B' \cap X),\]
  so $B'$ contains bases for each set in $\cX_0$. For each $X \in \cX_1$, we have
  \begin{align*}
      X &\ss \cl_M(B \cap X) \\
      &= \cl_M(\br{B \cap (\cap \cX_1)} \cup \br{B \cap (X - \cap \cX_1)}) \\ 
      &= \cl_M\br{I_1 \cup (B' \cap (X - \cap \cX_1))} \\ 
      &\ss \cl_M(B' \cap X),
  \end{align*}
  so $B'$ contains a basis for $X$, and thus $B'$ is a mutual basis for $\cX \cup \{Y\}$.
\end{proof}

Combining Lemma~\ref{modularinsert} with a standard inductive argument gives the following, which 
we invoke freely to find modular bases for chains of small numbers of sets. 

\begin{corollary}
  Every finite chain of sets in a matroid $M$ is modular. 
\end{corollary}

We remark that the above fails for infinite chains; for instance, if $M$ is 
a circuit with ground set $\bN$, then the chain consisting of all initial subsegments of $\bN$ 
can easily be seen to have no mutual basis in $M$.

\begin{lemma}\label{modularcompl}
  Let $F_0, F, F_1$ be flats of a matroid $M$ with $F_0 \ss F \ss F_1$.
  Then there is a flat $F'$ of $M$ for which $(F,F')$ is a modular pair 
  with $F \cap F' = F_0$ and $\cl_M(F \cup F') = F_1$.
\end{lemma}
\begin{proof}
  Let $B$ be a mutual basis for the chain $\{F_0,F,F_1\}$, and let $F' = \cl_M(B \cap (F_0 \cup (F_1 - F)))$. 
  Clearly $(F,F')$ is a modular pair, and $F_1 = \cl_M(B \cap F_1) \supseteq \cl_M((B \cap F') \cup (B \cap F)) \supseteq \cl_M(F \cup F')$. By Lemma~\ref{closureinter} and the independence of $B$, we have 
  \begin{align*}
      F \cap F' &= \cl_M(B \cap F) \cap \cl_M(B \cap (F_0 \cup (F_1 - F))) \\
      & = \cl_M(B \cap (F \cap (F_0 \cup (F_1 - F)))) \\ 
      &= \cl_M(B \cap F_0) = F_0, 
  \end{align*}
  as required. 
\end{proof}

\section{Skewness}\label{skewsec}

Recall that a collection $\ab{X_a : a \in A}$ is \emph{skew} in a matroid $M$ if 
$\ab{X_a}$ is modular in $M$, and $X_a \cap X_b \ss \cl_M(\es)$ for all 
distinct $a,b \in A$. This section, which proves some basic properties of skew collections, 
contains no surprises; 
the properties are familiar from the world of finite matroids, 
and the proofs are fairly straightforward. 
However, we include them all to make sure the theory 
carries through in the infinite case. 

We first prove a slightly weaker sufficient condition for a set family to be skew. 

\begin{lemma}\label{skewofdjbase}
  If $B$ is a mutual basis in $M$ for a collection $\ab{X_a : a \in A}$ and the sets 
  $\ab{B \cap X_a}$ are pairwise disjoint, then $\ab{X_a}$ is skew in $M$. 
\end{lemma}
\begin{proof}
  If there is a nonloop $e \in X_a \cap X_b$ for some $a,b \in A$ with $a \ne b$, then $e \in \cl_M(X_a \cap B) \cap \cl_M(X_b \cap B)$, 
  and so there are circuits $C_a \ss ((X_a \cap B) \cup \{e\})$ and $C_b \ss ((X_b \cap B) \cup \{e\})$ both containing $e$. 
  Since $e$ is a nonloop and $X_a \cap X_b \cap B = \es$, 
  these circuits are distinct, so by circuit elimination, there is a circuit contained in 
  $(C_a \cup C_b) - \{e\} \ss B$, a contradiction. Therefore $X_a \cap X_b \ss \cl_M(\es)$ for all $a \ne b$, so $\ab{X_a}$ is skew. 
\end{proof}

The proof of the next lemma, which shows that we can ignore loops when reasoning about skewness,
is straightforward from the relevant definitions. 

\begin{lemma}\label{skewdj}
  A collection $\ab{X_a : a \in A}$ is skew in a matroid $M$ if and only the collection
  $\ab{X_a - \cl_M(\es)}$ is pairwise disjoint and modular in $M$. 
\end{lemma}

We now prove that skewness has the expected geometric properties:
skew families correspond to direct sums; contracting some members of a skew family does not affect the other members;
and circuits contained in a disjoint skew family do not cross different sets in the family.
The extra hypothesis is needed to avoid technicalities where loops stop the sets in 
the collection being disjoint; in particular, it holds if the $X_a$ are
known to be disjoint, or if no $X_a$ contains a loop of $M$. 

\begin{theorem}\label{skewequiv}
  Let $\ab{X_a : a \in A}$ be collection of sets in a matroid $M$ so that no loop of $M$ is contained in more than one $X_a$.
  The following are equivalent: 
  \begin{enumerate}[(1)]
    \item\label{seskew} $\ab{X_a}$ is skew in $M$;
    \item\label{sesum} The sets $\ab{X_a}$ are pairwise disjoint, and $M | \cup_{a \in A} X_a = \oplus_{a \in A} (M | X_a)$;
    \item\label{secon} 
      The sets $\ab{X_a}$ are pairwise disjoint, and each pair $(A_1, A_2)$ of 
      disjoint subsets of $A$ satisfies 
      $(M \con \cup_{a \in A_1} X_a) | \cup_{a \in A_2} X_a = M | \cup_{a \in A_2} X_a$;
    \item\label{secct} The sets $\ab{X_a}$ are pairwise disjoint, and each circuit of $M | \cup_{a \in A} X_a$ is contained in some $X_a$;
    \item\label{sesingle} for each $s \in A$, the sets $X_s$ and $(\cup_{a \in A}X_a) - X_s$ are skew in $M$;
    \item\label{sesinglebutone} for all but at most one $s \in A$, 
      the sets $X_s$ and $(\cup_{a \in A}X_a) - X_s$ are skew in $M$;
  \end{enumerate}
\end{theorem}
\begin{proof}
  It is easy to see from the definition of direct sum that (\ref{sesum}) implies (\ref{secon}). 
  \begin{claim} (\ref{seskew}) implies (\ref{sesum}).
  \end{claim}
  \begin{subproof}
    Suppose that (\ref{seskew}) holds, let $B$ be a mutual basis for $\ab{X_a}$, 
    and let $B_a = B \cap X_a$ for each $a \in A$. 
    Since no loop is contained in more than one $X_a$, the $X_a$ are pairwise disjoint. 
  
  It is clear that each basis of $M | \cup_a X_a$ is independent in 
  $\oplus_{a \in A}(M | X_a)$. Conversely, let $I$ be a basis of 
  $\oplus_{a \in A}(M | X_a)$, so each set $I_a = I \cap X_a$ is a basis for $X_a$. 
  If $I$ is dependent in $M$, then there is some $s \in A$ and some $e \in I_s$ such that 
  $e \in \cl_M(I - e)$. Since $\cl_M(I_a) = \cl_M(B_a)$ for each $a$, this gives
  \[e \in \cl_M(I-e) = \cl_M\br{I_{A-\{s\}} \cup (I_s - \{e\})}
  = \cl_M\br{B_{A-s} \cup (I_s - \{e\})},\]
  so $B_{A-s} \cup I_s$ contains a circuit $C$. Since $I_s$ is independent, there
  exists $f \in C \cap B_{A-s}$, 
  and now 
  \[f \in \cl_M(C-\{f\}) \ss \cl_M((B_{A-s} - \{f\}) \cup I_s) \ss 
  \cl_M\br{(B_{A-s} - \{f\}) \cup B_s},\]
  and so $f \in \cl_M(B-\{f\})$, contradicting the independence of $B$. 
  \end{subproof}
  \begin{claim}(\ref{secon}) implies (\ref{secct}). \end{claim}
  \begin{subproof}
    Suppose that (\ref{secon}) holds, and let $C$ be a circuit of $M | \bigcup_a X_a$
    that intersects $X_a$ and $X_b$ for some $a \ne b$. Let $I = C - X_a$ and $e \in I$, 
    noting that $I$ is independent in $M$. 
    Then (\ref{secon}) with $(A_1,A_2) = (\{a\},A - \{a\})$ gives that 
    $(M \con X_a) | I = M | I$, so
    \[e \in \cl_M(C-e) = \cl_{M \con X_a}(C- X_a - \{e\}) = \cl_{M \con X_a}(I-\{e\}) = \cl_M(I-\{e\}).\]
    Therefore the proper subset $I$ of $C$ is dependent in $M$, a contradiction.
  \end{subproof}
  \begin{claim}(\ref{secct}) implies (\ref{seskew}).
  \end{claim}
  \begin{subproof}
    Suppose that (\ref{secct}) holds, and let $B$ be a basis for $M | \cup_a X_a$.
    Let $s \in A$ and $e \in X_s - B$. Let $C$ be a circuit of $M$ with $e \in C \ss B \cup \{e\}$.  Since $e \in C \cap X_s$, we must have $C \ss X_s$ by assumption, 
    so $e \in \cl_M(C - \{e\}) \ss \cl_M(B \cap X_s)$. 
    It follows that $X_s - B \ss \cl_M(B \cap X_s)$ for all $s \in A$, 
    whence $B$ is a mutual basis for $\ab{X_a}$ in $M$. 
    Since the $X_a$ are pairwise disjoint, we see that (\ref{seskew}) holds. 
  \end{subproof}
  
  The above claims imply that the first four conditions are equivalent.
  Using the equivalence of (\ref{seskew}) and (\ref{secct}), 
  it is easy to see that each of (\ref{sesingle}) and (\ref{sesinglebutone})
  is equivalent to (\ref{secct}), and the theorem follows. 
\end{proof}

The equivalence of (\ref{seskew}) and (\ref{secct}) above gives a simpler characterization
of skewness in collections of independent sets. 

\begin{corollary}\label{skewindep}
  If $\ab{I_a : a \in A}$ is a family of independent sets in a matroid $M$, 
  then $\ab{I_a}$ is skew if and only if the $I_a$ are disjoint with independent union. 
\end{corollary}

The next lemma gives some properties of arbitrary skew families without assumption about loops. 

\begin{lemma}\label{skewgoodbase}
  Let $\ab{X_a : a \in A}$ be a skew collection of sets in a matroid $M$. Then 
  \begin{enumerate}[(i)]
    \item every basis $B$ for $M | \cup_a X_a$ is a mutual basis for $\ab{X_a}$ in $M$, 
    \item for every collection $\ab{I_a}$ of sets such that each $I_a$ is a basis for $X_a$, 
      the sets $\ab{I_a}$ are pairwise disjoint, and $\cup_{a \in A} I_a$ 
      is a mutual basis for $\ab{X_a}$.
    \item if the sets $\ab{X_a}$ are pairwise disjoint, then $M | \cup_a X_a = \oplus_a (M | X_a)$. 
    \item $(M \dcon \cup_{a \in A_1} X_a) | \cup_{a \in A_2} X_a = M | \cup_{a \in A_2} X_a$ 
      for every pair $(A_1, A_2)$ of disjoint subsets of $A$.
  \end{enumerate}
\end{lemma}
\begin{proof}
  Let $X_a' = X_a - \cl_M(\es)$ for each $a \in A$. 
  By Lemma~\ref{skewdj} and Theorem~\ref{skewequiv}, we have 
  $M | (\cup_a X'_a) = \oplus_{a} (M | X'_a)$, and now all parts follow easily
  from the fact that independent sets contain no loops.
\end{proof}

We now show that skewness is preserved under taking closures. 

\begin{lemma}\label{clskew}
  If $\ab{X_a : a \in A}$ is a skew collection of sets in a matroid $M$, 
  then $\ab{\cl_M(X_a)}$ is skew in $M$. 
\end{lemma}
\begin{proof}
  Let $B$ be a mutual basis for the $X_a$. 
  We have $\cl_M(X_a) = \cl_M(B \cap X_a) \ss \cl_M(B \cap \cl_M(X_a))$, 
  so $B$ is a mutual basis for $(\cl_M(X_a) : a \in A)$. 
  Since $B \cap X_a$ is maximally independent in $\cl_M(X_a)$ for each $a$,
  we have $B \cap \cl_M(X_a) = B \cap X_a$ for each $a$, 
  so the sets $\ab{B \cap \cl_M(X_a)}$ are pairwise disjoint, 
  which implies the result by Lemma~\ref{skewofdjbase}. 
\end{proof}

The next lemma shows that skewness is monotone in several senses: 
taking subsets, coarsening, deletion, and contracting/projecting subsets of the union. 

\begin{lemma}\label{skewmono}
  Let $\ab{X_a : a \in A}$ be a skew collection of sets in a matroid $M$. 
  Then 
  \begin{enumerate}[(i)]
    \item\label{skewsubsets} for each collection $\ab{Y_a}$ with $Y_a \ss X_a$ for all 
    $a$, the sets $\ab{Y_a}$ are skew in $M$,
    \item\label{skewrefine} for every partition $\cB$ of $A$, the sets 
      $\ab{\cup_{b \in B} X_b : B \in \cB}$ are skew in $M$. 
    \item\label{skewdelete} for each set $D \ss E(M)$, 
      the sets $\ab{X_a - D}$ are skew in $M \del D$, 
    \item\label{skewproject} for each $C \ss \cup_a X_a$, 
      the sets $\ab{X_a}$ are skew in $M \dcon C$, and
    \item\label{skewcontract} for each $C \ss \cup_a X_a$,
    the sets $\ab{X_a - C}$ are skew in $M \con C$.
  \end{enumerate}
\end{lemma}
\begin{proof}
  For the first three parts, we may assume using Lemma~\ref{skewdj} that the $X_a$ 
  contain no loops of $M$. Hence (\ref{skewsubsets}), (\ref{skewrefine}) and (\ref{skewdelete}) 
  follow immediately from the circuit characterization of
  skewness in Theorem~\ref{skewequiv}. 

  For (\ref{skewproject}), for each $a \in A$ let $J_a$ be a basis for $X_a$ 
  containing a basis $I_a$ for $X_a \cap C$. Note that $\cl_M(C) = \cl_M(\cup_a I_a)$. 
  By Lemma~\ref{skewgoodbase}, the set $\cup_a J_a$ is a mutual basis for $X_a$, 
  so the set $\cup_a (J_a - I_a)$ is independent in $M \dcon \cup_a I_a = M \dcon C$. 
  For each $a \in A$, we have 
  $X_a \ss \cl_M(J_a) \ss \cl_{M \dcon I_a}(J_a - I_a) \ss \cl_{M \dcon C}(J_a - I_a)$, 
  so $\cup_a(J_a - I_a)$ is a mutual basis for $(X_a : a \in A)$ in $M \dcon C$. 
  Since $X_a \cap X_b \ss \cl_M(\es) \ss \cl_{M \dcon C} (\es)$ for all $a \ne b$, 
  it follows that $(X_a : a \in A)$ are skew in $M \dcon C$. 
  Now, since $M \con C = M \dcon C \del C$, we see that (\ref{skewcontract}) follows from (\ref{skewproject}) and (\ref{skewdelete}).
\end{proof}

\begin{corollary}\label{skewpairrestrict}
  If $X$ and $Y$ are sets in a matroid $M$, then the following are equivalent: 
  \begin{enumerate}[(1)]
    \item\label{sprskew} $X$ and $Y$ are skew in $M$;
    \item\label{sprrestr} $(M \dcon X) | Y = M | Y$; 
    \item\label{sprbasis} there is a set $I$ that is a basis for $X$ in both $M$ and $M \dcon Y$. 
  \end{enumerate}
\end{corollary}
\begin{proof}
  The fact that (\ref{sprskew}) implies (\ref{sprrestr}) follows immediately from 
  Lemma~\ref{skewequiv}(\ref{secon}), 
  and the fact that (\ref{sprrestr}) implies (\ref{sprbasis}) is immediate. 
  Suppose that (\ref{sprbasis}) holds for some set $I$, and let $J$ be a basis for $Y$ in $M$. 
  Since $I$ is independent in $M \dcon Y = M \dcon J$, the sets $I$ and $J$ are disjoint
  with independent union in $M$, and so by Lemma (\ref{skewindep}), they are skew. 
  Since $X \ss \cl_M(I)$ and $Y \ss \cl_M(J)$, it follows from Lemma~\ref{clskew} and~\ref{skewmono}
  that (\ref{sprskew}) holds. 
\end{proof}

The next lemma shows that the closures of collections of disjoint independent sets
are, in a sense, maximally skew. 

\begin{lemma}\label{maxskew}
  Let $\ab{I_a : a \in A}$ be a partition of an independent set $I$ of a matroid $M$. 
  If $\ab{X_a}$ is a skew collection of subsets of $\cl_M(I)$ such that $I_a \ss X_a$ for each $a$, 
  then $X_a \ss \cl_M(I_a)$ for each $a \in A$. 
\end{lemma}
\begin{proof}
  By Lemma~\ref{skewdj}, the sets $\ab{X'_a = X_a - \cl_M(\es)}$ are skew and pairwise disjoint. Clearly $I_a \ss X_a'$ for each $a$. 
  Then $M | \bigcup_a X'_a = \oplus_a (M | X_a')$, and since $\bigcup_a \cX'_a \ss \cl_M(I)$, the set $I$ 
  is a basis for $M | \bigcup_a \cX'_a$, which implies that $I \cap X_a'$ is a basis for $M | X_a'$ for each $a$. 
  Since the $X_a'$ are pairwise disjoint, we must have $I \cap X_a' = I_a$ for each $a$, giving the lemma. 
\end{proof}

We now characterize modularity of pairs in terms of skewness. 

\begin{lemma}\label{modulartoskew}
  If $X$ and $Y$ are sets in a matroid $M$, then the following are equivalent:
  \begin{enumerate}[(1)]
    \item\label{msmodpair} $(X,Y)$ is a modular pair in $M$;
    \item\label{msproj} $X$ and $Y$ are skew in $M \dcon (X \cap Y)$;
    \item\label{mscon} $X - Y$ and $Y-X$ are skew in $M \con (X \cap Y)$. 
  \end{enumerate}
\end{lemma}
\begin{proof}
  First; since the elements of $X \cap Y$ are loops in $M \dcon (X \cap Y)$, 
  it is immediate from Lemma~\ref{skewdj} that (\ref{msproj}) and (\ref{mscon}) are equivalent. 
  It therefore suffices to show that (\ref{msmodpair}) is equivalent to (\ref{mscon}). 

  Suppose that (\ref{msmodpair}) holds, and let $B$ be a mutual basis for $X$ and $Y$ in $M$. 
  By Lemma~\ref{blattice} we know that $B \cap X \cap Y$ is a basis for $X \cap Y$ in $M$. 
  Then $B - (X \cap Y)$ is independent in $M \con (X \cap Y)$, and contains the bases $B \cap (X - Y)$
  and $B \cap (Y-X)$ for $X-Y$ and $Y-X$ respectively in $B \con (X \cap Y)$. 
  These sets are disjoint, so Lemma~\ref{skewofdjbase} gives that $X-Y$ and $Y-X$ are skew in $M \con (X \cap Y)$. 

  Conversely, suppose that $X - Y$ and $Y-X$ are skew in $M \con (X \cap Y)$. 
  Let $B$ be a modular basis in $M \con (X \cap Y)$ for these two sets, and let $I$ be a basis for $X \cap Y$ in $M$. 
  Then $B \cup I$ is independent in $M$ and contains the bases $(B \cap (X - Y)) \cup I$ 
  and $(B \cap (Y-X)) \cup I$ for $X$ and $Y$ respectively, so $(X,Y)$ is modular in $M$. 
\end{proof}

Finally, we prove that skewness coincides with the familiar definition in terms of rank for finite-rank matroids. 

\begin{lemma}\label{skewrank}
  Let $M$ be a matroid. A collection $\ab{X_a : a \in A}$ of sets in $M$ with
  $r_M(\cup_{a \in A}X_a) < \infty$
  is skew in $M$ if and only if 
  $r_M(\cup_{a \in A} X_a) = \sum_{a \in A} r_M(X_a)$. 
\end{lemma}
\begin{proof}
  We may assume by Lemma~\ref{skewdj} that no set $X_a$ contains a loop of $M$. 
  If the $X_a$ are skew in $M$, then Lemma~\ref{skewequiv} gives that 
  $r_M(\cup_a X_a) = r(M | \cup_a X_a) = r\br{\oplus_a (M | X_a)} = \sum_a r_M(X_a)$, 
  as required. 

  Conversely, suppose that $\sum_a r_M(X_a) = r_M\br{\cup_a X_a}$. 
  For each $a \in A$, let $B_a$ be a basis for $X_a$ in $M$. Let $B$ be a basis 
  for $\cup_{a}X_a$ in $M$. We have 
  \[\sum_{a \in A} |B_a| = \sum_{a \in A} r_M(X_a) = r_M\br{\cup_a X_a} = |B| \le \sum_{a \in A} |B \cap B_a| \le \sum_{a \in A} |B_a|.\]
  Equality holds throughout, and the second term is finite, so in fact we have 
  that $B_a \ss B$ for each $a$, and that the $B_a$ are pairwise disjoint. 
  Therefore $B = \cup_a B_a$ is a mutual basis for $\cX$, and the result follows from Lemma~\ref{skewofdjbase}.
\end{proof}

\section{Nullity}\label{nullitysec}

Recall that the nullity of $X \ss E(M)$ is the value $n_M(X) = r^*(M|X)$, 
and therefore that $n_M(X) = |X - I|$ for each basis $I$ of $X$. 
This concept will be very useful for reasoning about connectivity; 
our first lemma establishes a slew of basic properties.
Most are immediate for finite matroids via rank calculations, 
but these often fail in the infinite case. 

\begin{lemma}\label{nullityprop}
  Let $X$ and $Y$ be sets in a matroid $M$. Then
  \begin{enumerate}[(i)]
    \item\label{nullityle} $n_M(X) \le |X|$;
    \item\label{nullitymono} if $X \ss Y$, then $n_M(X) \le n_M(Y)$;
    \item\label{nullityssclosure} if $X \ss Y \ss \cl_M(X)$, then $n_M(Y) = n_M(X) + |Y-X|$;
    \item\label{nullityleaddcard} if $X \ss Y$, then $n_M(Y) \le n_M(X) + |Y-X|$;
    \item\label{nullityrestr} if $X \ss Y$, then $n_{M | Y}(X) = n_M(X)$;
    \item\label{nullitycontract} $n_M(X \cup Y) = n_{M \con X}(Y-X) + n_M(X)$;
    \item\label{nullitysupermod} $n_M(X) + n_M(Y) \le n_M(X \cup Y) + n_M(X \cap Y)$. 
  \end{enumerate}
\end{lemma}
\begin{proof}
  (\ref{nullityle}), (\ref{nullitymono}), (\ref{nullityleaddcard}) and (\ref{nullityrestr}) 
  are immediate from the fact that $n_M(X) = r^*(M | X)$ and the minor-monotonicity of rank. 
  For (\ref{nullityssclosure}), let $I$ be a basis for $X$ in $M$, 
  noting that $I$ is also a basis for $Y \ss \cl_M(X) = \cl_M(I)$. 
  Now $n_M(Y) = |Y - I| = |X - I| + |Y-X| = n_M(X) + |Y-X|$.

  For (\ref{nullitycontract}), let $I$ be a basis for $X$ in $M$, and $J$ be a 
  basis for $X \cup Y$ in $M$ containing $I$. Since $J - I$ is a basis for 
  $Y-X$ in $M \con X$, we have
  \begin{align*}
    n_{M \con X}(Y-X) + n_M(X) &= \abs{(Y-X)-(J-I)} + \abs{X-I} \\
    &= |(X \cup Y) - J| \\ 
    &= n_M(X \cup Y).
  \end{align*}

  To see (\ref{nullitysupermod}), let $I$ be a basis for $X \cap Y$ in $M$, 
  and $J$ be a basis for $Y$ in $M$ that contains $I$.  
  Since $J - I$ contains a basis for $Y-X$ in $M \con X$, we have
  $n_{M \con X}(Y-X) \ge |(Y-X) - (J - I)|$. Now (\ref{nullitycontract}), 
  together with the fact that $J \cap X = I$, gives
  \begin{align*}
    n_M(X \cup Y) + n_M(X \cap Y) &= n_M(X) + n_{M \con X}(Y-X) + n_M(X \cap Y) \\
    &\ge n_M(X) + |(Y-X) - (J-I)| + |(X \cap Y) - I|\\
    &= n_M(X) + |Y-J| \\
    &= n_M(X) + n_M(Y),
  \end{align*}
  as required. 
\end{proof}

The next lemma handles the interaction between nullity and projection.

\begin{lemma}\label{nullityproject}
  Let $X, Y$ and $C$ be sets in a matroid $M$. Then
  \begin{enumerate}[(i)]
    \item\label{npge} $n_{M \dcon C}(X) \ge n_M(X)$;
    \item\label{npcomm} 
      $n_{M \dcon X}(Y) + n_M(X) = n_M(X \cup Y) + |X \cap Y| = n_{M \dcon Y}(X) + n_M(Y)$, 
    \item\label{npconind}
      If $C$ is independent in $M$, then $n_{M \dcon C}(X) = n_M(X \cup C) + |X \cap C|$. 
    \item\label{npconindind}
      If $X$ and $Y$ are independent in $M$, 
      then $n_{M \dcon X}(Y) = n_{M \dcon Y}(X) = n_M(X \cup Y)$;
    \item\label{nple} if $n_M(X) \le n_M(Y)$ and $\cl_M(X) \ss \cl_M(Y)$, then $n_{M \dcon C}(X) \le n_{M \dcon C}(Y)$. 
    \item\label{npeq} if $n_M(X) = n_M(Y)$ and $\cl_M(X) = \cl_M(Y)$, then $n_{M \dcon C}(X) = n_{M \dcon C}(Y)$. 
  \end{enumerate}
\end{lemma}
\begin{proof}
  (\ref{npge}) is clear, since each basis for $Y$ in $M$ 
  contains a basis for $Y$ in $M \dcon X$. 
  For (\ref{npcomm}), note that $X \cap Y$ is a set of loops of $M \dcon X$, 
  so Lemma~\ref{nullityprop} (\ref{nullityssclosure}) and (\ref{nullityrestr}) give
  \begin{align*}
    n_{M \dcon X}(Y) + n_M(X) &= n_{M \dcon X}(Y-X) + |X \cap Y| + n_M(X)  \\
    &= n_M(X \cup Y) + |X \cap Y|,
  \end{align*}
  and the conclusion follows by symmetry. 
  Now (\ref{npconind}) and (\ref{npconindind}) also follow directly from (\ref{npcomm}), since independent sets have nullity zero. 

  For (\ref{nple}), we may assume that $C$ is independent in $M$, so $n_M(C) = 0$. 
  By (\ref{npge}), we have $n_{M \dcon X}(C) \le n_{M \dcon \cl_M(Y)}(C) = n_{M \dcon Y}(C)$. Using (\ref{npcomm}) applied twice, this gives 
  \begin{align*}
    n_{M \dcon C}(X) &= n_{M \dcon C}(X) + n_M(C) \\ 
    &= n_{M \dcon X}(C) + n_M(X) \\
    &\le n_{M \dcon Y}(C) + n_M(Y) \\
    &= n_{M \dcon C}(Y) + n_M(C) \\
    &= n_{M \dcon C}(Y).
  \end{align*}
  Now (\ref{npeq}) is immediate from (\ref{nple}).
\end{proof}

Since any two bases for a set $X$ have the same nullity and closure, 
we have the following corollary :

\begin{corollary}\label{basisswitch}
  If $I$ and $I'$ are bases for a set $X$ in a matroid $M$, 
  then $n_{M \dcon C}(I) = n_{M \dcon C}(I')$ for all $C \ss E(M)$. 
\end{corollary}

Recall that Definition~\ref{lcdef} defined the local connectivity of sets $X$ and $Y$ with the formula
$\sqcap_M(X,Y) = |I \cap J| + r^*(M | I \cup J)$, 
where $I$ and $J$ are arbitrary bases for $X$ and $Y$ respectively. 
As a corollary of the previous lemma, we get that this quantity is well-defined. 

\begin{corollary}\label{lcwelldef}
  Let $X$ and $Y$ be sets in a matroid $M$, and let $I,I'$ be bases for $X$ in $M$ and $J,J'$ be bases for $Y$ in $M$. 
  Then $|I \cap J| + n_M(I \cup J) = |I' \cap J'| + n_M(I' \cup J')$. 
\end{corollary}
\begin{proof}
  By Lemmas~\ref{nullityprop}(\ref{npconind}), 
  Corollary~\ref{basisswitch} and the fact that $\cl_M(I) = \cl_M(I')$,
  \begin{align*}
    |I \cap J| + n_M(I \cup J) &= n_{M \dcon I}(J) = n_{M \dcon I}(J') = n_{M \dcon I'}(J') = |I' \cap J'| + n_M(I' \cup J')
  \end{align*}
  as required. 
\end{proof}

The next few lemmas concern the nullity of arbitrary disjoint unions.
First, we relate nullity to skewness in a manner similar to Lemma~\ref{skewrank}.

\begin{lemma}\label{skewnullity}
  Let $\ab{X_a : a \in A}$ be a pairwise disjoint collection of sets in a matroid $M$.
  Then $n_M(\cup_{a \in A} X_a) \ge \sum_{a \in A}n_M(X_a)$.
  Moreover, if $n_M(\cup_{a \in A}X_a) < \infty$, 
  then $n_M(\cup_{a \in A} X_a) = \sum_{a \in A} n_M(X_a)$ if and only if 
      $(X_a : a \in A)$ is skew in $M$.  
\end{lemma}
\begin{proof}
  For each $a \in A$, let $I_a$ be a basis for $X_a$. 
  Let $I$ be a basis for $I_A$, noting that $I$ is also a basis for $X_A$. Then 
  \[n_M(X_A) = |X - I| = \sum_{a \in A}\abs{X_a-(I \cap I_a)} 
  \ge \sum_{a \in A}|X_a - I_a| = \sum_{a \in A} n_M(X_a),\]
  giving the required inequality. 

  If $n_M(X_A) < \infty$, then because all terms are finite and each sum has 
  only finitely many nonzero summands, we see that equality holds if and only if 
  $I \cap I_a = I_a$ for each $a \in A$, or equivalently $I = I_A$. 
  If $I = I_A$, then $I$ is a mutual basis for the $X_a$, so the $X_a$ 
  are skew in $M$. Conversely, if the $\ab{X_a}$ are skew in $M$, then 
  Lemma~\ref{skewgoodbase} gives that $I = I_A$, as required. 
\end{proof}

The next lemma shows that we can essentially apply the facts from Lemma~\ref{nullityproject}(\ref{nple}) and (\ref{npeq})
infinitely many times at once. 

\begin{lemma}\label{nullityunion}
  Let $\ab{X_a : a \in A}$ and $\ab{Y_a : a \in A}$ be collections of 
  sets in a matroid $M$, with each collection pairwise disjoint. Then
  \begin{enumerate}[(i)]
    \item\label{nule} if $X_a$ and $Y_a$ satisfy $n_M(X_a) \le n_M(Y_a)$ and $\cl_M(X_a) \ss \cl_M(Y_a)$ for all $a \in A$,
    then $n_M(\cup_{a \in A} X_a) \le n_M(\cup_{a \in A} Y_a)$;
    \item\label{nueq} if $X_a$ and $Y_a$ satisfy $n_M(X_a) = n_M(Y_a)$ and $\cl_M(X_a) = \cl_M(Y_a)$ for all $a \in A$,
    then $n_M(\cup_{a \in A} X_a) = n_M(\cup_{a \in A} Y_a)$.
  \end{enumerate} 
\end{lemma}
\begin{proof}
  It clearly suffices to prove (\ref{nule});
  consider a counterexample where $n_M(Y_A)$ is minimized, noting that $n_M(Y_A) < \infty$. 
  If $\ab{Y_a}$ is skew, then $\ab{\cl_M(Y_a)}$ is skew by Lemma~\ref{clskew}, 
  and so $\ab{X_a}$ is skew by Lemma~\ref{skewmono}. Thus
  $n_M(X_A) = \sum_{a \in A}n_M(X_a) \le \sum_{a \in A} n_M(Y_a) = n_M(Y_A)$ 
  by Lemma~\ref{skewnullity}, contrary to the choice of the $Y_a$ 
  as a counterexample.

  Otherwise, by Lemma~\ref{skewequiv}(\ref{sesingle}), there exists $s \in A$
  such that $Y_s$ and $Y_{A - \{s\}}$ are not skew. 
  By Lemma~\ref{skewnullity}, we therefore have $n_M(Y_A) > n_M(Y_{A-s}) + n_M(Y_s)$,
  and therefore  $\ab{Y_a : a \in A - \{s\}}$ is not 
  a counterexample by the minimality in the choice of the $Y_a$. 
  It follows that $n_M(X_{A-\{s\}}) \le n_M(Y_{A - \{s\}})$,
  and so $n_{M \dcon X_s}(X_{A-s}) \le n_{M \dcon X_s}(Y_{A-s})$ by 
  Lemma~\ref{nullityproject}(\ref{nple}). 
  Lemma~\ref{nullityproject}(\ref{npcomm}) now yields
  \begin{align*}
    n_M(Y_A) &= n_{M \dcon Y_s}(Y_{A-\{s\}}) + n_M(Y_s) \\
    &= n_{M \dcon X_s \dcon Y_s}(Y_{A - \{s\}}) + n_M(Y_s) \\
    &\ge n_{M \dcon X_s}(Y_{A - \{s\}}) + n_M(Y_s) \\ 
    &\ge n_{M \dcon X_s}(X_{A - \{s\}}) + n_M(X_s) \\ 
    &= n_M(X_A),
  \end{align*}
  as required. 
\end{proof}


Recall that Definition~\ref{gencon} defined the connectivity of a partition $\cX = \ab{X_a : a \in A}$ of $E(M)$ by 
$\lambda_M(\cX) = r^*(M | \cup_{a \in A} I_a)$, where $I_a$ is a basis in $M$ for $X_a$ for each $a$. 
An immediate corollary is that this quantity is well-defined. 

\begin{corollary}\label{lambdawelldef}
  If $\ab{X_a : a \in A}$ is a collection of pairwise disjoint sets in a matroid $M$, 
  and $\ab{I_a}$ and $\ab{J_a}$ are collections of sets such that for each $a \in A$,
  both $I_a$ and $J_a$ are bases for $X_a$ in $M$, 
  then $n_M(\cup_a I_a) = n_M(\cup_a J_a)$. 
\end{corollary}

Recall that a set $I$ is independent in $M$ if and only if $n_M(I) = 0$. We can use this in conjunction with the previous lemma 
to show that, if $I$ is independent, then removing various subsets of $I$ and replacing them with independent sets 
with smaller closure preserves independence. 

\begin{lemma}\label{indepclosureswitch}
  Let $I$ be an independent set of a matroid $M$, let $\ab{I_a : a \in A}$ be a pairwise disjoint family of 
  subsets of $I$, and let $\ab{J_a}$ be a collection of independent 
  sets of $M$ such that $J_a \ss \cl_M(I_a)$ for each $a \in A$. 
  Then the set $(I - \bigcup_{a \in A}I_a) \cup \bigcup_{a \in A} J_a$ is independent in $M$. 
\end{lemma}
\begin{proof}
  By adding a new set $I - \bigcup_{a \in A} I_a$ to both the collections $I_a$ and $J_a$, 
  we do not change the content of the lemma statement; we may thus assume that 
  $\bigcup_{a \in A} I_a = I$, so we wish to show that $J = \bigcup_{a \in A} J_a$ is independent. 
  By extending each $J_a$ to a basis of $\cl_M(I_a)$, we may also assume that $\cl_M(J_a) = \cl_M(I_a)$ for each $a$. 
  
  Note that the $J_a$ are pairwise disjoint; indeed, if there were some $e \in J_a \cap J_b$ for $a \ne b$,
  a circuit elimination argument would imply that the independent set $I_a \cup I_b$ contains a circuit. 
  It follows from Lemma~\ref{nullityunion} that $0 = n_M(I) = n_M(\cup_a I_a) = n_M(\cup_a J_a)$, 
  so $J$ is independent, as required. 
\end{proof}

\section{Extensions}

In this section, we prove Theorems~\ref{finmodcutdeletion} and~\ref{finmodcutextension}
for infinite matroids, showing that single-element extensions 
are parameterized by modular cuts as per Definition~\ref{infmodcut}. 
Recall that, for an extension $M'$ of a matroid $M$ by an element $e$,
the set $\cF_M^{M'}$ comprises the flats of $M$ that span $e$ in $M'$.  

\begin{theorem}\label{finmodcutdeletioninf}
  If $M'$ is a single-element extension of a matroid $M$,
  then $\cF_{M}^{M'}$ is a modular cut in $M$. 
\end{theorem}
\begin{proof}
  Let $e$ be such that $M = M' \del e$, and let $\cF = \cF_{M}^{M'}$.
  Clearly $\cF$ is closed under taking superflats; it suffices to show that, 
  for any nonempty modular family $\cF_0 \ss \cF$, we have $\bigcap \cF_0 \in \cF$. 
  Let $B$ be a mutual basis for such a family $\cF_0$. 
  We have $e \in \cl_M(F)$ for all $F \in \cF_0$, 
  so Lemma~\ref{closureinter} and the independence of $B$ gives 
  \[e \in \bigcap_{F \in \cF_0} \cl_M(F) = \bigcap_{F \in \cF_0} \cl_M(B \cap F) = 
    \cl_M\br{\bigcap_{F \in \cF_0} (B \cap F)} \ss \cl_M\br{\bigcap \cF_0},\] 
  and thus $\bigcap \cF_0 \in \cF$, as required. 
\end{proof}

The above argument showed that $\cF_{M}^{M'}$ is closed under taking intersections of 
arbitrary modular families, not just pairs and chains. In fact, this is true for every modular cut. 
The proof below uses a version of Zorn's lemma for set families: specifically, 
the statement that, if $\cX$ is collection of sets such that $\cX$ 
contains a lower bound for every chain $\cC \ss \cX$, 
then $\cX$ has a minimal element. 

\begin{lemma}\label{modcutinter}
  If $\cF$ is a modular cut of a matroid $M$, and $\cF_0 \subseteq \cF$ is a nonempty 
  modular family, then $\cap \cF_0 \in \cF$. 
\end{lemma}
\begin{proof}
  Let $F_0 = \cap \cF_0$. 
  Let $B$ be a mutual basis for $\cF_0$ and let $\cU = \set{\cl_M(I) \colon I \ss \cB}$. 
  $B$ is also a mutual basis for $\cU$ and hence for $\cU \cap \cF$. 
  
  \begin{claim}\label{uclinter}
    $\cap \cX \in \cU$ for all $\cX \ss \cU$.
  \end{claim}
  \begin{subproof}
    By the definition of $\cU$, each $F \in \cU$ satisfies $F = \cl_M(F \cap B)$.
    Lemma~\ref{closureinter} now gives 
  $\cap \cX = \cap_{F \in \cX}\cl_M(F \cap B) = \cl_M\br{\cap_{F \in \cX}(F) \cap B} \in \cU$.
  \end{subproof}
  
  
  We argue that $\cU \cap \cF$ has a minimal element. By Zorn's lemma, it suffices to 
  show that every chain $\cC \ss \cU \cap \cF$ satisfies $\cap \cC \in \cU \cap \cF$. 
  We get $\cap \cC \in \cU$ by~\ref{uclinter}. 
  If $\cC$ is finite, then $\cap \cC$ is the minimal element of $\cC$, so is contained in $\cC$.
  If $\cC$ is infinite, then since $B$ is a mutual basis for $\cU \cap \cF$, we get
  $\cap \cC \in \cF$ from (\ref{forallmodchain}). 

  Let $F_1$ be a minimal element of $\cU \cap \cF$. 
  Since $F = \cl_M(F \cap B)$ for each $F \in \cF_0$, 
  we have $\cF_0 \ss \cU$, so $F_0 \in \cU$ by~\ref{uclinter}. 
  For each $F \in \cF_0$, using~\ref{uclinter} and (\ref{forallmodpair}), 
  we have $F \cap F_1 \in \cU \cap \cF$, 
  so the minimality of $F_1$ gives that $F_1 \ss F$. It follows that $F_1 \ss \cap \cF_0 = F_0$. 
  Now (\ref{superflat}) gives $F_0 \in \cF$, as required. 
\end{proof}

This last lemma tells us that we could have equivalently defined a modular cut by replacing 
(\ref{forallmodpair}) 
and (\ref{forallmodchain}) by a statement about all modular families, but the way we defined it, 
it is easier to show that some collection is a modular cut.

We now prove the infinite analogue of Theorem~\ref{finmodcutextension}: each modular cut 
gives rise to a unique single-element extension. 

\begin{theorem}\label{modcutextension}
  If $\cF$ is a modular cut in a matroid $M$, and $e \notin E(M)$,
  then there is a unique extension $M'$ of $M$ by $e$ for which $\cF_M^{M'} = \cF$. 
\end{theorem}
\begin{proof}
  Let $E = E(M)$, and let $\cX$ be the collection of subsets of $E$ whose closure is in $\cF$. 
  Note that any superset of an element of $\cX$ is in $\cX$. 
  Let $\cI_M$ be the collection of $M$-independent sets, 
  let $\cI_0 = \cI_M \cap \cX$, let $\cI_1 = \cI_M - \cX$, and let $\cI_e = \{I \cup \{e\} : I \in \cI_1\}$. 
  Finally, let $\cI = \cI_0 \cup \cI_1 \cup \cI_e$. 
  Our first claim reduces the problem to showing that $\cI$ is the collection of independent 
  sets of a matroid. 
  
  \begin{claim}\label{ismatroid}
    If $\cI$ is the collection of independent sets of a matroid $M'$ on $E \cup \{e\}$, 
    then $M'$ is the unique matroid for which $M' \del e = M$ and $\cF_M^{M'} = \cF$.
  \end{claim}
  \begin{subproof}
    Suppose that the given $M' = (E \cup \{e\}, \cI)$ is a matroid. 
    Since $\cI_M$ is the collection of sets in $\cI$ not containing $e$, we know that $M' \del e = M$. 

    Let $I$ be an independent set of $M = M' \del e$. 
    By construction, we have $\cl_M(I) \in \cF_{M}^{M'}$ if and only if $e \in \cl_{M'}(I)$, 
    which holds if and only if $I \cup \{e\} \notin \cI_e$, or in other words $I \in \cI_0$. 
    This is equivalent to the condition that $\cl_M(I) \in \cF$. Since this holds for every 
    independent set $I$, we can conclude that $\cF$ and $\cF_{M}^{M'}$ contain exactly 
    the same closures of independent sets, so are equal. 

    Suppose now that $M_1$ is any matroid with $M_1 \del e = M$ and $\cF_{M}^{M_1} = \cF$. 
    For each set $I \subseteq E$, the set $I \cup \{e\}$ is independent in $M_1$ 
    if and only if $e \notin \cl_{M_1}(I)$, which holds if and only if 
    $\cl_M(I) \notin \cF_{M}^{M_1} = \cF$. This implies that any two choices $M_1, M_1'$ for $M_1$ 
    have the same independent sets containing $e$ and satisfy $M_1 \del e = M_1' \del e$, 
    so are equal. Therefore the matroid $M'$ is unique with the stated properties. 
  \end{subproof}
  
  We show that $\cI$ is the indeed the collection of independent sets of a matroid $M'$ with ground set $E(M) \cup \{e\}$. 
  Note that $\cI \cap 2^{E} = \cI_{M} \cap 2^E$. 
  

  \begin{claim}\label{notcoloop}
    If $E \notin \cF$, then $(E \cup \{e\}, \cI)$ is a matroid. 
  \end{claim}
  \begin{subproof}
    Since $\cF$ is closed under taking superflats, we have $\cF = \es$, which implies
    that $\cI = \{I \ss 2^{E \cup \{e\}} : I - \{e\} \in \cI\}$. In other words, 
    $\cI$ is the collection of independent sets of the matroid obtained from $M$ by adding 
    $e$ as a coloop. 
  \end{subproof}

  By the last claim, we can now assume that $E \in \cF$. To verify the independence axioms for $\cI$ 
  is the collection of independent sets of a matroid, we need to characterize 
  when a set $I$ is maximal $\cI \cap 2^X$ for some $X$. The next two claims
  do this in the the case where $e \notin I$ and the case where $e \in I$.

  \begin{claim}\label{foo1}
    Let $I \ss X \ss E(M) \cup \{e\}$ with $e \notin I$.
    Then $I$ is a maximal element of $\cI \cap 2^X$ if and only if 
    $I$ is a basis for $X - \{e\}$ in $M$, and if $e \in X$, then $X - \{e\} \in \cX$. 
  \end{claim}
  \begin{subproof}
    Suppose that $I$ is maximal in $\cI \cap 2^X$. Since $I \subseteq E(M)$ we have 
    $I \in \cI_M$ and $I \ss X - \{e\}$. If $I \subseteq J \subseteq X - \{e\}$ with $J \in \cI_M$, 
    then $J \in \cI \cap 2^X$, so $J = I$ by maximality. It follows that $I$ is a basis of $X - \{e\}$ in $M$. Furthermore, if $e \in X$ and $X - \{e\} \notin \cX$, so $I \notin \cX$, 
    which implies that $I \cup \{e\} \in \cI_e \cap 2^X$, contradicting the maximality of $I$. 
    So the forwards implication holds. 
    
    Conversely, suppose that $I$ is a basis for $X - \{e\}$ in $M$, and that if $e \in X$, then $X - \{e\} \in X$. Let $J \in \cI$ with $I \ss J \ss X$; we wish to show that $I = J$. 
    Since $J - \{e\} \in \cI_M$
    and $I \ss J - \{e\} \ss X - \{e\}$ and $I$ is a basis, we have $I = J - \{e\}$, 
    so it suffices to show that $e \notin J$. 
    
    If $e \in J$, then since $e \in X$ we have $X - \{e\} \in \cX$, 
    so $\cl_M(J - \{e\}) = \cl_M(I) = \cl_M(X - \{e\}) \in \cF$, and thus $J - \{e\} \in \cX$. 
    But the fact that $e \in J \in \cI$ implies that $J \in \cI_e$, so $J - \{e\} \notin \cX$, 
    a contradiction. 
  \end{subproof}

  \begin{claim}\label{foo2}
    Let $I \ss X \ss E \cup \{e\}$ with $e \in I$. Then $I$ is a maximal element of 
    $\cI \cap 2^X$ if and only if $I \in \cI_e$, and either 
    \begin{enumerate}[(i)]
      \item\label{xnotin} $X - \{e\} \notin \cX$ and $I - \{e\}$ is a basis of $X - \{e\}$ in $M$, or
      \item\label{xin} $X - \{e\} \in \cX$, and $\cl_M(I-\{e\})$ is a maximal proper subflat of $\cl_M(X - \{e\})$. 
    \end{enumerate}
  \end{claim}
  \begin{subproof}
    Suppose that $I$ is maximal in $\cI \cap 2^X$. Since $e \in I \in \cI$, 
    we have $I \in \cI_e$. 
    
    If $X - \{e\} \notin \cX$, then let $J \in \cI_M$ with
    $I - \{e\} \ss J \ss X - \{e\}$. Since $J \ss X - \{e\} \notin \cX$, 
    we have $J \notin \cX$, so $J \cup \{e\} \in \cI_e \cap 2^X \ss \cI \cap 2^X$, 
    so $I = J \cup \{e\}$ by the maximality of $I$. This implies that $J - \{e\} = I$; 
    it follows that $I$ is a basis of $X - \{e\}$ in $M$. 

    If $X - \{e\} \in \cX$, then let $I'$ be a basis of $X - \{e\}$ containing $I - \{e\}$. 
    For each $f \in I' - I$, since $I$ is maximal, we have $e \in I \cup \{f\} \notin \cI$, 
    which implies that $(I - \{e\}) \cup \{f\} \in \cX$. It follows that the modular cut 
    $\cF$ contains $\cl_M(I - \{e\} \cup \{f\})$ for all $f \in I' - I$. The set $I'$ 
    certifies that the flats $\cl_M(I - \{e\} \cup \{f\}) : f \in I' - I$ form a modular 
    family, and so their intersection lies in $\cF$ by Lemma~\ref{modcutinter}.
    This intersection has $I - \{e\}$ as a basis, which gives $I - \{e\} \notin \cX$, contradicting $I \in \cI_e$. 

    For the converse direction, suppose that (\ref{xnotin}) or (\ref{xin}) holds for $I$ and $X$. Let $J \in \cI$ with $I \ss J \ss X$; we wish to show that $J = I$, or equivalently 
    that $J - \{e\} = I - \{e\}$. Note that since $e \in J \in \cI$ we have $J \in \cI_e$. 
    
    If (\ref{xnotin}) holds, 
    Since $J \in \cI$ we have $J - \{e\} \in \cI_M$; 
    the fact that $I - \{e\}$ is a basis of $X - \{e\}$ gives that $J - \{e\} = I - \{e\}$ 
    and so $J = I$, as required. 

    Suppose now that (\ref{xin}) holds. For each $f \in J - I$, 
    since $F = \cl_M((I - \{e\} \cup \{f\}))$ is a proper superflat of $\cl_M(I - \{e\})$ 
    that is contained in $\cl_M(X - \{e\})$, it is equal to $\cl_M(X - \{e\})$ 
    and is thus in $\cF$. But this gives 
    $\cX \ni \cl_M((I - \{e\}) \cup \{f\}) \ss \cl_M(J - \{e\})$, 
    so $J - \{e\} \in \cX$, which contradicts $J \in \cI$. 
  \end{subproof}

  We now proceed to show that $\cI$ is the collection of independent sets of a matroid. 
  It is obvious that $\cI$ is nonempty and closed under taking subsets; 
  the next two claims verify the less trivial axioms. 

  \begin{claim}
    If $I$ is a non-maximal element of $\cI \cap 2^{E \cup \{e\}}$ and $B$ is a maximal element of 
    $\cI \cap 2^{E \cup \{e\}}$, then there exists $f \in B - I$ for which $I \cup \{f\} \in \cI$. 
  \end{claim}
  \begin{subproof}
    
    Suppose not. Note that $I$ is a maximal element of $\cI \cap 2^{I \cup B}$, 
    but there exists $f \in E - I \cup B$ such that $I \cup \{f\} \in \cI$. 

    
    \textbf{Case 1 : $e \notin I \cup B$}. By~\ref{foo1}, the set $I$ is a basis for $I \cup B$ in $M$
    and $B$ is a basis for $E$ in $M$, so $I \cup \{x\}$ is dependent in $M$, and so not in $\cI$,
    for each $x \in E-B$. Since $I \cup \{f\} \in \cI$ we have $f = e$, so $I \in \cI_1$ 
    and therefore $E = \cl_M(I) \notin \cF$, a contradiction. 

    \textbf{Case 2 : $e \in B - I$}. 
    By~\ref{foo1}, 
    the set $I$ is a basis for $I \cup (B - \{e\})$ and that $B - \{e\} \in \cX$. 
    But then $B \notin \cI_e$ and so $B \notin \cI$, contrary to the choice of $B$. 

    \textbf{Case 3: $e \in I - B$.}
    By~\ref{foo1}, the set $B$ is a basis for $E$ in $M$, so $\cl_M(B) = E \in \cF$ and thus $B \in \cX$ and so $(I \cup B) - \{e\} \in \cX$. Since $I$ is maximal in $\cI \cap 2^{I \cup B}$, 
    we get from ~\ref{foo2} that $\cl_M(I - \{e\})$ is a maximal proper subflat of $\cl_M((I \cup B)- \{e\})$. 
    
    The latter set contains $\cl_M(B) = E$, so $\cl_M(I - \{e\})$ is thus a hyperplane of $M$.
    Now $f \in E - (I \cup B)$ and $I \cup \{f\} \in \cI$, so $(I - \{e\}) \cup \{f\} \in \cI_1$. 
    But $I - \{e\}$ spans a hyperplane of $M$ and $(I - \{e\}) \cup \{f\}$ is independent, 
    so is a basis of $M$. Therefore $\cl_M(I - \{e\} \cup \{f\}) = E \in \cF$, 
    which contradicts $(I - \{e\}) \cup \{f\} \in \cI_1$. 

    \textbf{Case 4 : $e \in I \cap B$}. 
    The fact that $I,B \in \cI$ give that $I - \{e\} \notin \cX$ and $B - \{e\} \notin \cX$. 
    Since $I$ is a maximal element of $\cI \cap 2^{I \cup B}$, we get from~\ref{foo2} that 
    $I - \{e\}$ is a basis of $B - \{e\}$. Since $E \in \cX$, we also get from~\ref{foo2}
    that $\cl_M(B - \{e\})$ is a maximal proper subflat of $\cl_M(E) = E$, so 
    $\cl_M(B - \{e\})$ is a hyperplane of $M$.

    Now $f \notin I \cup B$ and $I \cup \{f\} \in \cI$, so $(I - \{e\}) \cup \{f\} \in \cI_1$
    and thus $f \notin \cl_M(I  - \{e\}) = \cl_M(B - \{e\})$. But $(I - \{e\})$ spans
    a hyperplane of $M$ and so $(I - \{e\}) \cup \{f\}$ is a basis of $M$. 
    Therefore $\cl_M((I - \{e\}) \cup \{f\}) = E \in \cF$, 
    contradicting $(I - \{e\}) \cup \{f\} \in \cI$. 
  \end{subproof}
  
  \begin{claim}
    If $X \ss E \cup \{e\}$, then each $I \in \cI \cap 2^X$ is contained in a maximal set in $\cI \cap 2^X$. 
  \end{claim}
  \begin{subproof}
    As before, we split into cases. 

    \textbf{Case 1 : $e \notin I$ and either $e \notin X$ or $X - \{e\} \in \cX$}.
    Since $e \notin \cI$, we have $I \in \cI_M$; let $J$ be a basis for $X$ in $M$ with $I \ss J$. 
    By~\ref{foo1}, the set $J$ is maximal in $\cI \cap 2^X$, as required. 

    \textbf{Case 2 : $e \in X - I$ and $X - \{e\} \notin \cX$.}
    Since $e \notin \cI$, we have $I \in \cI_M$; 
    let $J$ be a basis for $X - \{e\}$ in $M$ containing $I$. By~\ref{foo2}, the set $J \cup \{e\}$ 
    is maximal in $\cI \cap 2^X$, giving the result. 

    \textbf{Case 3 : $e \in I$ and $X - \{e\} \notin \cX$.}
    Since $e \in I \in \cI$, we have $I - \{e\} \in \cI_1$. Let $J$ be a basis for $X - \{e\}$ 
    in $M$ that contains $I - \{e\}$. 
    Since $J \ss X - \{e\} \in I_1$, we have $J \in \cI_1$ and so $J \cup \{e\} \in \cI_e$;
    By~\ref{foo2}, the set $J \cup \{e\}$ is 
    is maximal in $\cI \cap 2^X$, as required. 

    \textbf{Case 4 : $e \in I$ and $X - \{e\} \in \cX$.}
    Since $e \in I \in \cI$, we have $I - \{e\} \in \cI_1$. Let $J$ be a basis for $X - \{e\}$ 
    in $M$ that contains $I - \{e\}$. If there is some $f \in J$ for which 
    $\cl_M(J - \{f\}) \notin \cF$, then $J - \{f\} \in \cI_e$, 
    and so $J \cup \{e\}  - \{f\}$ contains $I$ and  is maximal in $\cI \cap 2^X$ by~\ref{foo2}. 

    Otherwise, we have $\cl_M(J \cup \{e\} - \{f\}) \in \cF$ for all $f \in J - (I - \{e\})$; 
    the collection of all such flats is modular as certified by the independent set $J$, 
    so using Lemma~\ref{modcutinter} we have
    \[\cF \ni \bigcap_{f \in J - (I - \{e\})} \cl_M(J - \{f\}) = \cl_M\br{\bigcap_f (J - \{f\})} = \cl_M(I - \{e\}).\]
    This contradicts $I \in \cI$. 
  \end{subproof}
  By the last two claims and~\ref{ismatroid}, the theorem holds. 
\end{proof}

\section{Quotients and projections}





Recall that $N$ is a \emph{quotient} of $M$, and $N \preceq M$, 
if $E(M) = E(N)$ and $\cl_M(X) \ss \cl_M(X)$ for all $X$, and that $N$ is a \emph{projection} of $M$ 
if there is a matroid $P$ and a set $X$ for which $N = P \con X$ and $M = P \del X$. 

Being careful not to rely on proofs in the literature that exploit finiteness, 
we now establish some basic properties for quotients and projections. 
It is easy to see that projections interact nicely with duality, and that projections are quotients. 

\begin{lemma}
  If $M$ and $N$ are matroids on the same ground set, then $M^*$ is a projection of $N^*$ 
  if and only if $N$ is a projection of $M$. Moreover, every projection of $M$ is a quotient of $M$.
\end{lemma}
\begin{proof}
  If $N = P \con X$ is a projection of $M = P \del X$, 
  then $M^* = P^* \con X$ is a quotient of $N^* = P^* \del X$ via $P^*$. 
  Similarly, if $M^*$ is a projection of $N^*$, then $N = N^{**}$ is a projection of $M = M^{**}$. 
  If $N = P \con X$ and $M = P \del X$, then for all $Y \ss E(M)$ we have 
  $\cl_M(Y) = \cl_P(Y) - X \ss \cl_P(Y \cup X) - X = \cl_N(Y)$, so $N \preceq M$ as required. 
\end{proof}

We can also choose the set $X$ in the projection to be independent and coindependent in $P$. 

\begin{lemma}\label{indepprojection}
  If $N$ is a projection of a matroid $M$, then there exists a matroid $P$ and an independent, coindependent set $X$
  of $P$ for which $N = P \con X$ and $M = P \del X$. 
\end{lemma}
\begin{proof}
  Let $Q$ and $Y$ be such that $N = Q \con Y$ and $M = Q \del Y$. Let $I$ be a basis for $Y$ in $Q$,
  and $Q' = Q \del (Y - I)$. We have $N = Q \con Y = Q \con I \del (Y - I) = Q' \con I$, 
  and also $M = Q \del Y = Q' \del I$. 
  Now let $X$ be a cobasis for $I$ in $Q'$ and let $P = Q' \con (I-X) = Q \con (I - X) \del (Y - I)$. We have
  $N = Q' \con I = (Q' \con (I-X)) \con X = P \con X$, 
  and (since $I-X$ comprises just coloops of $Q' \del I$), 
  we know that 
  \[M = Q' \del I = (Q' \del X) \del (I - X) = (Q' \del X) \con (I - X) = P \del X.\] 
  Now $X \cup (I - X) = I$ is independent in $Q$, so $X$ is independent in $P = Q \con (I - X) \del (Y - I)$. 
  Since $X$ is coindependent in $Q'$, it is also coindependent in $Q' \con (I-X) = P$. 
  It follows that $P$ and $X$ satisfy the lemma. 
\end{proof}

We now show that the various combinatorial characterizations of quotients that work 
for finite matroids are still equivalent for infinite matroids, and behave well under duality.
Most of these arguments essentially appear in the proof of [\ref{oxley}, Proposition 7.3.6], 
but some parts of that proof use finiteness, so we include the entire proof for general matroids here.

\begin{theorem}\label{quotequiv}
  Let $M$ and $N$ be matroids with ground set $E$. The following are equivalent: 
  \begin{enumerate}[(1)]
    \item\label{quot} $N \preceq M$,
    \item\label{quotdual} $M^* \preceq N^*$,
    \item\label{quotcl} $\cl_M(X) \subseteq \cl_N(X)$ for all $X \subseteq E$. 
    \item\label{quotclindep} $\cl_M(I) \subseteq \cl_N(I)$ for every independent set $I$ of $M$,
    \item\label{quotflat} every flat of $N$ is a flat of $M$,
    \item\label{quotcircuit} every circuit of $M$ is a union of circuits of $N$,
    \item\label{quotrank} $\relrank{N}{Y}{X} \le \relrank{M}{Y}{X}$ for all $X \ss Y \ss E$. 
  \end{enumerate}
\end{theorem}
\begin{proof}
  The equivalence of (\ref{quot}) and (\ref{quotcl}) holds by definition. 
  We now argue in turn that (\ref{quotcl}) is equivalent to each of
  (\ref{quotclindep}), (\ref{quotflat}), (\ref{quotcircuit}) and (\ref{quotrank}). 
  
  Clearly (\ref{quotcl}) implies (\ref{quotclindep}). 
  Conversely, suppose that (\ref{quotclindep}) holds, and let $X \ss E$. 
  Let $I$ be a basis for $X$ in $M$. Then $\cl_M(X) = \cl_M(I) \ss \cl_N(I) \ss \cl_N(X)$ 
  by monotonocity of closure, so (\ref{quotcl}) holds. 
  
  If (\ref{quotcl}) holds, and $F$ is a flat of $N$. 
  Then $F \ss \cl_M(F) \ss \cl_N(F) = F$, so $F$ is a flat of $M$. 
  Conversely, suppose (\ref{quotflat}) holds and let $X \ss E$. 
  For each flat $F$ of $N$ containing $X$, since $F$ is a flat of $M$
  containing $X$, we have $\cl_M(X) \ss F$. So $\cl_M(X)$ is contained 
  in the intersection of the flats of $N$ containing $X$, which is equal to $\cl_N(X)$. 
  So (\ref{quotcl}) and (\ref{quotflat}) are equivalent. 

  Suppose that (\ref{quotcl}) holds, let $C$ be a circuit of $M$, and let $e \in C$. 
  Then $e \in \cl_M(C - \{e\}) \ss \cl_N(C - \{e\})$, so some circuit $C_e$ of $N$ 
  satisfies $e \in C_e \ss C$. Therefore $C = \cup_{e \in C} C_e$ is a union of circuits of $N$.
  Conversely, assume that (\ref{quotcircuit}) holds, and let $X \subseteq E$ and $e \in \cl_M(X)$. 
  Then some circuit $C$ of $M$ satisfies $e \in C \ss X \cup \{e\}$, and $C$ is a union of 
  circuits of $N$, so $C$ contains a circuit $C_e$ of $N$ with $e \in C_e$. 
  Now $e \in C_e \ss X \cup \{e\}$, so $e \in \cl_N(X)$. 
  Therefore (\ref{quotcl}) and (\ref{quotcircuit}) are equivalent. 
  
  Suppose that (\ref{quotcl}) holds, and let $X \ss Y \ss E$. 
  Let $I$ be a basis for $(M \con X) | (Y - X)$. 
  Then $\cl_{N \con X}(I) = \cl_N(I \cup X) - X \supseteq \cl_M(I \cup X) - X = \cl_{M \con X}(I) = Y-X$, 
  so $I$ contains a basis $J$ for $(N \con X) | (Y -X)$ and therefore 
  $r_N(Y | X) = |J| \le |I| = r_M(Y | X)$. So (\ref{quotcl}) implies (\ref{quotrank}). 
  Suppose that (\ref{quotrank}) holds and let $X \ss E$ and $e \in \cl_M(X)$. 
  Then $0 = r_M(X \cup \{e\} | X) \ge r_N(X \cup \{e\} | X)$, so $e \in \cl_N(X)$.
  Therefore (\ref{quotcl}) and (\ref{quotrank}) are equivalent. 

  We now show that (\ref{quot}) implies (\ref{quotdual}); applying this fact to 
  $M^*$ and $N^*$ will immediately give that (\ref{quotdual}) implies (\ref{quot}). 
  Suppose that (\ref{quot}) holds, and let $F$ be a flat of $M^*$. 
  Let $\cH$ be the collection of hyperplanes of $M^*$ containing $F$, so $F = \bigcap \cH$. 
  For each $H \in \cH$, the set $E - H$ is a cocircuit of $M^*$ so is a circuit of $M$; 
  since (\ref{quotcircuit}) holds for $M$ and $N$, the set $E - H$ is a union of circuits of $N$, 
  so $E - H$ is an union of cocircuits of $N^*$, so its complement $H$ 
  is an intersection of hyperplanes of $N^*$ and is thus a flat of $N^*$. 
  This holds for each $H \in \cH$, so $F = \bigcap \cH$ is an intersection of flats of $N^*$
  and so is itself a flat of $N^*$. We have argued that each flat of $M^*$ is a flat of $N^*$; 
  the fact that (\ref{quotflat}) implies (\ref{quot}) for $M^*$ and $N^*$ now gives (\ref{quotdual}).
\end{proof}

We now state some basic properties of quotients whose proofs are straightforward using 
Lemma~\ref{quotequiv} and duality. The fact asserted by (\ref{weakmap})
is often phrased by saying that $N$ is a \emph{weak image} of $M$. 

\begin{lemma}\label{minorquot}
  If $M$ and $N$ are matroids with $N \preceq M$, then 
  \begin{enumerate}[(i)]
    \item $N \con C \del D \preceq M \con C \del D$ for all disjoint $C,D \ss E(N)$;
    \item $N \dcon C \preceq M \dcon C$ for all $C \ss E(N)$;
    \item\label{weakmap} every independent set in $N$ is independent in $M$;
    \item every basis for $N$ is contained in a basis for $M$; and
    \item every spanning set for $M$ is spanning in $N$.
  \end{enumerate}
\end{lemma}

We use the following simple fact to recognize when quotients are in fact equalities.

\begin{lemma}\label{quotientext}
  If $N$ is a quotient of a matroid $M$ for which every circuit of $N$ is dependent in $M$, then $N = M$. 
\end{lemma}
\begin{proof}
  Let $I$ be an independent set of $M$. Since $B$ contains no circuit of $M$, it contains 
  no circuit of $N$, so $I$ is $N$-independent. 
  Conversely, since $N \preceq M$, every $N$-independent set is $M$-independent; thus $N = M$. 
\end{proof}

To show where the equivalence of quotients and projections breaks down, 
we identify a property common to all projections, but not to quotients. 
In particular, the hypothesis holds if $M$ and $N$ have a basis in common. 

\begin{lemma}\label{commonbasis}
  Let $N$ be a projection of a matroid $M$. If $N$ has an independent set that is spanning in $M$, then $M = N$. 
\end{lemma}
\begin{proof}
  Let $P$ and $X$ be given by Lemma~\ref{indepprojection}, so $X$ is independent and coindependent in $P$, 
  while $N = P \con X$ and $M = P \del X$. 
  Suppose that $B$ is independent in $N$ and spanning in $M$.
  Then $B \cup X$ is an independent set of $P$ that contains the spanning set $B$ of $P$, 
  so $B = B \cup X$ and hence $X = \es$. It follows that $M = N$. 
\end{proof}

The next lemma shows that a pair of matroids related by a quotient may have a basis in common,
but may still not be equal. In fact, there may be a nested pair of bases for the two matroids 
with infinite set difference. 

Our construction uses infinite graphic matroids, which are discussed in~[\ref{bd11}]. 
The main idea is essentially topological; the matroids $M$ and $N$ both come from the same 
infinite graph, of which $N$ contains the finite and infinite cycles as its circuits, 
and $M$ contains only the finite cycles. 
Since we will not need topology anywhere else, 
we sidestep the issue of making the notion of an `infinite cycle' precise by 
working in the dual and phrasing the argument in terms of bonds. 

\begin{lemma}\label{badgraphic}
  There exists a pair of matroids $M,N$ with $N \preceq M$, such that $M$ is finitary and $N$ is cofinitary, 
  and there are sets $B, B_M, B_N$ so that
  \begin{itemize}
    \item $B$ is a basis of both $M$ and $N$,
    \item $B_M$ is a basis of $M$, and 
    \item $B_N$ is a basis of $N$ contained in $B_N$, and $B_M - B_N$ is infinite. 
  \end{itemize}
\end{lemma}

\begin{proof}
  Let $G$ be an infinite grid: i.e. the graph with vertex set $\bZ \times \bZ$, 
  whose edges are the pairs $((i,j), (i \pm 1, j))$ and $((i,j), (i, j \pm 1))$ for all $i, j \in \bZ$. 
  Let $E = E(G)$. 

  Let $M$ be the finite cycle matroid of $G$; that is, the finitary matroid on $E(G)$ whose circuits are the finite cycles of $G$. 
  Let $N$ be the dual of the finite bond matroid of $G$; that is, the cofinitary matroid on $E(G)$ whose cocircuits are the finite bonds of $G$. 
  By [\ref{bd11}, Theorem 1], the circuits of $M^*$ are precisely the bonds of $G$, finite or infinite. 
  
  \begin{claim}
    $N \preceq M$. 
  \end{claim}
  \begin{subproof}
    By construction, the circuits of $N^*$ are the finite bonds of $G$, and the circuits of $M^*$ are all the bonds of $G$. 
    Therefore every circuit of $N^*$ is a circuit of $M^*$, so by Lemma~\ref{quotequiv}, we have $M^* \preceq N^*$ and thus $N \preceq M$. 
  \end{subproof}

  \begin{claim}
    For every one-way infinite Hamilton path $P$ of $G$, the set $E(P)$ is a basis of both $M$ and $N$. 
  \end{claim}
  \begin{subproof}
    The fact that $P$ is a Hamilton path means that $E(P)$ is maximal containing no finite cycle of $G$,
    so is a maximal subset of $E$ containing no circuit in $M$, and is thus a basis of $M$. 
    Since $\cl_N(E(P)) \supseteq \cl_M(E(P)) = E$, it follows that $E(P)$ is spanning in $N$, 
    so it remains to show that $E(P) - \{e\}$ is not spanning in $N$ for each $e \in E(P)$. 

    Let $X_1$ and $X_2$ be the two components of $P - \{e\}$, where $X_1$ is finite. Since every vertex of $G$ has finite 
    degree, the set $\delta_G(X_1)$ is a finite bond of $G$ that is disjoint from $E(P) - \{e\}$,
    so $E(P) - \{e\}$ is disjoint from a cocircuit of $N$ 
    and is thus nonspanning in $N$, as required. 
  \end{subproof}

  Clearly $G$ has a one-way infinite Hamilton path (for instance, one spiralling outwards from the origin), 
  so this gives the required set $B$. 

  For each $i \in \bZ$, let $P_i$ be the edge set of the 
  `horizontal' two-way infinite path of $G$ with vertex set $\bZ \times \{i\}$, 
  and let $Q_i$ be the edge set of of the `vertical' two-way infinite path 
  with vertex set $\{i\} \times \bZ$. 

  \begin{claim}
    The set $S = \bigcup_{i \in \bZ} P_i$ is $N$-spanning, and $S \cup Q_0$ is $M$-independent. 
  \end{claim}
  \begin{subproof}
    To show that $S$ is spanning in $N$, we need to show that it intersects every cocircuit $K$ of $N$. 
    Each such $K$ is a finite bond of $G$, so there is a finite nonempty set $A \ss V(G)$ for which $K = \delta_G(A)$.
    Let $(i,j) \in A$ with $i$ as large as possible; since $(i+1,j) \notin A$, the edge from $(i,j)$ to $(i+1,j)$ is in $S \cap K$.

    The fact that $S \cup Q_0$ is independent in $Q$ follows easily from the fact that it contains the edges of no finite cycle of $G$. 
  \end{subproof}

  By the last claim, there is a basis $B_N$ of $N$ contained in $S$, and there is a basis $B_M$ of $M$ containing $S \cup Q_0$. 
  Since $Q_0$ is infinite, so is $B_M - B_N$, so $B_M$ and $B_N$ give the required sets to finish the proof. 
\end{proof}

The last two lemmas imply that quotients and projections do not coincide for general infinite matroids, 
even when the first matroid is cofinitary and the second is finitary. 

\begin{corollary}
  There exist matroids $N$ and $M$ for which $N$ is cofinitary and $M$ is finitary, and $N$ is a quotient of $M$ but not a projection of $M$. 
\end{corollary}
\begin{proof}
  Let $N$ and $M$ be the matroids from Lemma~\ref{badgraphic}. 
  Since $B_M$ is a basis of $M$ but not $N$, we have $M \ne N$. 
  Since the common basis $B$ exists,
  Lemma~\ref{commonbasis} implies that $N$ is not a projection of $M$. 
\end{proof}

Theorem~\ref{badgraphic} also implies that there exists matroids $N,M$ with $N \preceq M$ that have a common basis but are distinct. 
We show later that this cannot happen if $N$ is finitary or $M$ is cofinitary; see Theorem~\ref{finitarycommonbasis}. 

\section{Discrepancy}

In this section and the next, we study the circumstances in which a quotient $N$ of an 
infinite matroid $M$ can be concluded to be a projection of $M$. 
For this purpose, it will be helpful to develop a notion of the `distance' between $N$ and $M$, 
which we define as the difference in cardinality
between an $N$-basis $I_0$ for $X$, and an $M$-basis $I$ for $X$ that contains $I_0$. 
In this section, we show that if $N$ is finitary, this difference does not depend on the choice of bases. 
We first give a definition and a pair of lemmas to facilitate coming discussion. 

\begin{definition}
  Let $N,M$ be matroids. A \emph{$(N,M)$-basis pair} for a set $X$ is a pair 
  $(B_0,B)$, where $B$ is an $M$-basis for $X$ and $B_0$ is an $N$-basis for $X$ with $B_0 \ss B$.  
\end{definition}

If $E(M) = E(N)$ and we do not specify the set $X$, then an \emph{$(N,M)$-basis pair} means an $(N,M)$-basis pair 
for the common ground set. 

\begin{lemma}
  Let $N$ be a quotient of a matroid $M$, and $X \ss E(M)$. Then every $N$-basis $I$ for $X$ 
  is contained in an $(N,M)$-basis pair for $X$, and every $M$-basis $J$ for $X$ is contained
  in an $(N,M)$-basis pair for $X$. 
\end{lemma}
\begin{proof}
  Let $I$ be an $N$-basis for $X$. Since $I$ is $M$-independent, it is contained in an $M$-basis $J$ for $X$, 
  and $(I,J)$ is as $(N,M)$-basis pair for $X$. 
  Similarly, let $J$ be an $M$-basis for $X$. Since $X \ss \cl_M(J) \ss \cl_N(J)$, there is an $N$-basis $I$ for $X$ with $I \ss J$, 
  giving the $(N,M)$-basis pair $(I,J)$. 
\end{proof}

We now prove a technical lemma relating basis pairs to cardinalities. 

\begin{lemma}\label{basispairtech}
  Let $N$ be a quotient of a matroid $M$, and let $I \ss X \ss E(M)$.
  If $(I_0,I)$ is an $(N,M)$-basis pair for $I$ with $I_0$ finite, and $(B_0,B)$ is an $(N,M)$-basis pair for $X$, 
  then $|I - I_0| \le |B - B_0|$. 
\end{lemma}
\begin{proof}
  Since $I$ is independent in $M$ and $X \ss \cl_M(I \cup B)$, there is a $M$-basis $B'$ for $X$ with $I \ss B' \ss I \cup B$. 
  By the choice of $B'$, we have 
  \begin{claim}\label{diffeq1}
    $B' - B = I - B$. 
  \end{claim}
  Since $I_0 \ss B'$ is independent in $N$, and $X \ss \cl_M(B') \ss \cl_N(B)$, there is an $N$-basis $B''$ for $X$
  satisfying $I_0 \ss B'' \ss B'$. 
  Since $B'' \cap I$ is an $N$-independent subset of $I$ containing the $N$-basis $I_0$ of $I$, we have $B'' \cap I = I_0$. 

  \begin{claim}
    $I_0 - B = B'' - B$. 
  \end{claim}
  \begin{subproof}
    The forwards inclusion follows from $I_0 \ss B''$. We get the reverse inclusion from 
    $B'' = B'' \cap B' \ss B'' \cap (I \cup B) = (B'' \cap I) \cup (B'' \cap B) \ss I_0 \cup B$. 
  \end{subproof}

  \begin{claim}
    $|B-B''| = |B-B_0| + |I_0 - B|$
  \end{claim}
  \begin{subproof}
    \begin{align*}
      |B-B''| &= |B_0 - B''| + |(B - B_0) - B''| \\
      &= |B'' - B_0| + |(B - B'') - B_0| \\ 
      &= |(B \cup B'') - B_0| \\ 
      &= |B - B_0| + |(B'' - B) - B_0| \\ 
      &= |B - B_0| + |(I_0 - B) - B_0|; 
    \end{align*}
    since $B_0 \ss B$, the result follows. 
  \end{subproof}
  \begin{claim}
    $|B \cap (B' - B'')| + |(I - I_0) - B| = |B - B_0|$. 
  \end{claim}
  \begin{subproof}
    By the last three claims, we have 
    \begin{align*}
      |B - B_0| + |I_0 - B| &= |B - B''| \\ 
      &= |B \cap (B' - B'')| + |B' - B| \\ 
      &= |B \cap (B' - B'')| + |I - B| \\ 
      &= |B \cap (B' - B'')| + |(I - I_0) - B| + |I_0 - B|, 
    \end{align*}
    and since $I_0 - B$ is finite, we can cancel the term $|I_0-B|$ additively.
  \end{subproof}
  Note that $I - I_0 = I - (B'' \cap I) = I - B'' \ss B' - B''$. 
  By the last claim, 
  \[|B - B_0| = |B \cap (B' - B'')| + |(I - I_0) - B| \ge |B \cap (I - I_0)| + |(I - I_0)-B|
    = |I - I_0|,\] as required.  
\end{proof}

\begin{lemma}\label{basispairmono}
  Let $N$ be a finitary quotient of a matroid $M$, 
  and let $X \ss Y \ss E(M)$. If $(I_0, I)$ is an $(N,M)$-basis pair for $X$, 
  and $(J_0, J)$ is an $(N,M)$ basis pair for $Y$, then $|I-I_0| \le |J - J_0|$. 
\end{lemma}
\begin{proof}
  Suppose not; then $I-I_0$ has a subset $S$ for which $\infty > |S| > |J-J_0|$.
  Since $N$ is finitary and $S \ss I \ss \cl_N(I_0)$,
  there is a finite subset $K_0$ of $I_0$ for which $S \ss \cl_N(K_0)$. 
  Let $K = K_0 \cup S$, noting that $K \ss I$ so $K$ is $M$-independent, 
  and that $K \ss \cl_N(K_0)$. 
  Therefore $(K_0, K)$ is an $(N,M)$-basis pair for $K_0 \cup S$. 

  Since $K \ss I \ss X \ss Y$ and $K_0$ is finite, Lemma~\ref{basispairtech} implies that 
  $|S| = |K - K_0| \le |J - J_0|$, contrary to the choice of $S$. 
\end{proof}

Applying the above lemma in both directions in the case where $X = Y$ immediately gives the following. 

\begin{corollary}\label{basispairstable}
  Let $N$ be a finitary quotient of a matroid $M$. 
  If $(I_0, I)$ and $(J_0,J)$ are $(N,M)$-basis pairs for a set $X$, 
  then $|I - I_0| = |J - J_0|$. 
\end{corollary}

Informally this is a version of relative rank for quotients;
it is stating that the `difference between the $M$-rank and the $N$-rank' of a set $X$ 
is a well-defined quantity whenever $N$ is a finitary quotient of a matroid $M$. 
We give this concept a name, using a minimum in the definition so that it makes sense even for non-finitary $N$.

\begin{definition}
  Let $N$ be a quotient of a matroid $M$, and let $X \ss E(M)$. The \emph{discrepancy} $\delta_{N,M}(X)$
  is the minimum value of $|I - I_0|$, taken over all $(N,M)$-basis pairs $(I_0,I)$ for $X$.
\end{definition}

We write $\delta(N,M)$ for $\delta_{N,M}(E(M))$.
Observe that $\delta_{N,M}(X) = \delta(N|X, M|X)$, since these two pairs of matroids have the same basis pairs for $X$. 



Note also that, if $N$ is a quotient of $M$, then $(B_0,B)$ is an $(N,M)$-basis pair
if and only if $(E-B, E-B_0)$ is an $(M^*,N^*)$-basis pair. Since $M^* \preceq N^*$, 
this gives:

\begin{lemma}\label{discdual}
  If $N$ is a quotient of a matroid $M$, then $\delta(N,M) = \delta(M^*,N^*)$. 
\end{lemma}

The following result shows that, if $N$ is a quotient of $M$ with 
$N$ finitary or $M$ cofinitary, then discrepancy behaves reasonably for all sets.

\begin{theorem}\label{discgood}
  Let $N$ be a quotient of a matroid $M$ on ground set $E$, 
  such that either $N$ is finitary or $M$ is cofinitary. Then 
  \begin{enumerate}[(i)]
    \item\label{discdefined} every $(N,M)$-basis pair $(I_0,I)$ for a set $X \ss E$ satisfies $|I - I_0| = \delta_{N,M}(X)$;
    \item\label{discmono} $\delta_{N,M}(X) \le \delta_{N,M} (Y)$ for all $X \ss Y \ss E$;
    \item\label{discfinite} $r_N(X) + \delta_{N,M}(X) = r_M(X)$ for every set $X \ss E$;
    \item\label{discnullity} $n_N(X) = n_M(X) + \delta_{N,M}(X)$ for every set $X \ss E$. 
  \end{enumerate}
\end{theorem}
\begin{proof}
  If $N$ is finitary, then (\ref{discdefined}) is immediate from Lemma~\ref{basispairstable}. 
  If $M$ is cofinitary, then let $(I_0,I)$ be an $(N,M)$-basis pair for $X$. 
  Then $(I_0,I)$ is also an $(N | X, M | X)$-basis pair for $X$, 
  so $(X-I, X-I_0)$ is an $((M|X)^*, (N|X)^*)$-basis pair for $X$. Since $(M|X)^*$ is finitary, 
  we have 
  \[\delta_{N,M}(X) = \delta(N|X,M|X) = \delta((M|X)^*, (N|X)^*) = |(X-I) - (X-I_0)| = |I-I_0|.\]

  If $N$ is finitary, then (\ref{discmono}) follows from Lemma~\ref{basispairmono} and the definition of discrepancy. 
  If $M$ is cofinitary, then $M^*$ is finitary, so $\delta_{N,M}(X) = \delta_{M^*,N^*}(X) \le \delta_{M^*,N^*}(Y) = \delta_{N,M}(Y)$. 

  Finally, let $(I_0,I)$ be an $(N,M)$-basis pair for a set $X \ss E$. 
  Then $r_M(X) = |I| = |I_0| + |I - I_0| = r_N(X) + \delta_{N,M}(X)$, 
  and $n_N(X) = |X - I_0| = |X - I| + |I-I_0| = n_M(X) + \delta_{N,M}(X)$. 
\end{proof}

The finitary/cofinitary hypotheses are not the broadest setting where this theorem holds. 
For example, if $N = P \con A$ is a projection of $M = P \del A$, then it is routine 
to show that $\delta_{N,M}(X) = r_P((X \cup A) | X)$, from which the properties of Theorem~\ref{discgood} above can be derived.
It seems like an interesting question to characterize the largest range of cases where 
discrepancy is well-behaved (i.e. satisfies (\ref{discdefined}) of Lemma~\ref{discgood}). 

Lemma~\ref{badgraphic} shows that it is \emph{not} enough to insist that $M$ is cofinitary and $N$ is finitary;
In the matroids $N$ and $M$ it gives, 
both $(B,B)$ and $(B_M,B_N)$ are $(M,N)$-basis pairs, 
so $\delta(M,N) = |B-B| = 0$, but $|B_M - B_N| = \infty$. 
Even if (\ref{discmono}) and (\ref{discfinite}) and (\ref{discnullity}) still happen to hold, 
this feels unimpressive in the context of~(\ref{discdefined}) failing.

Finally, we show that we can recover the conclusion from Lemma~\ref{commonbasis} if we know that $N$ is a finitary quotient of $M$. 

\begin{theorem}\label{finitarycommonbasis}
  Let $N$ be a quotient of a matroid $M$ such that $M$ and $N$ have a common basis. If $N$ is finitary or $M$ is cofinitary, 
  then $M = N$. 
\end{theorem}
\begin{proof}
  Let $B$ be a common basis of $M$ and $N$. Since $(B,B)$ is an $(N,M)$-basis pair, 
  we have $\delta(N,M) = 0$, so every $(N,M)$ basis pair $(B_0,B_1)$ satisfies $|B_1 - B_0| = 0$, so $B_0 = B_1$. 
  Each basis for $N$ or $M$ is contained in some basis pair, so it follows that $N$ and $M$ have the same bases. 
\end{proof}

\section{Quotients with finite discrepancy}

Theorem~\ref{finitarycommonbasis} can be alternatively phrased as saying that, 
if $N$ is a quotient of a matroid $M$ with $N$ finitary or $M$ cofinitary, and $\delta(N,M) = 0$, then $M = N$. 
In this section, we generalize this statement by showing that, if $\delta(N,M)$ is finite, then $N$ is a projection of $M$. 

\begin{lemma}\label{disccontract}
  Let $N$ be a finitary quotient of a matroid $M$, and let $X, Y \ss E(M)$. 
  Then $\delta_{N \con X, M \con X}(Y-X) + \delta_{N,M}(X) = \delta_{N,M}(X \cup Y)$. 
\end{lemma}
\begin{proof}
  Let $(I_0,I)$ be an $(N,M)$-basis pair for $X$. Let $J$ be an $M$-basis for $X \cup Y$ containing $I$, 
  and let $J_0$ be an $N$-basis for $I$ containing $I_0$. We have $\cl_N(J_0) = \cl_N(I) \supseteq \cl_M(I) \supseteq X \cup Y$, 
  so $(J_0, J)$ is an $(N,M)$ basis pair for $X \cup Y$. 

  By construction, the set $J_0 - I_0$ is $(N \con I_0)$-independent, so is $(N \con X)$-independent, 
  and $\cl_{N \con X}(J_0 - I_0) = \cl_N(J_0 \cup X) - X \supseteq Y - X$. 
  Similarly, the set $J - I$ is $(M \con I)$-independent, so is $(M \con X)$-independent, 
  and $\cl_{M \con X}(J - I) = \cl_M(J \cup X) - X \supseteq Y-X$. 
  It follows that $(J_0-I_0, J-I)$ is an $(N \con X, N \con X)$-basis pair for $Y-X$. 

  Since $N$ and $N \con X$ are finitary and basis pairs are nested, Lemma~\ref{discgood} gives
  \[\delta_{N \con X, M \con X}(Y-X) + \delta_{N,M}(X) = |(J - I) - (J_0 - I_0)| + |I - I_0| = |J - J_0| = \delta_{N,M}(X \cup Y),\]  
  as required. 
\end{proof}



Recall that a modular cut $\cF$ of $M$ gives rise to a single-element extension $M +_{\cF} e$ for each $e \notin E(M)$. 
We can contract such an $e$ to obtain a projection of $M$. Write $M \dcon \cF$ for the single-element projection $(M +_{\cF} e) \con e$. 
If $N$ is finitary, the flats $F$ of $M$ for which $N \dcon X = M \dcon X$ actually form a modular cut of $M$, 
and projecting through this cut maintains the property that $N$ is a quotient. 

\begin{lemma}\label{tightquot}
  Let $N$ be a finitary quotient of a matroid $M$, and
  $\cF$ be the family of flats $F$ of $M$ for which $N \dcon F = M \dcon F$. 
  Then $\cF$ is a modular cut of $M$, and moreover $N$ is a quotient of $M \dcon \cF$. 
\end{lemma}
\begin{proof}
  The collection $\cF$ is trivially closed under taking superflats,
  so to show it is a modular cut, 
  we need to show that $\cF$ is closed under taking intersections of modular pairs
  and modular chains. 
  
  Let $F_1, F_2$ be a modular pair in $M$, with mutual basis $B$. 
  By Lemma~\ref{blattice}, the set $B \cap F_1 \cap F_2$ is an $M$-basis for $F_1 \cap F_2$, 
  so $B - (F_1 \cap F_2)$ is a basis for $M \dcon (F_1 \cap F_2)$, 
  and therefore $B - (F_1 \cap F_2)$ is spanning in $N \dcon (F_1 \cap F_2)$. 
  For each $i \in \{1,2\}$, since $F_i \in \cF$ and the set $B \cap F_i$ is an $M$-basis for $F_i$, 
  the set $B - F_i$ is independent in $M \dcon F_i$. 

  Since $B-F_1$ is independent in $M \dcon F_1 = N \dcon F_1$, 
  the subset $I = B \cap (F_2 - F_1)$ is independent in $N \dcon (F_1 \cap F_2)$. 
  Now $B - F_2$ is independent in $M \dcon F_2$, and 
  \[M \dcon F_2 = N \dcon F_2 = N \dcon (F_1 \cap F_2) \dcon F_2,\]
  so $I \cup (B - F_2) = B - (F_1 \cap F_2)$ is independent in $N \dcon (F_1 \cap F_2)$. 
  Therefore $B - (F_1 \cap F_2)$ is a basis for both $M \dcon (F_1 \cap F_2)$ and
  $N \dcon (F_1 \cap F_2)$. The latter is a finitary quotient of the former, 
  so these two matroids are equal by Theorem~\ref{finitarycommonbasis}.
  This, in turn, implies that $F_1 \cap F_2 \in \cF$, as required. 

  Let $\cC \ss \cF$ be a chain having a mutual basis $B$, and let $F_0 = \bigcap \cC$. 
  We wish to prove that $N \dcon F_0 = M \dcon F_0$, and as before, it suffices to 
  show that $B-F_0$ is a basis for both. By Lemma~\ref{blattice}, 
  we know that $B \cap F_0$ is an $M$-basis for $F_0$, so $B-F_0$ is a basis for $M \dcon F_0$. 
  The set $B-F_0$ is spanning in $N \dcon F_0$ and therefore in $M \dcon F_0$, 
  so we need to show it is independent; let $C \ss B - F_0$ be a circuit of $N \dcon F_0$. 

  Since $N$ is finitary, $C$ is finite. Since $C$ is disjoint from $F_0 = \bigcap \cC$, 
  for each $e \in C$ there is some $F_e \in \cC$ with $e \notin \cC$. The family
  $\{F_e : e \in C\}$ is a finite subchain of $\cC$, so contains a minimal element $F$, 
  which is a set in $\cC$ containing $F_0$ and disjoint from $C$. 
  The set $B \cap F$ is a basis for $F$ in $M$, and $C \ss B - F$, 
  so $C$ is independent in $M \dcon F = N \dcon F$.
  On the other hand, $C$ is a circuit of $N \dcon F_0$, and $F_0 \ss F$, 
  so $C$ is dependent in $N \dcon F$, giving a contradiction.

  We now wish to show that $N \preceq M \dcon \cF$. Let $e \notin E(M)$ and $P = M +_{\cF} e$, 
  so we have $P \del e = M$ and $P \con e = M \dcon \cF$, and $\cF$ is precisely the collection
  of flats $F$ of $M$ for which $e \in \cl_P(F)$. 

  \begin{claim}\label{consame}
    Each $X \ss E$ satisfies $N \dcon X = M \dcon X$ if and only if $\cl_M(X) \in \cF$. 
  \end{claim}
  \begin{subproof}
    Since $X \ss \cl_M(X) \ss \cl_N(X)$, we have $N \dcon X = N \dcon \cl_M(X)$, 
    so $N \dcon X = M \dcon X$ if and only if $N \dcon \cl_M(X) = M \dcon \cl_M(X)$, 
    as required. 
  \end{subproof}

  By Lemma~\ref{quotequiv}, it suffices to show that $\cl_{M \dcon \cF}(I) \ss \cl_N(I)$ for 
  every independent set $I$ of $M \dcon \cF$,
  or equivalently that $\cl_P(I \cup \{e\}) - \{e\} \ss \cl_N(I)$ for all such $I$.  
  We may thus assume that there exists $f \ne e$ 
  for which $f \in \cl_P(I \cup \{e\})$ but $f \notin \cl_N(I)$. 
  
  The independence of $I$ in $M \dcon \cF$ gives that $I$ is independent in $M$, 
  and also that $\cl_M(I) \notin \cF$, so $e \notin \cl_P(I)$. 
  Therefore $I \cup \{e\}$ is independent in $P$. Since $f \notin \cl_N(I)$ 
  and $N \preceq M$, we also have $f \notin \cl_M(I)$, so $f \notin \cl_P(I)$.  
  Thus $e \in \cl_P(I \cup \{f\})$ by the closure exchange axiom,
  which implies that $\cl_M(I \cup \{f\}) \in \cF$, 
  and therefore $N \dcon (I \cup \{f\}) = M \dcon (I \cup \{f\})$ by \ref{consame}.

  Let $B$ be a basis for this matroid. 
  Since $f \notin \cl_N(I)$, we see that $f$ is a nonloop of $N \dcon I$, and therefore, also of $M \dcon I$. 
  It follows that $B \cup \{f\}$ is a basis for $N \dcon I$ and $M \dcon I$, so $M \dcon I = N \dcon I$ by Lemma~\ref{commonbasis}. 
  But we have established that $\cl_M(I) \notin \cF$, so this contradicts~\ref{consame}. 
\end{proof}

We can now prove Theorem~\ref{quotientisprojectintro}; the statement below 
is equivalent, because a basis $B$ of $N$ for which $M \con B$ has finite rank 
implies that $\delta(N,M) < \infty$. 

\begin{theorem}\label{quotientisproject}
  Let $N$ be a quotient of a matroid $M$. If $N$ is finitary or $M$ is cofinitary, and $\delta(N,M) < \infty$, 
  then $N$ is a projection of $M$. 
\end{theorem}
\begin{proof}
  By duality and Lemma~\ref{discdual},
  we may assume that $N$ is finitary. We prove this by induction on the value of $d = \delta(N,M)$. 
  If $d = 0$, then $N$ and $M$ have a common basis and therefore $N = M$ by Lemma~\ref{commonbasis}, 
  so suppose that $d > 0$ and that the result holds for matroids $M'$ where $\delta(N,M') < d$. 

  Let $\cF$ be the set of flats $F$ of $M$ for which $N \dcon F = M \dcon F$. 
  We know from Lemma~\ref{tightquot} that $\cF$ is a modular cut, and that $N \preceq M \dcon \cF$. 
  Let $P$ be a matroid so that $P \del e = M$ for some $e$, and $P \con e = M \dcon \cF$. 
  
  Let $(B_0,B)$ be an $(N, M \dcon \cF)$-basis pair. 
  Since $d \ne 0$, we have $M \ne M \dcon \cF$, so $e$ is not a loop or coloop of $P$. 
  Therefore $B \cup \{e\}$ is a basis for $P$. Since $e$ is not a coloop, there is some $f \ne B \cup \{e\}$ 
  such that $B \cup \{f\}$ is a basis for $P$ and also, therefore, for $M$. 
  Therefore $(B_0, B \cup \{f\})$ is an $(N,M)$-basis pair, which implies that $\delta(N,M) = \delta(N, M \dcon \cF) + 1$. 
  By induction, we see that $N$ is a projection of $M \dcon \cF$. Since $M \dcon \cF$ is a projection of $M$, 
  it follows that $N$ is a projection of $M$. 
\end{proof}


\section{Connectivity}

\subsection*{Set Families}
Recall that Definition~\ref{lcdef} defines $\sqcap_M(X,Y) = n_M(I \cup J) + |I \cap J|$ for 
bases $I,J$ for $X$ and $Y$, 
and that, for a partition $\ab{X_a : a \in A}$ of $E(M)$ with respective bases $\ab{I_a : a \in A}$, 
Definition~\ref{gencon} defines $\lambda_M\br{\ab{X_a}} = n_M(\cup_a I_a)$. 

As discussed in the introduction, it is possible to simultaneously generalize these two definitions
to a notion of `local connectivity' for arbitrary indexed collections $\ab{X_a}$. 
In this section, we take this approach, first defining this generalization 
and proving the standard lemmas, 
then deriving consequences for the local connectivity of pairs 
and the connectivity of partitions. 

If the $X_a$ are disjoint, we can reasonably define $\sqcap(\ab{X_a})$
by choosing bases $\ab{I_a}$ and taking $\sqcap_M(\ab{X_a}) = n_M(\cup_a I_a)$. 
However, this will not work if the $X_a$ are not disjoint; 
as is evident from that $|I \cap J|$ term in Definition~\ref{lcdef}, 
we have $\sqcap_M(I,J) \ne n_M(I \cup J)$. 

Luckily, this issue is fairly superficial; 
our definition amounts to sidestepping disjointness requirements by using parallel extensions.
Since it requires extending infinitely many elements infinitely many times, 
we take care to precisely say what this means by giving an auxiliary definition. 

\begin{definition}
Given a matroid $M$ and a function $f : S \to E(M)$, 
the \emph{preimage} $f^{-1}(M)$ is the matroid with ground set $S$, 
in which a set $I \ss S$ is independent if and only if $f$ is injective on $S$, and the image $f(I)$ 
is independent in $M$. 
\end{definition}

We omit the routine proof that this is a matroid
(which is formalized in [\ref{mathlib}] as \smalltt{Matroid.comap}). 
Informally, the preimage
is obtained by replacing each $e \in E(M)$ with a set $f^{-1}(e) \ss S$ of parallel elements (or loops, if $e$ is a loop). 
If $X \ss S$, and the restriction of $f$ to $X$ is injective,
then $f$ yields an isomorphism between the two matroids $f^{-1}(M) | X$ and $M | f(X)$. 
It is also easy to see that for all $X \ss S$, 
we have $\cl_{f^{-1}(M)}(X) = f^{-1}(\cl_M(f(X))$ 
and $r_{f^{-1}(M)}(X) = r_M(f(X))$. 
We also use the fact that $f^{-1}(M \dcon f(C)) = f^{-1}(M) \dcon C$ 
for all $C \ss S$; this is most easily proved by showing that the two matroids
have the same ground set and closure function. 

We define local connectivity for arbitrary (non-disjoint) collections of sets 
by mapping bases which may intersect in $M$ to disjoint bases in an appropriate preimage of $M$, 
then taking the nullity of the union in this preimage. 

\begin{definition}\label{mlcdef}
  Let $M$ be a matroid, and let $\ab{X_a : a \in A}$ be sets in $M$.
  Let $\pi \colon E(M) \times A \to E(M)$ be the projection map defined by $\pi(e, a) = e$. 
  Then the \emph{local connectivity} of $\ab{X_a}$ 
  is the expression $\sqcap_M(\ab{X_a}) = n_{\pi^{-1}(M)}\br{\cup_{a \in A} (I_a \times \{a\})}$, 
  where $I_a$ is a basis for $X_a$ in $M$ for each $a \in A$. 
\end{definition}


We first prove that this is well-defined (independent of the choice of the $I_a$). 
\begin{lemma}
  Let $\ab{X_a : a \in A}$ be an indexed collection of sets in a matroid $M$, and for each $a \in A$, let $I_a$ and $J_a$ 
  be bases for $X_a$ in $M$. Then 
  \[n_{\pi^{-1}(M)}\br{\cup_{a \in A} (I_a \times \{a\})} = n_{\pi^{-1}(M)}\br{\cup_{a \in A} (J_a \times \{a\})}.\]
\end{lemma}
\begin{proof}
  We prove this by using Lemma~\ref{nullityunion}(\ref{nueq}), so we have to show that for each $a \in A$, 
  the sets $I_a \times \{a\}$ and $J_a \times \{a\}$ have the same nullity and closure in $\pi^{-1}(M)$.  
  For each $a \in A$, the function $\pi$ is injective on $X_a \times \{a\}$, 
  so we have 
  \[n_{\pi^{-1}}(M)(I_a \times \{a\}) = n_M(\pi(I_a \times \{a\})) = n_M(I_a) = 0,\] 
  and similarly $n_M(J_a) = 0$. Moreover, 
  \[\cl_{\pi^{-1}(M)}(I_a \times \{a\}) = \pi^{-1}\br{\cl_M(\pi(I_a \times \{a\}))}= \pi^{-1}\br{\cl_M(I_a)} = \pi^{-1}\br{\cl_M(X_a)},\] 
  and the same holds for $J_a$, so the result follows. 
\end{proof}

We now give a few ways to interpret this definition in terms of the properties of the matroid $M$ itself, 
rather than its preimage.  
In particular, property (\ref{mlcdj}) below implies that $\sqcap_M(X_a : a \in A) = \lambda_M(X_a : a \in A)$ in the 
case where $X_a$ is a partition of $E(M)$, and (\ref{mlcpair}) gives that the indexed version of $\sqcap$ agrees 
with the version for pairs in Definition~\ref{lcdef}. Property (\ref{mlcrank}) shows that familiar formula 
for rank holds in the finite-rank case. 

\begin{lemma}\label{mlcprop}
  Let $\ab{X_a : a \in A}$ be sets in a matroid $M$, 
  and for each $a \in A$, let $I_a$ be a basis for $X_a$ in $M$. Then 
  \begin{enumerate}[(i)]
    \item\label{mlcndj} $\sqcap_M(\ab{X_a}) 
    = n_M\br{\cup_{a \in A}I_a} + \sum_{e \in \cup_a I_a} \br{\abs{\setof{a \in A}{e \in I_a}}-1}$;
    \item\label{mlcdj} if the sets $I_a$ are pairwise disjoint, 
    then $\sqcap_M(\ab{X_a}) = n_M(\cup_{a \in A}I_a)$;
    \item\label{mlcpair} 
      if $A = \{1,2\}$, then $\sqcap_M(\ab{X_a}) = n_M(I_1 \cup I_2) + |I_1 \cap I_2|$;
    \item\label{mlcrank} $\sqcap_M(\ab{X_a}) + r_M\br{\cup_{a \in A} X_a} = \sum_{a \in A} r_M(X_a)$.
    \item\label{mlcpartbasis} if $\ab{J_a}$ are pairwise disjoint sets so that $J_a \ss I_a$ for each $a$, 
      and $\cup_a J_a$ is a basis for $\cup_a I_a$, 
      then $\sqcap_M(\ab{X_a}) = \sum_{a \in A}|I_a - J_a|$. 
  \end{enumerate}
\end{lemma}
\begin{proof}
  Let $M' = \pi^{-1}(M)$ and Let $I' = \cup_a (I_a \times \{a\})$, 
  so $\sqcap_M(\ab{X_a}) = n_{M'}(I')$ by definition. 
  Let $\phi : \cup_a I_a \to A$ be a function such that $e \in I_{\phi(e)}$ 
  for all $e \in \cup_a I_a$. 
  Let $\wh{I} = \{(e, \phi (e)) : e \in \cup_a I_a\}$. We have 
  \[\wh{I} \ss I' \ss \cl_{M'}(I') = \pi^{-1}(\cl_M(\pi(I'))) = \pi^{-1}(\cl_M(\pi(\wh{I}))
  = \cl_{M'}(\wh{I}).\]
  
  Since $\pi$ is injective on $\wh{I}$ and $\pi(\wh{I}) = \cup_{a}I_a$, 
  the matroids $M' | \wh{I}$ and $M | \cup_a I_a$ are isomorphic;
  it follows that $n_{M'}(\wh{I}) = n_M(\cup_a I_a)$. 
  Lemma~\ref{nullityprop}(\ref{nullityssclosure}) thus gives
  \[\sqcap_M(\ab{X_a}) = n_{M'}(I') = n_{M'}(\wh{I}) + |I' - \wh{I}| = n_M(\cup_a I_a) + |I' - \wh{I}|.\]
  but now the choice of $I'$ and $\wh{I}$ yields
  \[\abs{I' - \wh{I}\,} = \abs{\setof{(e,a)}{e \in I_a, a \ne \phi(e)}} = \sum_{e \in \cup_a I_a} \br{\abs{\setof{a}{e \in I_a}} -1},\]
  which implies (\ref{mlcndj}). 

  Now (\ref{mlcdj}) follows immediately from (\ref{mlcndj}), since each summand on the 
  right is zero if the sets $I_a$ are disjoint. It is also clear that (\ref{mlcdj}) implies (\ref{mlcpair}), 
  since each nonzero summand is $1$, and there is one such summand for each element of $I_1 \cap I_2$. 
  
  For (\ref{mlcrank}), 
  note that the condition $r_M(\cup_a I_a) < \infty$ implies that each $I_a$ is finite. 
  We have $r_M(\cup_a X_a) = r_M(\cup_a I_a) = r_M(\pi(I')) = r_M'(I')$, 
  which gives that 
  \[\sqcap_M(\ab{X_a}) + r_M(\cup_a X_a) = n_{M'}(I') + r_{M'}(I') = |I'| = \sum_a |I_a| = \sum_{a}r_M(X_a),\]
  so (\ref{mlcrank}) holds.  

  Finally, consider sets $\ab{J_a}$ as in (\ref{mlcpartbasis}), and let $J' = \cup_a (J_a \times \{a\})$. 
  Since the $J_a$ are disjoint, and $\pi(J') = \cup_a J_a$, it is easy to see that $J'$ is a basis for $I'$ in $M'$, 
  so $\sqcap_{M}(\ab{I_a}) = n_{M'}(I') = \abs{I' - J'} = \sum_{a}|I_a - J_a|$, as required for (\ref{mlcpartbasis}). 
\end{proof}

\begin{corollary}\label{lcclosure}
  If $\ab{X_a : a \in A}$ is a collection of sets in a matroid $M$, 
  then $\sqcap_M(\ab{X_a}) = \sqcap_M(\ab{\cl_M(X_a)})$. 
\end{corollary}
\begin{proof}
  Let $I_a$ be a basis for each $X_a$. Since each $I_a$ is also a basis for $\cl_M(X_a)$,
  the result is immediate by applying Lemma~\ref{mlcprop}(\ref{mlcndj}). 
\end{proof}

We now prove that local connectivity is preserved under taking appropriate preimages.

\begin{lemma}\label{lcpreimage}
  Let $\ab{X_a : a \in A}$ be a collection of sets in a matroid $M$, 
  and $f : S \to E(M)$ be a function. If $\ab{Y_a : a \in A}$ are subsets of $S$
  so that $f(Y_a) = X_a$ for each $a \in A$, then $\sqcap_{f^{-1}(M)}\ab{Y_a} = \sqcap_M(\ab{X_a})$. 
\end{lemma}
\begin{proof}
  Let $\pi : E(M) \times A \to E(M)$ 
  be the map with $\pi(e,a) = e$, and $\psi : S \times A \to S$ be the map with $\psi(e,a) = e$.
  Let $\hat{f} : S \times A \to E \times A$ be the map with $\hat{f}(e,a) = (f(e),a)$. 
  Note that $f \circ \psi = \pi \circ \wh{f}$. 

  Let $I_a$ be a basis for $X_a$ in $M$ for each $a$, and let $J_a \ss Y_a$ be a set such that $f$ induces 
  a bijection between $I_a$ and $J_a$. It is clear that $J_a$ is a basis for $Y_a$ in $f^{-1}(M)$, 
  and also that $\psi^{-1}(J_a) = (\pi \circ \wh{f})^{-1}(I_a)$.
  Define matroids $M_1$ and $M_2$ by  
  $M_1 = (\psi^{-1} (f^{-1}(M))) | \cup_a \psi^{-1}(J_a)$ and $M_2 = \pi^{-1}(M) | \cup_a \pi^{-1}(I_a)$. 
  By the definition of nullity, we have $\sqcap_{f^{-1}(M)} \ab{Y_a} = r^*(M_1)$ and $\sqcap_M(\ab{X_a}) = r^*(M_2)$, 
  so it suffices to show that $\hat{f}$ induces an isomorphism between $M_1$ and $M_2$. 
  We omit the easy proof that $\hat{f}$ gives a bijection between their ground sets. 
  \begin{claim}
    For each set $X \ss E(M_1)$, the function $f \circ \psi$ is injective on $X$ if and only if $\pi$ is injective on $\hat{f}(X)$. 
  \end{claim}
  \begin{subproof}
    The restriction of $\wh{f}$ to $X$ is injective, so $\pi$ is injective on $\hat{f}(X)$ if and only if
    $\pi \circ \hat{f}$ is injective on $X$. But $f \circ \psi = \pi \circ \hat{f}$, and the claim follows. 
  \end{subproof}
  Each set $X \ss E(M_1)$ is independent if and only if $f \circ \psi$ is injective on $X$,
  and the set $(f \circ \psi)(X)$ is independent in $M$. Using the claim, the fact that $(f \circ \psi)(X) = \pi(f(X))$ 
  and the definition of independence in a preimage, it follows that $X$ is independent in $M_1$ if and only if $\hat{f}(X)$ 
  is independent in $M_2$, giving the required isomorphism. 
\end{proof}

Setting $f = \pi$ and $S = E(M) \times A$ gives the following corollary,
which implies that the local connectivity of an arbitrary collection 
can be realized as a local connectivity of a disjoint collection in an appropriate preimage. 
It can be used to reduce the proof of statements about $\sqcap$ to the disjoint case. 

\begin{corollary}\label{lctodj}
  If $\ab{X_a : a \in A}$ is a collection of sets in a matroid $M$, 
  and $\pi : E(M) \times A \to E(M)$ is the projection map, then 
  $\sqcap_{\pi^{-1}(M)}\br{\ab{X_a \times \{a\}}} = \sqcap_M(\ab{X_a})$. 
\end{corollary}
\begin{proof}
  Let $I_a$ be a basis for $X_a$ in $M$ for each $a \in A$. It is easy to see 
  that $I_a \times \{a\}$ is a basis for $X_a \times \{a\}$ in $\pi^{-1}(M)$ 
  for each $a \in A$, and the sets $I_a \times \{a\}$ are pairwise disjoint. 
  By Lemma~\ref{mlcprop}(\ref{mlcdj}) and the definition of $\sqcap$, we have 
  \[\sqcap_{\pi^{-1}(M)}(\ab{X_a \times \{a\}}) 
  = n_{\pi^{-1}M}(\cup_a (I_1 \times \{a\})) = \sqcap_M(\ab{X_a}),\] 
  as required. 
\end{proof}

The next lemma handles the interactions between local connectivity and deletion; 
in particular, removing a few elements from all the sets in a collection reduces
the local connectivity, but only by a bounded amount. 

\begin{lemma}\label{mcdel}
  Let $\ab{X_a : a \in A}$ be sets in a matroid $M$, and $D \ss E(M)$. Then 
  \begin{enumerate}[(i)]
    \item\label{mcrestr} if $X_a \ss D$ for all $a \in A$, then $\sqcap_{M | D}(\ab{X}) = \sqcap_M(\ab{X})$;
    \item\label{mcmono}
      $\sqcap_M(\ab{Y_a}) \le \sqcap_M(\ab{X_a})$ for all sets $\ab{Y_a}$ with $Y_a \ss X_a$ for all $a$;
    \item\label{mcdelbd} $\sqcap_{M}(\ab{X_a - D}) \le \sqcap_M(\ab{X_a}) \le \sqcap_{M}(\ab{X_a-D}) + \sum_{a \in A}|X_a \cap D|$;
    \item\label{mcdeldj} if $\ab{X_a}$ are pairwise disjoint, 
    then $\sqcap_{M}(\ab{X_a}) \le \sqcap_{M}(\ab{X_a - D}) + |D|$. 
  \end{enumerate}
\end{lemma}
\begin{proof}
  (\ref{mcrestr}) is immediate from the definition. 
  For (\ref{mcmono}), let $I_a$ be a basis for $X_a$ containing a basis $J_a$ for $Y_a$
  for each $a$. the monotonicity of nullity gives
  \begin{align*}
    \sqcap_M{\ab{Y_a}} = n_{\pi^{-1}M}(\cup_a (J_a \times \{a\})) 
    &\le n_{\pi^{-1}M}(\cup_a (I_a \times \{a\})) = \sqcap_M{\ab{X_a}}.
  \end{align*}
  as required. 
  
  We now show (\ref{mcdeldj}). For each $a \in A$, let $J_a$ be a basis for $X_a \del D$, 
  and let $I_a$ be a basis for $X_a$ containing $J_a$. Since $I_a - J_a \ss D$ for each $a$, we have 
  \[\sqcap_M(\ab{X_a}) = n_M(\cup_a I_a) \le n_M(\cup_a J_a) + \abs{\cup_a I_a - \cup_a J_a} \le n_M(\cup_a J_a) + |D|,\]
  and $\sqcap_{M \del D}(\ab{X_a - D}) = n_M(\cup_a J_a)$, so the bound follows. 
  
  The lower bound in (\ref{mcdelbd}) forllows from (\ref{mcmono}), 
  so it remains to show the upper bound. Let $M' = \pi^{-1}(M)$, 
  let $X_a' = X_a \times \{a\}$ for each $a \in A$, 
  and let $D' = \cup_{a \in A}((X_a \cap D) \times \{a\})$. 
  Since the $X_a'$ are disjoint, 
  Lemma~\ref{lctodj} and (\ref{mcdeldj}) gives that 
  \begin{align*}
      \sqcap_M(\ab{X_a}) &= \sqcap_{M'}(\ab{X_a'}) 
      \\ &\le \sqcap_{M'}(\ab{X_a' - D'}) + |D'| \\
      &= \sqcap_{M'}(\br{(X_a - D) \times \{a\}}) + \sum_{a \in A}|X_a \cap D|,
  \end{align*}
  and the bound follows by applying Lemma~\ref{lctodj} again. 
\end{proof}

\section{Connectivity and Projections}

We now treat the interactions between local connectivity, contractions and projections.
Our first relates nullity and local connectivity. 

\begin{lemma}\label{lcindep}
  If $X$ and $I$ are sets in a matroid $M$ with $I$ independent, then $\sqcap_M(X, I) = n_{M \dcon X}(I)$. 
\end{lemma}
\begin{proof}
  Let $J$ be a basis for $X$. Then, using Lemma~\ref{nullityproject}, we have
  $n_{M \dcon X}(I) = n_{M \dcon J}(I) = n_M(J \cup I) + |J \cap I| = \sqcap_{M}(X, I)$, as required. 
\end{proof}

Note that $\sqcap_M(X,Z) \le \sqcap_M(Y,Z)$ for all $X \ss Y$; 
our next lemma gives an explicit term for the slack in this inequality.
\begin{lemma}\label{lcpairslack}
  Let $X, Y$ and $Z$ be sets in a matroid $M$ with $X \ss Y$. Then
  $\sqcap_M(X,Z) + \sqcap_{M \dcon X}(Y, Z) = \sqcap_M(Y,Z)$. 
\end{lemma}
\begin{proof}
  Let $J$ be a basis for $Y$ containing a basis $I$ for $X$. Note that $J - I$ is a basis for $Y$ in $M \dcon I = M \dcon X$. 
  By Lemma~\ref{nullityproject} and Lemma~\ref{lcclosure}, we have
  \begin{align*}
    \sqcap_M(X,Z) + \sqcap_{M \dcon X}(Y, Z) 
    &= \sqcap_M(I, Z) + \sqcap_{M \dcon I}(J - I, Z) \\ 
    &= n_{M \dcon Z}(I) + n_{M \dcon I \dcon Z}(J - I) \\
    &= n_{(M \dcon Z) \dcon I}(J - I) + n_{M \dcon Z}(I) \\
    &= n_{M \dcon Z}((J - I) \cup I) + |I \cap (J - I)| \\
    &= n_{M \dcon Z}(J),
  \end{align*}
  and the result now follows from Lemma~\ref{lcindep}.   
\end{proof}

Using the simple facts that $X \ss X \cup Y$ and $\sqcap_{M \dcon X}(X \cup Y, Z) = \sqcap_{M}(Y, Z)$,
we have the following corollary. 
\begin{corollary}\label{lcunionslack}
  Let $X,Y$ and $Z$ be sets in a matroid $M$. Then $\sqcap_M(X \cup Y, Z) = \sqcap_M(X, Z) + \sqcap_{M \dcon X}(Y,Z)$. 
\end{corollary}

Our next lemma shows that the local connectivity of a pair is equal to the slack in the inequality $n_{M}(X) \le n_{M \dcon C}(X)$
of Lemma~\ref{nullityproject}(\ref{npge}). 
\begin{lemma}\label{nullitycontractadd}
  If $M$ is a matroid, then $n_{M \dcon C}(X) = n_M(X) + \sqcap_M(X,C)$ for all $C, X \ss E(M)$.
\end{lemma}
\begin{proof}
  Let $K$ be a basis for $X \cap C$, and $I$ and $J$ be bases for $C$ and $X$ respectively containing $K$. 
  Note that $I \cap X = K$ and that $I \cup X = (I \cup J) \cup (X - J)$ by the maximality of $K$. 
  Note that $X - (I \cup J) \ss \cl_M(J) \ss \cl_M(I \cup J)$;
  by Lemma~\ref{nullityproject}(\ref{npconind}) and Lemma~\ref{nullityprop}(\ref{nullityssclosure}), we have
  \begin{align*}
    n_{M \dcon C}(X) &= n_{M \dcon I}(X) \\
    &= n_M(I \cup X) + |I \cap X| \\
    &= (n_M(I \cup J) + |X - J|) + |K| \\
    &= n_M(I \cup J) + n_M(X) + |K| \\
    &= n_M(X) + \sqcap_M(X,C),
  \end{align*}
  as required. 
\end{proof}

We now prove a lemma about general local connectivity in a contraction, where the sets involved are disjoint and independent. 

\begin{lemma}\label{lcconweak}
  Let $J$ be an independent set in a matroid $M$, and $\ab{I_a : a \in A}$ be pairwise disjoint independent sets in $M$, 
  each disjoint from $J$. Let $t \in A$. Then 
  \[\sqcap_{M \dcon J}(\ab{I_a}) + \sum_{i \ne t} \sqcap_M(I_i, J) = 
  \sqcap_M(\ab{I_a}) + \sqcap_{M \dcon I_t}(\cup_{i \ne t} I_i, J)\]
\end{lemma}
\begin{proof}
  Let $I' = \cup_{i \ne t} I_i$.
  For each $a \in A$, let $K_a$ be a basis for $I_a$ in $M \dcon J$, so using Lemma~\ref{nullityproject}(\ref{npconind}) 
  and disjointness, we have $\sqcap_M(I_a,J) = n_M(I_a \cup J) = n_{M \dcon J}(I_a) = |I_a - K_a|$. 
  It also holds that $\sum_{i \ne t} \abs{I_i - K_i} = \abs{\cup_{i \ne t} (I_i - K_i)}$ 
  and $\cup_{i \ne t}(I_i - K_i) \ss \cl_{M \dcon J}(K_i)$,
  so using Lemma~\ref{nullityprop}(\ref{nullityssclosure}) and Lemma~\ref{nullitycontractadd}, we have
  \begin{align*}
    \sqcap_{M \dcon J}(\ab{I_a}) + \sum_{i \ne t} \sqcap_M(I_i, J)
    &= n_{M \dcon J}(\cup_i K_i) + \sum_{i \ne t} \abs{I_i - K_i} \\
    &= n_{M \dcon J}(\cup_i K_i \cup \cup_{i \ne t} (I_i - K_i)) \\ 
    &= n_{M \dcon J}(I' \cup K_t) \\
    &= n_{M \dcon J \dcon K_t}(I').
  \end{align*}
  By Lemma~\ref{nullityproject}(\ref{npconind}) and Lemma~\ref{nullitycontractadd}, we also have 
  \begin{align*}
    \sqcap_M(\ab{I_a}) + \sqcap_{M \dcon I_t}(I', J)
    &= n_M(\cup_i I_i) + \sqcap_{M \dcon I_t}(I', J) \\
    &= n_M(I_t \cup I') + \sqcap_{M \dcon I_t}(I', J) \\
    &= n_{M \dcon I_t}(I') + \sqcap_{M \dcon I_t}(I', J) \\
    &= n_{M \dcon I_t \dcon J}(I').
  \end{align*} 
  But $cl_{M \dcon J}(K_t) = \cl_{M \dcon J}(I_t)$, and so $M \dcon I_t \dcon J = M \dcon J \dcon I_t = M \dcon J \dcon K_t$;
  the lemma now follows from the two calculations above. 
\end{proof}

We now prove a strengthening of the previous lemma with the independence and disjointness assumptions dropped. 
This is done by taking bases and using an appropriate preimage to reduce to the problem to the independent, disjoint case. 

\begin{lemma}\label{mcprojindex}
  Let $C$ and $\ab{X_a : a \in A}$ be sets in a matroid $M$, and let $t \in A$. Then 
  \[\sqcap_{M \dcon C}(\ab{X_a}) + \sum_{i \ne t} \sqcap_M(X_i, C) = 
  \sqcap_M(\ab{X_a}) + \sqcap_{M \dcon X_t}(\cup_{i \ne t} X_i, C)\]
\end{lemma}
\begin{proof}
  Let $0$ represent an arbritary index outside $A$, and let $\pi : E(M) \times (A \cup \{0\}) \to E(M)$ be the map with $\pi(e,i) = e$.
  Let $J$ be a basis for $C$ in $M$, and let $I_a$ be a basis for $X_a$ in $M$ for each $a \in A$. Let $M' = \pi^{-1}(M)$. 
  Let $J' = J \times \{0\}$, and $I_a' = I_a \times \{a\}$ for each $a$; note that $\pi(J') = J$ and $\pi(I_a') = I_a$ for all $a$. 
  By Corollary~\ref{lcclosure}, and Lemma~\ref{lcpreimage}, we have 
  $\sqcap_{M \dcon C}(\ab{X_a}) = \sqcap_{M \dcon J}(\ab{I_a}) = \sqcap_{M' \dcon J'}(\ab{I_a'})$. 

  Using the same two results,
  we have $\sqcap_M(X_a, C) = \sqcap_{M'}(X_a', J')$ for all $a$ and $\sqcap_M(\ab{X_a}) = \sqcap_{M'}(\ab{I_a'})$, and 
  lastly $\sqcap_{M \dcon X_t}(\cup_{i \ne t}X_i, C) = \sqcap_{M' \dcon I_t'}(\cup_{i \ne t}I_i', J')$. It therefore
  suffices to show that 
  \[\sqcap_{M' \dcon J'}(\ab{I_a'}) + \sum_{i \ne t} \sqcap_{M'}(I_i', J') = 
  \sqcap_{M'}(\ab{I_a'}) + \sqcap_{M' \dcon I_t'}(\cup_{i \ne t} I_i', J'),\]
  which holds by Lemma~\ref{lcconweak}, since the sets in question are now independent and pairwise disjoint. 
\end{proof}

This shows that the relationship between $\sqcap_M(\ab{X_a})$ and $\sqcap_{M \dcon C}(\ab{X_a})$ 
is a little complicated, and is not monotone in either direction. 

The next corollary is a simplified version of the above which does not use a particular distinguished index. 

\begin{lemma}\label{mcprojsymm}
  Let $C$ and $\ab{X_a : a \in A}$ be sets in a matroid $M$. Then   
  \[\sqcap_{M \dcon C}(\ab{X_a}) + \sum_{a \in A}\sqcap_M(X_a, C) = \sqcap_M(\ab{X_a}) + \sqcap_M(\cup_a X_a, C).\]
\end{lemma}
\begin{proof}
  let $I$ be a basis for $C$. If $A$ is empty, then all terms are zero. Otherwise, let $t \in A$.
  Using Lemma~\ref{mcprojindex} and Lemma~\ref{lcclosure}, we have 
  \begin{align*}
    \sqcap_{M \dcon C}(\ab{X_a}) + \sum_{a \in A} \sqcap_M(X_a, C)
    &= \sqcap_{M \dcon I}(\ab{X_a}) + \sum_{a \in A} \sqcap_M(X_a, I) \\ 
    &= \sqcap_{M \dcon I}(\ab{X_a}) + \sum_{a \ne t} \sqcap_M(X_a, I) + \sqcap_M(X_t, I) \\
    &= \sqcap_M(\ab{X_a}) + \sqcap_{M \dcon X_t}(\cup_{i \ne t} X_i, I) + \sqcap_M(X_t, I) \\
    &= \sqcap_M(\ab{X_a}) + \sqcap_{M \dcon X_t}(\cup_{i \ne t} X_i, I) + \sqcap_M(X_t, I) \\
    &= \sqcap_M(\ab{X_a}) + \sqcap_M(\cup_i X_i, I),
  \end{align*}
  where the last line uses Lemma~\ref{lcunionslack}. 
\end{proof}

(The lemma above is not necessarily a `formula' for $\sqcap_{M \dcon C}$ in terms of
$\sqcap_M$, since the nullity terms and the summation can be infinite, 
even when both connectivity terms are finite. However, 
if $C$ and $A$ are finite sets, then this is not the case, 
and each connectivity term is can be derived from the other. )

If $C$ is independent, then Lemma~\ref{lcindep} gives that $\sqcap_M(X, C) = n_{M \dcon X}(C)$ for each $C$, 
so the above can be restated with more nullity terms:

\begin{corollary}\label{mcprojauxindep}
  Let $C$ be an independent set of a matroid $M$, and let $\ab{X_a : a \in A}$
  be sets in $M$. Then  
  \[\sqcap_{M \dcon C}(\ab{X_a}) + \sum_{a \in A}n_{M \dcon X_a}(C) = \sqcap_M(\ab{X_a}) + n_{M \dcon \cup_a X_a}(C).\]
\end{corollary}

We can now provide bounds on how much local connectivity can change under projection 
and contraction. The bounds depend on whether the removed elements belong to sets in the collection or not. 

\begin{lemma}\label{mccon}
  Let $\ab{X_a : a \in A}$ be sets in a matroid $M$ with $A$ nonempty. Then 
  \begin{enumerate}[(i)]
    \item\label{mcproj} $\sqcap_{M \con C}(\ab{X_a - C}) = \sqcap_{M \dcon C}(\ab{X_a})$ for all $C \ss E(M)$;
    \item\label{mcconbd} 
      $\sqcap_{M \dcon C}(\ab{X_a}) \le \sqcap_M(\ab{X_a}) \le \sqcap_{M \dcon C}(\ab{X_a}) + (|A| - 1)r_M(C)$ for all $C \ss \cup_a X_a$.
    \item\label{mcconub}
      $\sqcap_{M \dcon C}(\ab{X_a}) \le \sqcap_M(\ab{X_a}) + r_M(C)$ for all $C \ss E(M)$. 
  \end{enumerate}
  Moreover, if $r_M(C)$ and $\sqcap_M(\ab{X_a})$ are finite, 
  then equality in (\ref{mcconub}) holds if and only if $C \ss \cl_M(\cup_a X_a)$, 
  and $C$ is skew to each $X_a$ in $M$. 
\end{lemma}
\begin{proof}
  The fact that $M \con C$ and $M \dcon C$ have the same independent sets implies that so do 
  $\pi^{-1}(M \con C)$ and $\pi^{-1}(M \dcon C)$, which in turn gives (\ref{mcproj}).
 
  We now show (\ref{mcconbd}). 
  We may assume that $|A| \ge 2$, 
  since otherwise all the connectivity terms are zero 
  and the bound is trivial. 
  We can also assume that $C$ is independent, 
  since replacing $C$ with a basis for $C$
  affects none of the terms in the inequalities. 
  For each $a$, let $I_a$ be a basis for $X_a$ 
  containing $I_a \cap C$, and let $J_a$ be a basis for $I_a$ in $M \dcon C$, 
  noting that $J_a$ is disjoint from $C$. 
  Let $J_a' = J_a \times \{a\}$, let $I_a' = I_a \times \{a\}$, 
  and let $C' = C \times \{a_0\}$ for some arbitrary $a_0 \in A$. 
  Let $M' = \pi^{-1}(M)$. As before, we have $\pi^{-1}(M \dcon C) = M' \dcon C'$;
  by construction, the set $C'$ is indepedendent in $M'$ and disjoint from $\cup_a {J_a'}$. 
  Now Lemma~\ref{nullityprop}(\ref{npconind}) and the monotonicity
  of nullity give
  \begin{align*}
    \sqcap_{M \dcon C}(\ab{X_a}) &= n_{M' \dcon C'}(\cup_a J_a')
    = n_{M'}(\cup_a J_a' \cup C) \le n_{M'} (\cup_a I_a') = \sqcap_{M}(\ab{X_a}),
  \end{align*}
  so the lower bound holds. 

  Since $|A| \ge 2$, the upper bound is trivial if $C$ is infinite, 
  so we may assume that $|C| = r_M(C) < \infty$. 
  Lemma~\ref{mcprojauxindep} gives 
    $\sqcap_M(\ab{X_a}) + n_{M \dcon \cup_a X_a} (C) = \sqcap_{M \dcon C}(\ab{X_a}) + \sum_{a \in A}n_{M \dcon X_a}(C)$.
  Since $C$ contains only loops of $M \dcon \cup_a X_a$, the nullity term on 
  the left-hand side is equal to $|C|$, and the sum on the right-hand side is at most 
  $|A||C|$. We can cancel $|C|$ subtractively by finiteness, 
  and the upper bound in (\ref{mcconbd}) follows. 

  For (\ref{mcconub}), we can also assume that $C$ is independent; 
  since replacing $C$ with a basis for $C$ does not affect the bound, 
  nor the conditions in the equality characterization. 

  By Lemma~\ref{mcprojauxindep}, we have 
  \begin{align*}
    \sqcap_{M \dcon C}(\ab{X_a}) &\le \sqcap_{M \dcon C}(\ab{X_a}) + \sum_{a \in A}n_{M \dcon X_a}(C) \\ 
    &=  \sqcap_M(\ab{X_a}) + n_{M \dcon \cup_a X_a} (C) \\
    &\le \sqcap_M(\ab{X_a}) + |C|,
  \end{align*}  
  and since $|C| = r_M(C)$, the bound follows. 

  If $\sqcap_{M \dcon C}(\ab{X_a})$ and $|C|$ are finite, then 
  equality holds if and only if the summation is zero, and $C$ contains 
  only loops of $M \dcon \cup_a X_a$. 
  The latter condition is equivalent to the statement that $C \ss \cl_M(\cup_a X_a)$, 
  and the former to the condition that $n_{M \dcon X_a}(C) = 0$ for all $a$. 
  For each $a$, we have $n_{M \dcon X_a}(C) = 0$ if and only if $C$ is independent 
  in $M \dcon X_a$,
  which is equivalent to skewness of $C$ and $X_a$ by 
  Corollary~\ref{skewpairrestrict}(\ref{sprbasis}) with $I = X = C$.
\end{proof}

\begin{lemma}\label{lcskewiff}
  If $\ab{X_a : a \in A}$ are sets in a matroid $M$, then $\sqcap_M(\ab{X_a}) = 0$ 
  if and only if the $X_a$ are skew in $M$. 
\end{lemma}
\begin{proof}
  Let $I_a$ be a basis for $X_a$ for each $a \in A$.
  Then 
  \[\sqcap_M(X_a) = n_M(\cup_a I_a) + \sum_{e \in \cup_a I_a}\br{\abs{\setof{a \in A}{e \in I_a}} - 1}.\] 
  This is equal to zero precisely when the $I_a$ are disjoint with independent union;
  by Corollary~\ref{skewindep}, this is equivalent to the statement that the $I_a$ 
  are skew. By Lemma~\ref{clskew} and Lemma~\ref{skewmono}(\ref{skewsubsets}), 
  this is equivalent to skewness of the $X_a$.
\end{proof}

\subsection*{Pairs}

We now prove that the local connectivity parameter for pairs
(defined by $\sqcap_M(X,Y) = |I \cap J| + n_M(I \cup J)$ for bases $I,J$ for $X,Y$)
has the expected properties: it is monotone, unchanged by taking closures, 
and interacts with minors, $\lambda$ and the rank function as expected. 

Most of them are immediate from the previous section, 
since Lemma~\ref{mlcprop}(\ref{mlcpair}) shows that the parameter for 
pairs is a special case of the general one. 
We include the statements, which shed some of the notational baggage of indexing,
mostly for convenience in future work. 
The last property shows that $\sqcap$ can be defined by choosing only one basis, not two. 

\begin{lemma}\label{lcprop}
  Let $X$ and $Y$ be sets in a matroid $M$. Then
  \begin{enumerate}[(i)]
    \item\label{lcmono} $\sqcap_M(X_0,Y_0) \le \sqcap_M(X,Y)$ for all $X_0 \ss X$ and $Y_0 \ss Y$;
    \item\label{lccl} $\sqcap_M(\cl_M(X),\cl_M(Y)) = \sqcap_M(X,Y)$;
    \item\label{lclambda} if $(X,Y)$ is a partition of $E(M)$, then $\sqcap_M(X,Y) = \lambda_M(X)$;
    \item\label{lcdelete} If $D$ is disjoint from $X \cup Y$, then $\sqcap_{M \del D}(X,Y) = \sqcap_M(X,Y)$;
    \item\label{lcproject} If $C \ss E(M)$, then $\sqcap_{M \dcon C}(X,Y) = \sqcap_{M \con C}(X-C,Y-C)$;
    \item\label{lcrank} $r_M(X) + r_M(Y) = r_M(X \cup Y) + \sqcap_M(X,Y)$;
    \item\label{lcbase} if $J$ is a basis for $Y$ in $M$, then $\sqcap_M(X,Y) = n_{M \dcon X}(J)$. 
    \item\label{lcskew} $\sqcap_M(X,Y) = 0$ if and only if $X$ and $Y$ are skew in $M$. 
  \end{enumerate}
\end{lemma}
\begin{proof}
  Recall that by Lemma~\ref{mlcprop}(\ref{mlcpair}), the parameter $\sqcap_M(X,Y)$ 
  agrees with its generalization to arbitrary indexed collections of sets;
  all statements above except (\ref{lclambda}) and (\ref{lcbase}) are
  therefore easy consequences of the lemmas earlier in this section. 

  Property (\ref{lclambda}) follows from the definition of $\lambda$ in (\ref{bwdef}), 
  since $\lambda_M(X) = r_{M^*}(I \cup J) = n_M(I \cup J) = \sqcap_M(X, E-X)$
  for bases $I$ and $J$ for $X$ and $E-X$ respectively. 

  For (\ref{lcbase}), if we let $I$ be a basis for $X$ in $M$, then 
  \[\sqcap_M(X,Y)  = |I \cap J| + n_M(I \cup J) = n_{M \dcon I}(J) = n_{M \dcon X}(J), \]
  where we use Lemma~\ref{nullityproject} and the fact that $M \dcon I = M \dcon X$. 
\end{proof}

\begin{theorem}\label{lcmod}
  If $X, Y$ are sets in a matroid $M$, then $\sqcap_M(X,Y) \ge r_M(X \cap Y)$.
  If $\sqcap_M(X,Y) < \infty$, then equality holds if and only if $(X,Y)$ is a modular pair.
\end{theorem}
\begin{proof}
  Let $I_0$ be a basis for $X \cap Y$, and $I_X, I_Y$ be bases for $X$ and $Y$ respectively, 
  with $I_0 \ss I_X \cap I_Y$. The maximality of $I_0$ gives that $I_0 = I_X \cap I_Y$. 
  Now 
  \[\sqcap_M(X,Y) = |I_X \cap I_Y| + n_M(I_X \cup I_Y) \ge |I_0| + 0 = r_M(X \cap Y),\]
  giving the inequality. 
  
  If $\sqcap_M(X,Y) < \infty$ and equality holds, then $n_M(I_X \cup I_Y) = 0$, 
  so the set $I_X \cup I_Y$ is independent in $M$. 
  It follows that $I_X \cup I_Y$ is a mutual basis for $X$ and $Y$, 
  so $(X,Y)$ is a modular pair. 

  Conversely, if $X$ and $Y$ are modular, and $B$ is a mutual basis for $X$ and $Y$, 
  then with the choices $(I_0,I_X,I_Y) = (B \cap X \cap Y, B \cap X, B \cap Y)$ 
  above, we have that $I_X \cup I_Y$ is independent, so equality holds. 
\end{proof}

\subsection*{Partitions}

Finally, we prove some statements about the connectivity parameter $\lambda$, 
which is just the specialization of $\sqcap$ to partitions. 
(When we write $\lambda_M(X)$ for a set $X$, this is simply shorthand for 
$\sqcap_M(X, E-X)$, so it is immediate that $\lambda_M(X) = \lambda_M(E - X)$.)

We start with Theorem~\ref{shortconndef}, giving an expression for infinite matroid connectivity that is 
simpler to work with than Definition~\ref{bwdef}. 

\begin{theorem}\label{conneqrelrank}
  Let $M$ be a matroid, and $I$ be a basis for a set $X$ in $M$. Then $\lambda_M(X) = \relrank{M^*}{X}{X-I}$. 
\end{theorem}
\begin{proof}
  Let $E = E(M)$, and $J$ be a basis for $E-X$ in $M$, so $\lambda_M(X) = r^*(M | (I \cup J))$ by definition. 
  Using the definitions of relative rank and nullity, we have 
  \begin{align*}
    \relrank{M^*}{X}{X-I} &= r(M^* \del (E-X) \con (X-I)) \\ 
    &= r^*(M \con (E-X) | I) \\
    &= n_{M \con (E-X)}(I).
  \end{align*}
  Since $\cl_M(J) = \cl_M(E-X)$ and $I$ is disjoint from $E-X$, 
  Lemma~\ref{nullityproject} gives $n_{M \con (E-X)}(I) = n_{M \dcon J}(I) = n_M(I \cup J) = \lambda_M(X),$ as required.
\end{proof}

This gives a shorter proof of the Lemma from [\ref{bw12}] that $\lambda$ is self-dual. 

\begin{lemma}\label{connselfdual}
  If $M$ is a matroid and $X \ss E(M)$, then $\lambda_M(X) = \lambda_M^*(X)$. 
\end{lemma}
\begin{proof}
  Let $E = E(M)$, and $J$ be a basis for $M \con (E-X)$. Then $J$ is independent in $M$; let $I$ be a basis for $X$ in $M$ 
  with $I \sps J$. Note that $X-J$ is a basis for $(M \con (E-X))^* = M^* | X$, so by Lemma~\ref{conneqrelrank} and the choice of $I$, 
  we have $\lambda_M^*(X) = \relrank{M}{X}{J} = \relrank{M}{I}{J} = |I-J|$. 

  Since $X-J$ is a basis for $M^* | X$, and $X-I \ss X- J$, 
  Lemma~\ref{conneqrelrank} also gives $\lambda_M(X) = \relrank{M^*}{X}{X-I} = \relrank{M^*}{X-J}{X-I} = |X-J| - |X-I| = |I-J|$, 
  which completes the proof. 
\end{proof}

Lemma~\ref{mlcprop}(\ref{mlcrank}) implies that
$\lambda_M(\ab{X_a})$ and $\lambda_M^*(\ab{X_a})$ 
are in general different for partitions $\ab{X_a : a \in A}$ with three or more parts, 
but they are both zero under the same circumstances: when the parts are skew. 

\begin{lemma}\label{lambdazero}
  If $M$ is a matroid and $\ab{X_a : a \in A}$ is a partition of $E(M)$, then the following are equivalent:
  \begin{enumerate}[(1)]
    \item\label{lzsk} $\ab{X_a}$ is skew in $M$;
    \item\label{lzskdual} $\ab{X_a}$ is skew in $M^*$;
    \item\label{lzsum} $M = \oplus_{a \in A} (M | X_a)$;
    \item\label{lz} $\lambda_M(\ab{X_a}) = 0$;
    \item\label{lzdual} $\lambda_M^*(\ab{X_a}) = 0$.
  \end{enumerate}
\end{lemma}
\begin{proof}
  If (\ref{lzsum}) holds, then $M^* = \oplus_X(M | X)^*$ and for each $X \in \cX$ we have $(M | X)^* = M^* \con (\cup \cX - X) = M^*|X$, 
  so in fact that (\ref{lzsum}) also holds for $M^*$.
  
  By Lemma~\ref{skewgoodbase}, we know that (\ref{lzsk}) implies (\ref{lzsum}), and the fact that (\ref{lzsum}) implies (\ref{lzsk})
  is clear. Therefore (\ref{lzsk}) is equivalent to both (\ref{lzsum}) 
  and the dual of (\ref{lzsum}); thus (\ref{lzsk}) is equivalent to (\ref{lzskdual}). 
  The result now follows easily from Lemma~\ref{lcskewiff}. 
\end{proof}

\subsection*{Modular Cuts}

Finally, we deal with projections by modular cuts. 
A modular cut $\cF$ of $M$ is \emph{proper} if it is not the collection of
all flats of $M$, or equivalently if $\cl_M(\es) \notin \cF$.  
The first lemma shows that the difference between $\sqcap_M(\ab{X_a})$ and $\sqcap_{M \dcon \cF}(\ab{X_a})$ 
is determined only by which of the flats spanned by the individiaul $X_a$ and the union of the $X_a$ belong to $\cF$. 

\begin{lemma}\label{modcutcon}
  Let $\cF$ be a proper modular cut in a matroid $M$, and $\ab{X_a : a \in A}$ be sets in $M$. 
  Then 
  \[\sqcap_M(\ab{X_a}) = \sqcap_{M \dcon \cF}(\ab{X_a}) + \abs{\setof{a \in A}{\cl_M(X_a) \in \cF}} - i,\] 
  where $i = 1$ if $\cl_M(\cup_a X_a) \in \cF$, and $i = 0$ otherwise. 
\end{lemma}
\begin{proof}
  Let $P$ be a matroid such that $P \del e = M$ and $P \con e = M \dcon \cF$ 
  for some $e \in E(P)$. Since $\cF$ is proper, the element $e$ is not a loop of $P$. 
  It follows that, for each set $Y \ss E(M)$, we have $n_{P \dcon Y}(\{e\}) = 1$ 
  if and only if $e \in \cl_P(Y)$, which in turn holds if and only if $\cl_M(Y) \in \cF$. 
  By Lemma~\ref{mcprojauxindep}, we have 
  \begin{align*}
    \sqcap_M(\ab{X_a}) + i &= \sqcap_P(\ab{X_a}) + n_{M \dcon \cup_a X_a}(\{e\}) \\ 
    &= \sqcap_{P \dcon \{e\}}(\ab{X_a}) + \sum_{a \in A}n_{M \dcon X_a}(\{e\}) \\ 
    &= \sqcap_{M \dcon \cF}(\ab{X_a}) + \abs{\setof{a \in A}{\cl_M(X_a) \in \cF}},
  \end{align*}
  and the result follows because $i < \infty$. 
\end{proof}

The dual version of this statement gives a formula for the dual connectivity in a projected modular cut. 

\begin{lemma}\label{dualmodcutcon}
  Let $\cF$ be a nonempty modular cut in a matroid $M$, and $\ab{X_a : a \in A}$ be subsets of $E = E(M)$. 
  Then 
  \[\sqcap^*_{M \dcon \cF}(\ab{X_a}) = \sqcap_{M}^*(\ab{X_a}) + \abs{\setof{a \in A}{\cl_M(E(M) - X_a) \notin \cF}} - i,\] 
  where $i = 0$ if $\cl_M(E - \cup_a X_a) \in \cF$, and $i = 1$ otherwise. 
\end{lemma}
\begin{proof}
  Let $P$ be a matroid such that $P \del e = M$ and $P \con e = M \dcon \cF$ 
  for some $e \in E(P)$. 

  For each set $Y \ss E$, since $e$ is a nonloop of $P^*$ and $e \notin Y$, 
  we have $n_{P^* \dcon Y}(\{e\}) = r_{P \con (E - Y)}(\{e\})$ by the definition of nullity, 
  so $n_{P^* \dcon Y}(\{e\}) = 1$ if and only if $e \notin \cl_P(E - Y)$, 
  which by the choice of $P$ and $e$ holds if and only if $\cl_M(E - Y) \notin \cF$. 
  Using Lemma~\ref{mcprojauxindep} now gives 
  \begin{align*}
    \sqcap^*_M(\ab{X_a}) + \abs{\setof{a \in A}{\cl_M(E - X) \notin \cF}} 
      &= \sqcap_{P^* \con e}(\ab{X_a}) + \sum_{a \in A}n_{P^* \dcon X_a}(\{e\}) \\ 
      &= \sqcap_{P^*}(\ab{X_a}) + n_{P^* \dcon \cup_a X_a}(\{e\}) \\ 
      &= \sqcap_{M \dcon \cF}^*(\ab{X_a}) + i,
  \end{align*}
  where we use  Lemma~\ref{mccon}(\ref{mcproj}) and the fact that $P^* \del e = M \dcon \cF$. 
  Again, the result follows because $i < \infty$. 
\end{proof}

This last lemma is a simpler version of Lemma~\ref{modcutcon} that describes
the interaction of the single-set-valued $\lambda$ with projection by a modular cuts. 
One can also use Lemma~\ref{mcprojauxindep} to prove a similar upper bound 
of $\lambda_{M}(X) + 1 \ge \lambda_{M \dcon \cF}(X)$, but we will not need this. 

\begin{lemma}\label{modcutlambda}
  If $\cF$ is a proper modular cut of a matroid $M$, and $X \ss E(M)$, then 
  $\lambda_M(X) \le \lambda_{M \dcon \cF}(X) + 1$. 
  If $\lambda_M(X) \ne \infty$, then equality holds 
  if and only if $\cF$ contains both $\cl_M(X)$ and $\cl_M(E-X)$. 
\end{lemma}
\begin{proof}
  By the definition of $\lambda$, we have $\lambda_M(X) = \lambda_M^*(X) =  \sqcap_M^*(\ab{X,E(M)-X})$, 
  and similarly $\lambda_{M \dcon \cF}(X) = \sqcap_{M \dcon \cF}^*(\ab{X,E(M)-X})$. 
  The result is now immediate from Lemma~\ref{dualmodcutcon}. 
\end{proof}

\section{Guts Extensions}

Given sets $\ab{X_a : a \in A}$ in a matroid $M$, let $\gutscut{M}{\cX}$ denote the collection of flats $F$ of $M$ 
for which $\ab{X_a}$ is skew in $M \dcon F$. 
(Since loops do not alter skewness, this is equal to the set of flats $F$
for which $\ab{X_a-F}$ is skew in $M \con F$, but we prefer the version that
does not modify the ground set.)
Note that $E(M) \in \gutscut{M}{\cX}$ for any $\cX$, so $\gutscut{M}{\cX}$ is always nonempty. 

We now prove Theorem~\ref{gutscutintro}, which states that if $\cup_a X_a = E(M)$, 
then $\gutscut{M}{\ab{X_a}}$ 
is a modular cut. An extension using this modular cut is the unique freeest possible extension 
in which all the $X_a$ span the new element. 

\begin{theorem}\label{gutscut}
  If $\ab{X_a : a \in A}$ are sets in a matroid $M$ with $\cup_a X_a = E(M)$, then $\gutscut{M}{\ab{X_a}}$ is a modular cut of $M$. 
\end{theorem}
\begin{proof}
  Let $E = E(M)$. 
  If $F \in \cF$ and $F'$ is a flat of $M$ containing $F$, 
  then the collection $\ab{X_a}$ is skew in $M \dcon F$ and therefore in $M \dcon F' = (M \dcon F) \dcon F'$ by
  Lemma~\ref{skewmono}, so $F' \in \cF$. 

  It is therefore enough to show that $\cF$ is closed under taking intersections of modular families. 
  Let $B$ be a basis of $M$, and let $\cF_0 \ss \cF \cap \cL_M(B)$, so we want to prove that $\bigcap \cF_0 \in \cF$. 
  Let $F_0 = \bigcap \cF_0$. 

   \begin{claim}\label{gc1}
    For each $F \in \cF_0$ and $a \in A$, we have $X_a \ss \cl_M(B \cap (X_a \cup F))$. 
  \end{claim}
  \begin{subproof}
    Since $E-X_a \ss \cup_{b \in A - \{a\}}X_b$ is skew to $X_a$ in $M \dcon F$ and is disjoint from $X_a$,
    Lemma~\ref{skewgoodbase} gives that $(M \dcon F) | X_a = (M \dcon F) \con (E-X_a)$. 
    Using the fact that $B \cap F$ is a basis for $F$ in $M$, we have 
    \begin{align*}
      X_a \cap \cl_M(B \cap (X_a \cup F)) &= \cl_{M \dcon (B \cap F) | X_a}(B \cap X_a) \\ 
      &= \cl_{M \dcon F | X_i}(B \cap X_a) \\
      &= \cl_{M \dcon F \con (E - X_a)}(B \cap X_a) \\
      &= \cl_M((B \cap X_a) \cup (E-X_a)) - (E - X_a) \\
      &\supseteq \cl_M(B) - (E-X_a) = X_a,
    \end{align*}
    which implies the claim. 
  \end{subproof}
    \begin{claim}\label{gc2}
    $X_a \ss \cl_{M \dcon F_0}((B - F_0) \cap X_a)$ for each $a \in A$.
  \end{claim}
  \begin{subproof}
  By~\ref{gc1} and Lemma~\ref{closureinter}, we have 
  \[X_a \ss \bigcap_{F \in \cF_0} \cl_M\br{B \cap (X_a \cup F)} = \cl_M\br{\bigcap_{F \in \cF_0} B \cap (X_a \cup F)} = 
  \cl_M(B \cap \br{X_a \cup F_0}).\]
  Moreover, since $F_0 \in \cL_M(B)$, we have 
  \[X_a \ss \cl_M(B \cap (X_a \cup F_0)) = \cl_M((B \cap X_a) \cup (B \cap F_0)) = \cl_M((B \cap X_a) \cup F_0).\] 
  Therefore $X_a \ss \cl_{M \dcon F_0}(B \cap X_a) = \cl_{M \dcon F_0}((B-F_0) \cap X_a)$, as required. 
  \end{subproof}
  Since $F_0 \in \cL_M(B)$, the set $B \cap F_0$ is a basis for $F_0$ in $M$, so $B - F_0$ is a basis for $M \dcon F_0$. 
  By~\ref{gc2}, we have $X_a \in \cL_{M \dcon F_0}(B - F_0)$ for each $a \in A$, so in fact the set $B - F_0$ is a mutual 
  basis for $\ab{X_a}$ in $M \dcon F_0$. Thus $F_0 \in \cF_0$, as required. 
\end{proof}

The requirement that $\cup_a X_a = E(M)$ cannot be removed, even if $|A| = 2$ and $M$ is finite. 
Such difficulties are generally well-studied (see, for example, [\ref{o19}]), 
but we give a self-contained counterexample here. 
The \emph{Vamos matroid} is the rank-four matroid $V$ with ground set 
$S_1 \cup S_2 \cup S_3 \cup S_4$,
where $S_1,S_2,S_3,S_4$ are pairwise disjoint two-element sets, and the nonspanning circuits 
are the sets of the form $S_i \cup S_j$, where $1 \le i < j \le 4$ and $(i,j) \ne (3,4)$. 
\begin{lemma}
  If $\cX = \{S_1,S_2\}$, then $\gutscut{V}{\cX}$ is not a modular cut in $V$. 
\end{lemma}
\begin{proof}
  Let $k \in \{3,4\}$. 
  Since $S_1 \cup S_k$ and $S_2 \cup S_k$ are circuits, we have $r_{M \dcon S_k}(S_1) = r_{M \dcon S_k}(S_2) = 1$ 
  and $r_{M \dcon S_k}(S_1 \cup S_2) = 4 - 2 = 2$, so $\cX$ is skew in $M \dcon S_k$ and therefore 
  $S_k = \cl_V(S_k) \in \gutscut{V}{\cX}$. 
  Since $S_3 \cup S_4$ is independent, the pair $(S_3,S_4)$ is modular in $V$, so if $\gutscut{V}{\cX}$ is a modular cut, 
  we have $\es = S_3 \cap S_4 \in \gutscut{V}{\cX}$. But $\cX$ is not skew in $V$, so this is a contradiction. 
\end{proof}

The next lemma shows that, for a partition $\cX$, applying a guts projection 
reduces the value of $\lambda^*_M$ by exactly one, provided that $\cX$ is not already skew. 

\begin{lemma}\label{gpsub}
  Let $\ab{X_a : a \in A}$ be a partition of the ground set of a matroid $M$.
  If $\ab{X_a}$ is not skew in $M$, then 
  $\lambda_{M}^*(\ab{X_a}) = \lambda_{M \dcon \gutscut{\cX}{M}}^*(\ab{X_a}) + 1$. 
\end{lemma}
\begin{proof}
  By Lemma~\ref{dualmodcutcon}, it suffices to show that the modular cut $\cF = \gutscut{\cX}{M}$ is proper,
  nonempty, and that $\cl_M(E - X_a) \in \cF$ for all $a \in A$.
  Since $\ab{X_a}$ is not skew in $M$, we have $\cl_M(\es) \notin \cF$, 
  so $\cF$ is a proper modular cut. Since $E(M) \in \cF$, 
  we know that $\cF \ne \es$.

  For each $a \in A$ and each $b \ne a$, the set $X_b$ contains only loops in 
  $M \dcon \cl_M(E - X_a)$; it is therefore easy to see by Theorem~\ref{skewequiv}
  that $\ab{X_a}$ is skew in $M \dcon \cl_M(E-X_a)$ and therefore that 
  $\cl_M(E-X_a) \in \cF$, as required. 
\end{proof}

Recall that $\gutsproj{M}{\cX}{k}$ is the matroid obtained from $M$ by iterating 
the operation $M \mapsto M \dcon \gutscut{\cX}{M}$ a total of $k$ times. 
Theorem~\ref{lambdadualeq}, which we restate here, is an immediate consequence 
of the lemma above. 

\begin{theorem}
  If $\ab{X_a : a \in A}$ is a partition of the ground set of a matroid $M$, 
  then $\lambda_M^*(\ab{X_a})$ is the minimum $k \in \bN$ 
  such that the matroid $M_0 = \gutsproj{M}{\ab{X_a}}{k}$ satisfies $\lambda_{M_0}(\ab{X_a}) = 0$, 
  or is $\infty$ if no such $k$ exists. 
\end{theorem}

\section{Minors and Majors}

Let $M'$ be an extension of $M$ by an element $e$ with corresponding modular cut $\wh{\cF}$. 
If $\cF$ is the modular cut of $M \con C$ corresponding to its extension $M' \con C$, 
then it is routine to argue that $F \cup C \in \wh{\cF}$ for all $F \in \cF$; 
this holds because any flat of $M \con C$ that spans $e$ in $M' \con C$ must give
a flat of $M$ that spans $e$ in $M'$. 

The following theorem can be seen as a sort of reversal of this observation.
It gives a way to pass from a modular cut $\cF$ describing an extension of $M \con C$ 
to a certain modular cut $\wh{\cF}$ describing an extension of $M$ that commutes with the contraction. 
There are potentially many other such extensions of $M$, but the one constructed here is the freeest; 
the only flats of $M$ that span $e$ in the extension are the ones that are forced to by the above argument. 

\begin{theorem}\label{contractmodularcut}
  Let $C$ be a set in a matroid $M$, and $\cF$ be a modular cut of $M \con C$. 
  Then $\wh{\cF} = \{F \cup C : F \in \cF\}$ is a modular cut in $M$. 

  Moreover, if $e \notin E(M)$, and $M'$ is the extension of $M$ by $e$ corresponding to $\wh{\cF}$, 
  then $M' \con C$ is the extension of $M \con C$ by $e$ corresponding to $\cF$. 
\end{theorem}
\begin{proof}
  It is immediate that $\wh{\cF}$ only contains flats, and is closed under taking superflats; it remains 
  to show that $\wh{\cF}$ is closed under taking intersections of modular families. 
  Let $\wh{\cF_0} \ss \wh{\cF}$ be such a modular family, and $\cF_0 = \{F - C : F \in \cF_0\}$ 
  the corresponding subfamily of $\cF$. 
  it suffices to show that $\cF_0$ is modular in $N$, 
  since this will imply that $\cF_0 \in \cF$ and therefore that $\cap \wh{\cF_0} = (\cap \cF_0) \cup C \in \wh{\cF}$. 

  Since $C \ss F$ for all $F \in \wh{\cF_0}$, the family $\{C\} \cup \wh{\cF_0}$ is modular in $M$ by Lemma~\ref{modularinsert}. 
  Let $B$ be a mutual basis in $M$ for $\{C\} \cup \wh{\cF_0}$. 
  For each flat $F \in \cF_0$, the set $B \cap (F \cup C)$ is a basis for $F \cup C$ in $M$, 
  that contains the basis $B \cap C$ for $C$ in $M$.
  Since $F$ and $C$ are disjoint, it follows that $B \cap F = (B - C) \cap F$ is a basis for $F$ in $M \con C$. 
  Since $B \cap C$ is a basis for $C$ and $B$ is independent, the set $B-C$ is also independent in $M \con C$,
  and so $B - C$ is a mutual basis for $\cF_0$ in $M \con C$, giving the required modularity. 

  For the second part, let $M_0'$ be the extension of $M \con C$ corresponding to $\cF$. 
  We have $(M' \con C) \del e = (M' \del e) \con C = M \con C = M_0 \del e$, so $M_0$ and $M' \con C$ 
  agree when $e$ is deleted. It therefore suffices to show that, for every flat $F$ of $M \con C$, 
  we have $e \in \cl_{M' \con C}(F)$ if and only if $e \in \cl_{M_0}(F)$. 

  Indeed, by the choice of $M_0$, we have $e \in \cl_{M_0}(F)$ if and only if $F \in \cF$, 
  and $e \in \cl_{M' \con C}(F)$ if and only if $e \in \cl_{M'}(F \cup C)$, which holds 
  when $F \cup C \in \wh{\cF}$. The result follows by the definition of $\wh{\cF}$. 
\end{proof}

\begin{theorem}\label{majorofminor}
  Let $M$ and $N$ be matroids, and $C$ and $D$ be disjoint sets, not both infinite. If $M \con C = N \del D$,
  then there is a matroid $P$ with $P \del D = M$ and $P \con C = N$. 
\end{theorem}
\begin{proof}
  By duality, it suffices to prove this for $D$ finite. If $D = \es$, then $P = M$ works. 
  Otherwise, let $e \in D$. Since $M \con X = (N \del \{e\}) \del (D - \{e\})$, 
  by induction there exists a matroid $P'$ with $P' \del (D - \{e\}) = M$ and $P' \con C = N \del \{e\}$. 
  
  Now $N$ is an extension of $N \del \{e\} = P' \con C$ by an element $e$.
  Let $\cF$ be the corresponding modular cut of $P' \con C$, and 
  $P$ be the associated extension of $P'$ by $e$ given by Theorem~\ref{contractmodularcut}, 
  so $P \del e = P'$ and $P \con C = N$. 
  Now 
  \[P \del D = (P - \{e\}) \del (D - \{e\}) = P' \del (D - \{e\}) = M,\] 
  so $P$ is the required matroid. 
\end{proof}

\begin{theorem}\label{projectseq}
  Let $M_0, M_1, \dotsc, M_n$ be matroids with ground set $E$. 
  If $M_{i+1}$ is a single-element projection of $M_i$ for all $0 \le i < n$, 
  then there is a matroid $P$ and an $n$-element set $X \ss E(P)$ such that $P \del X = M_0$ and $P \con X = M_n$. 
\end{theorem}
\begin{proof}
  This is immediate for $n = 0$; suppose that $n > 0$ and it holds for smaller $n$. 
  Inductively, there is a matroid $P'$ and an $(n-1)$-element set $X'$ with $P' \del X' = M_0$ and $P' \con X' = M_n$. 
  Let $e \notin X' \cup E$; since $M_n$ is a single-element project of $M_{n-1}$, there is
  a matroid $Q$ with $Q \del e = M_{n-1}$ and $Q \con e = M_n$. 

  Since $Q \del e = M_{n-1} = P' \con X'$, Theorem~\ref{majorofminor} gives that there is a matroid $P$ 
  for which $P \del e = P'$ and $P \con X' = Q$. Now $P \con X = P \con X' \con e = Q \con e = M_n$, 
  and $P \del X = P \del e \del X' = P' \del X' = M_0$, as required. 
\end{proof} 

We can now prove Theorem~\ref{lambdaminproj}; it is enough to show just the dual version
of the statement. 

\begin{theorem}
  Let $\cX$ be a partition of the ground set of a matroid $M$. 
  Then $\lambda_M^*(\cX)$ is the minimum cardinality of a set $K$
  such that there is a matroid $P$ for which $P \del K = M$, 
  and $\cX$ is skew in $P \con K$. 
\end{theorem}
\begin{proof}
  Let $P$ be a matroid and $K$ 
  be a set such that $P \del K = M$ and $\cX$ is skew in $P \con K$. 
  Note that $\cX$ is also skew in $(P \con K)^* = P^* \del K$ by 
  Lemma~\ref{lambdazero}. 
  Lemma~\ref{mccon} now gives that 
  \[\lambda^*_M{\cX} = \lambda_{M^*}(\cX) = \lambda_{P^* \con K}(\cX) \le \lambda_{P^*}(\cX) \le \lambda_{P^* \del K}(\cX) + |K| = |K|,\]
  so $\lambda^*_M(\cX)$ does not exceed the stated minimum. 

  In the other direction, it is enough to show that if $\lambda_{M^*}(\cX) = k < \infty$, 
  then there is a matroid $P$ and a $k$-element set $K$ with $P \del K = M$ 
  and $\cX$ skew in $P \con K$. For each $n \in \bN$, let $M_n = \gutsproj{M}{\cX}{n}$.
  By construction, each $M_{i+1}$ is a single-element projection of $M_i$.
  By Theorem~\ref{projectseq}, there is a matroid $P$ and a $k$-element set $K \ss E(P)$
  such that $P \del K = M$ and $P \con K = M_k$. 
  By Theorem~\ref{lambdadualeq}, the set $\cX$ is skew in $M_k$, and the result follows. 
\end{proof}

\section{Modular Flats}

Recall that a flat $F$ is $\emph{modular}$ in a matroid $M$ if $(F,F')$ is a modular pair in $M$ for every flat $F'$ of $F$.

\begin{lemma}\label{contractonto}
  Let $F$ be a flat of a matroid $M$. Then $F$ is modular in $M$ if and only if, 
  for every $C \subseteq E(M)$ and every nonloop $e$ of $M \con C$ with $e \in \cl_{M \con C}(F-C)$, 
  there exists $f \in F$ that is parallel to $e$ in $M \con C$. 
\end{lemma}
\begin{proof}
  Suppose that $F$ is a modular flat of $M$, let $C \subseteq E(M)$,
  and let $e$ be a nonloop of $M \con C$ with $e \in \cl_{M \con C}(F - C) = \cl_M(F \cup C)$. 
  Let $F' = \cl_M(C \cup \{e\})$.
  If $F \cap F' \ss \cl_M(C)$, then since $\cl_M(C) \ss C'$ and $F - F', F' - F$ 
  are skew in $M \con (F \cap F')$ by Lemma~\ref{modulartoskew}, 
  the sets $F - C$ and $F' - C$ are skew in $M \con C$.
  Since $e \in \cl_{M \con C}(F-C)$ and $e \in \cl_{M \con C}(F' - C)$ by construction, 
  it follows that $e$ is a loop of $M \con C$, a contradiction. 
  Therefore $F \cap F'$ is not a subset of $\cl_M(C)$; let $f \in F \cap F' - \cl_M(C)$. 
  Now $f \in F$, and $f \in \cl_M(C \cup \{e\}) - \cl_M(C)$ by construction, 
  so $f$ is parallel to $e$ in $M \con C$, as required. 

  Conversely, suppose that the minor condition holds, and $F'$ be an arbitrary flat of $M$. 
  Let $B$ be a mutual basis in $M$ for the chain $\{F \cap F', F, F \cup F'\}$.
  If $B \cap F'$ is a basis for $F'$, then $(F, F')$ is a modular pair in $M$. 
  Otherwise, let $C = B \cap F'$ and $e \in F' - \cl_M(C)$. 
  Now $e$ is a nonloop of $M \con C$, 
  and $e \in \cl_M(B) = \cl_{M \con C}(B \cap (F - F')) \ss \cl_{M \con C}(F - C)$, 
  so there exists $f \in F$ that is parallel to $e$ in $M \con C$. 
  Therefore $f \in F \cap \cl_M(C \cup \{e\}) \ss F \cap F' = \cl_M(B \cap F \cap F') \ss \cl_M(C)$, 
  which contradicts the fact that $f$ is a nonloop of $M \con C$. 
\end{proof}

This gives that the property of modularity of a flat is essentially preserved by contraction. 

\begin{corollary}\label{modularcontract}
  If $F$ is a modular flat of a matroid $M$, 
  then $\cl_{M \con C}(F - C)$ is a modular flat of $M \con C$ for all $C \subseteq E(M)$. 
\end{corollary}
\begin{proof}
  Let $C' \ss E(M \con C)$ and let $e$ be a nonloop of $(M \con C) \con C'$
  for which $e \in \cl_{(M \con C) \con C'}(\cl_{M \con C}(F-C) - C')$. 
  Then $e \in \cl_{M \con (C \cup C')}(F - (C \cup C'))$, 
  and now the result easily follows from the modularity of $F$ and Lemma~\ref{contractonto}.
\end{proof}

\begin{corollary}\label{linemodularcontract}
  Let $M$ be a matroid, and let $C \ss E(M)$. If every line of $M$ is modular, 
  then every line of $M \con C$ is modular. 
\end{corollary}
\begin{proof}
  Let $L$ be a line of $M \con C$, and $I$ be a basis for $L$ in $M \con C$. 
  Then $I$ is independent in $M$, so the line $\cl_M(I)$ is modular in $M$, 
  and therefore $L = \cl_{M \con C}(I)$ is modular in $M \con C$ by Corollary~\ref{modularcontract}. 
\end{proof}

We now prove some characterizations of modularity that are more lattice-theoretic in 
flavour. The results were all proved by Sachs [\ref{sachs}]
in the setting of the lattice of flats of a finitary matroid. 
(Specifically, the lattice assumption ([\ref{sachs}], Definition 2 (7)) is tantamount 
to the assertion that the associated matroid is finitary.)

We say that two flats $F,F'$ are \emph{complementary} in $M$ if $F \cup F'$ is spanning in $M$, and $F \cap F' = \cl_M(\es)$. 

\begin{lemma}\label{modulariffcompl}
  Let $F$ be a flat of a matroid $M$. The following are equivalent: 
  \begin{enumerate}[(i)]
    \item\label{micm} $M$ is modular;
    \item\label{micc} every complementary flat of $F$ in $M$ is skew to $F$;
    \item\label{micdj} every flat $G$ of $M$ with $F \cap G \ss \cl_M(\es)$ is skew to $F$. 
  \end{enumerate}
\end{lemma}
\begin{proof}
  Suppose that $F$ is modular, and let $G$ be a flat of $M$ with $F \cap G \ss \cl_M(\es)$.
  Let $B$ be a mutual basis for $F$ and $G$ in $M$. 
  Since $B$ is independent, the sets $B \cap F$ and $B \cap G$ are disjoint, 
  and hence $F = \cl_M(B \cap F)$ and $G = \cl_M(B \cap G)$ are skew. 
  Hence (\ref{micm}) implies (\ref{micdj}). It is also immediate that (\ref{micdj}) implies (\ref{micc}). 
  

  \begin{claim}
    (\ref{micc}) implies (\ref{micdj}).
  \end{claim}
  \begin{subproof}
    Let $G$ be a flat of $M$ with $F \cap G \ss \cl_M(\es)$. 
    By Lemma~\ref{modularcompl}, there is a flat $H$ of $M$ such that $(\cl_M(F \cup G),H)$ is a modular pair, 
    while $H \cap \cl_M(F \cup G) = G$ and $\cl_M(H \cup \cl_M(F \cup G)) = E(M)$. 
    We have $H \cap F = H \cap \cl_M (F \cup G) \cap F = G \cap F = \cl_M(\es)$, 
    and since $G \ss H$ we have $\cl_M(H \cup F) \sps \cl_M(H \cup G \cup F) = E(M)$,
    so $F$ and $H$ are complementary and are therefore skew. Since $G \ss H$, the flats $F$ and $G$ are skew. 
  \end{subproof}

  It now remains to show the (\ref{micdj}) implies (\ref{micm}); suppose that (\ref{micdj}) holds. 
  We use Lemma~\ref{contractonto} to argue that $F$ is modular; let $C \ss E(M)$
  and let $e$ be a nonloop of $M \con C$ for which $e \in \cl_{M \con C}(F-C)$;
  therefore $e \in \cl_M(F \cup C) - \cl_M(C)$.  
  Suppose that no $f \in F$ is parallel to $e$ in $M \con C$. 
  Let $B$ be a mutual basis in $M$ for the chain $\{F, F \cup C\}$, 
  and let $I = (B \cap (C - F))$. By construction, we have $\cl_M(I) \cap F = \cl_M(\es)$. 
  \begin{claim}\label{techss}
    $\cl_M(I \cup \{e\}) \cap F \ss \cl_M(\es)$.
  \end{claim}
  \begin{subproof}
    Suppose that $f \in \cl_M(I \cup \{e\}) \cap F$ with $f \notin \cl_M(\es)$. 
    If $f \in \cl_M(I)$, then $f \in F \cap \cl_M(I) = \cl_M(\es)$, a contradiction. 
    Therefore $f \in \cl_M(I \cup \{e\}) - \cl_M(I)$, so $e \in \cl_M(I \cup \{f\})$. 
    If $f \in \cl_M(C)$, it follows that $e \in \cl_M(C)$, a contradiction. 
    Therefore $f \notin \cl_M(C)$, so $f \in \cl_M(I \cup \{e\}) - \cl_M(C) \ss \cl_M(C \cup \{e\}) - \cl_M(C)$. 
    It follows that $f$ and $e$ are parallel in $M \con C$, again a contradiction.
  \end{subproof}
  By the two claims, the flats $\cl_M(I \cup \{e\})$ and $F$ are skew in $M$. 
  Applying Lemma~\ref{maxskew} with the sets $(I_1, I_2) = (I, B-I)$ and $(X_1, X_2) = (\cl_M(I \cup \{e\}), F)$, 
  we conclude that $\cl_M(I \cup \{e\}) \ss \cl_M(I)$ and so $e \in \cl_M(C)$, 
  a contradiction. 
\end{proof}

\begin{lemma}\label{modularfinite}
  Let $F$ be a flat of a loopless matroid $M$.
  \begin{itemize} 
    \item If $r_M(F) < \infty$, then $F$ is modular if and only if $F$ intersects every flat $G$ with $r(M \con G) = r_M(F) - 1$ , and
    \item If $r(M \con F) < \infty$, then $F$ is modular if and only if $F$ intersects every flat $G$ with 
    $r_M(G) = r(M \con F) + 1$.
  \end{itemize}
\end{lemma}
\begin{proof}
  In all cases, we will use the characterization of modularity from Lemma~\ref{modulariffcompl}, 
  and the fact that a modular flat $F$ of $M$ is skew to every flat $G$ for which $F \cap G = \es$. 
  The lemma is trivial if $r_M(F) = 0$; suppose first, therefore, that $r_M(F) = k +1 < \infty$. 
  
  Assume that $F$ is modular. 
  If $G$ is a flat with $r(M \con G) = k$ that is disjoint from $F$, 
  then $G$ and $F$ are skew, so $k + 1 = r_M(F) = r((M \con G) | F) \le r(M \con G) = k$, a contradiction. 
  So $F$ intersects every such $G$, as required. 
  
  Conversely, assume that $F$ intersects every flat $G$ with $r(M \con G) = k$; and let $G'$ be a complementary flat to $F$. 
  Let $B$ be a basis for $F \cup G'$ such that $B \cap G'$ is a basis for $G'$. Since $F \cup G'$ is spanning, 
  the set $B$ is a basis for $M$. If $|B \cap F| \ge k+1$, then $B \cap F$ is a basis for $F$ in $M$;
  since $F$ and $G'$ are disjoint, the basis $B$ certifies that $F$ and $G'$ are skew in $M$. 
  Otherwise, since $B \cap F$ spans $M \con (B \cap G')$, we have $r(M \con G') = r(M \con (B \cap G')) \le |B \cap F| \le k$, 
  so $G'$ contains a flat $G_0$ with $r(M \con G_0) = k$. 
  Now $F$ intersects $G_0$, and therefore $G'$, by assumption. 

  Now suppose that $r(M \con F) = k < \infty$. 
  If $M$ is modular, and $G$ is a flat with $r_M(G) = k+1$ that is disjoint from $F$,
  then $F$ and $G$ are skew, so $k + 1 = r_M(G) = r((M \con F) | G) \le r(M \con F) = k$, a contradiction. 
  So $F$ intersects every such $G$, as required. 

  Now assume that $F$ intersects every flat $G$ with $r_M(G) = k+1$, and let $G'$ be a complementary flat to $F$. 
  Let $B$ be a basis for $F \cup G'$ such that $B \cap F$ is a basis for $F$. 
  Since $F \cup G'$ is spanning, the set $G'$ spans $M \con F$ so satisfies $r_M(G') \ge r_{M \con F}(G') = k$. 
  Since $F$ does not intersect $G'$, we have $r_M(G') < k+1$, so equality holds. 
  It follows that $M|G' = (M \con F)|G'$, so $F$ and $G'$ are skew. Therefore $F$ is modular in $M$. 
\end{proof}

\begin{lemma}\label{modulardistrib1}
  Let $F$ be a flat of a matroid $M$. Then $F$ is modular in $M$ if and only if, for every pair $G_1,G_2$
  of flats of $M$ with $G_1 \ss G_2$, we have $\cl_M(G_1 \cup F) \cap G_2 = \cl_M(G_1 \cup (F \cap G_2))$. 
\end{lemma}
\begin{proof}
  For both directions, we use the equivalence in Lemma~\ref{contractonto}. 

  First suppose that $F$ is modular in $M$. The fact that $\cl_M(G_1 \cup (F \cap G_2)) \ss \cl_M(G_1 \cup F) \cap G_2$ 
  is immediate from the monotonicity of the closure function.  
  Let $x \in \cl_M(G_1 \cup F) \cap G_2$, and suppose for a contradiction that $x \notin \cl_M(G_1 \cup (F \cap G_2))$. 
  Then $x \in \cl_M(G_1 \cup F) - \cl_M(G_1)$, so there exists $y \in F$
  that is parallel to $x$ in $M \con G_1$. Now $y \in F \cap \cl_M(G_1 \cup \{x\}) \ss F \cap G_2$, 
  and therefore $x \in \cl_{M \con G_1}(\{y\}) = \cl_M(G_1 \cup \{y\}) \ss \cl_M(G_1 \cup (F \cap G_2))$, 
  a contradiction. 

  Conversely, suppose that $\cl_M(G_1 \cup F) \cap G_2 = \cl_M(G_1 \cup (F \cap G_2))$ for all pairs of flats $G_1 \ss G_2$. 
  Let $C \ss E(M)$ and $e \in \cl_M(F \cup C) - \cl_M(C)$. 
  Let $G_1 = \cl_M(C)$ and $G_2 = \cl_M(C \cup \{e\})$. 
  If $F \cap G_2 \ss G_1$, then by the choice of $G_1$ and $G_2$ and assumption, we have
  \[e \in \cl_M(G_1 \cup F) \cap G_2 = \cl_M(G_1 \cup (F \cap G_2)) = \cl_M(G_1) = \cl_M(C),\]
  a contradiction. Therefore there exists $f \in (F \cap G_2) - G_1 = F \cap \cl_M(C \cup \{e\}) - \cl_M(C)$. 
  Now $f$ and $e$ are parallel in $M \con e$, so we can conclude that $F$ is modular. 
\end{proof}

\begin{lemma}\label{modulardistribpair}
  Let $F$ be a flat of a matroid $M$ and let $X \ss E(M)$. Then $(F,X)$ is a modular pair in $M$ 
  if and only if every flat $G$ of $M$ with $G \ss F$ satisfies $F \cap \cl_M(X \cup G) = \cl_M((F \cap X) \cup G)$. 
\end{lemma}
\begin{proof}
  Suppose first that $(F,X)$ is a modular pair. 
  The inclusion $\cl_M((F \cap X) \cup G) \ss F \cap \cl_M(X \cup G)$ is immediate. 
  Let $e \in F \cap \cl_M(X \cup G)$, 
  and suppose for a contradiction that $e \notin \cl_M(F \cap (X \cup G))$. 
  Let $C = (F \cap X) \cup G$, so $e$ is a nonloop of $M \con C$. 

  Using $G \ss F$, the definition of $C$, and the fact that $e \notin C$, we have 
  \begin{align*}
      e \in \cl_M(X \cup G) - C
      = \cl_M((X-F) \cup (F \cap (X \cup G))) - C
      = \cl_{M \con C}(X-F).
  \end{align*} 
  Using the fact that $F$ is a flat, we also have 
  \begin{align*}
    e \in F - C = \cl_M\br{(F - (X \cup G)) \cup (F \cap (X \cup G))} - C = \cl_{M \con C}(F - (X \cup G)).
  \end{align*}
  Note that $(C_1, C_2) = (F \cap X, G - (F \cap X))$ is a partition of $C$. 
  By Lemma~\ref{modulartoskew}, the sets $F - X$ and $X - F$ are skew in $M \con C_1$.
  We have $C_2 \ss F - (F \cap X) = F- X$,
  so by Lemma~\ref{skewmono}(\ref{skewcontract}), the sets $X-F$ and $F-X-C_2$ are skew in the matroid
  $M \con C_1 \cup C_2 = M \con C$. But it is routine to check that $F - X - C_2 = F - (X \cup G)$, 
  so we have concluded that $e$ is a nonloop of $M \con C$ that is spanned by two sets 
  in the skew pair $(X-F, F-X-C_2)$. This is a contradiction.  

  Conversely, suppose that every flat $G$ of $M$ with $G \ss F$ satisfies $F \cap \cl_M(X \cup G) = \cl_M((F \cap X) \cup G)$. 
  Let $B$ be a mutual basis for the chain $\{F \cap X, X, F \cup X\}$. 
  Using the fact that $B \cap (F \cup X)$ spans $F \cup X$ and $B \cap F \cap X$ spans $F \cap X$,
  and our assumption with $G = \cl_M(B \cap F)$, we have 
  \begin{align*}
    F &= F \cap (F \cup X)\\ 
    &\ss F \cap \cl_M(X \cup \cl_M(B \cap F)) \\
    &= \cl_M((F \cap X) \cup \cl_M(B \cap F)) \\
    & = \cl_M((B \cap F \cap X) \cup (B \cap F)) = \cl_M(B \cap F)
  \end{align*}
  It follows that $B$ is a mutual basis for $X$ and $F$, as required.   
\end{proof}

This gives us another lattice-like characterization of modular flats. 

\begin{corollary}\label{modulardistrib2}
  Let $F$ be a flat of a matroid $M$. Then $F$ is a modular flat of $M$ if and only if, 
  for every pair of flats $F_0, G$ of $M$ with $F_0 \ss F$, we have $F \cap \cl_M(G \cup F_0) = \cl_M((F \cap G) \cup F_0)$. 
\end{corollary}

\begin{corollary}\label{modularfiniteinter}
  If $\cF$ is a finite collection of modular flats of a matroid $M$, then $\cap \cF$ 
  is a modular flat of $M$. 
\end{corollary}
\begin{proof}
It suffices by induction 
    to show that the intersection of two modular flats $F,F'$ of $M$ is a modular flat.
    Let $G$ be a flat of $M$, and $F_0$ be a flat of $M | F \cap F'$. 
    Then by applying Corollary~\ref{modulardistrib2} twice, we get
    \begin{align*}
        F \cap F' \cap \cl_M(G \cup F_0) &= F \cap \cl_M((F' \cap G) \cup F_0) \\ 
        &= \cl_M\br{(F \cap (F' \cap G)) \cup F_0},
    \end{align*}
    so, by Corollary~\ref{modulardistrib2}, the flat $F \cap F'$ is modular in $M$, as required. 
\end{proof}

Combining Corollary~\ref{modulardistrib2} with the statements of 
Lemmas~\ref{modulartoskew}, ~\ref{contractonto}, ~\ref{modulariffcompl} and ~\ref{modulardistrib1} 
gives a number of characterizations of modularity for flats. 

\begin{theorem}\label{modequiv}
  Let $F$ be a flat of a matroid $M$. The following are equivalent: 
  \begin{enumerate}[(1)]
    \item $F$ is a modular flat of $M$;
    \item $F-G$ and $G-F$ are skew in $M \con (F \cap G)$ for every flat $G$ of $M$;
    \item $F$ and $G$ are skew in $M \dcon (F \cap G)$ for every flat $G$ of $M$;
    \item every flat $G$ with $F \cap G \ss \cl_M(\es)$ is skew to $F$;
    \item every complementary flat of $F$ in $M$ is skew to $F$;
    \item
    for every $C \subseteq E(M)$ and every nonloop $e$ of $M \con C$ with $e \in \cl_{M \con C}(F-C)$, 
    there exists $f \in F$ that is parallel to $e$ in $M \con C$;
    \item 
    for every pair of flats $F_0, G$ of $M$ with $F_0 \ss F$, we have $F \cap \cl_M(G \cup F_0) = \cl_M((F \cap G) \cup F_0)$;
    \item 
    for every pair $G_1,G_2$
    of flats of $M$ with $G_1 \ss G_2$, we have $\cl_M(G_1 \cup F) \cap G_2 = \cl_M(G_1 \cup (F \cap G_2))$;
  \end{enumerate}
\end{theorem}

\section{Intersections of Modular Flats}

This section proves that an infinite intersection of modular flats in a finitary matroid is a modular flat; 
this will be needed to prove Theorem~\ref{modularfinitary}.
These proofs are reproductions of ones from Sachs [\ref{sachs}] 
that are modified to use our language, but unlike those in the last section
they add no generality, since they only deal with finitary matroids.
We include them here only for completeness.  

A collection $\cX$ of sets is \emph{directed} if for all $X,X' \in \cX$, 
there exists $Y \in \cX$ for which $X \cup X' \ss Y$. 
(And, thus, for all finite $\cX_0 \ss \cX$, there exists $X \in \cX$ with 
$\cup \cX_0 \ss X$.)

The first lemma shows that the lattice of flats of a finitary matroid is `meet-continuous'. 
The finitary hypothesis is required; otherwise, a counterexample would be the
circuit matroid on ground set $[0,1]$, where $F = \{1\}$, and $\cG$ is the family
of flats of the form $[0,a) : a < 1$. 

\begin{lemma}\label{interclosuredirected}
  Let $F$ be a flat of a finitary matroid $M$, and let $\cG$ be a directed family 
  of flats of $M$. Then $F \cap \cl_M(\cup \cG) = \cl_M(\cup_{G \in \cG} (F \cap G))$. 
\end{lemma}
\begin{proof}
  It is easy to see that $\cl_M(\cup_{G \in \cG}(F \cap G)) \ss F \cap \cl_M(\cup \cG)$. 
  Suppose, therefore, that $e \in F \cap \cl_M(\cup \cG)$. 
  Since $M$ is finitary, there is therefore a finite set $I \ss \cup \cG$ 
  for which $e \in \cl_M(I)$; 
  let $\cG_0$ be a finite subset of $\cG$ for which $I \ss \cup \cG_0$. 
  Since $\cG$ is directed and $\cG_0$ is finite, there exists $G \in \cG$ with 
  $\cup \cG_0 \ss G$. Thus $e \in \cl_M(I) \ss \cl_M(G) = G$. 

  Since $e \in F$, it follows that $e \in F \cap G \ss \cl_M(\cup_{G \in \cG}(F \cap G))$, 
  as required. 
\end{proof}

\begin{lemma}\label{modpairdirected}
  Let $F$ be a flat of a finitary matroid $M$, and let $\cX$ be an directed family 
  of subsets of $E(M)$. If $(F,X)$ is a modular pair for all $X \in \cX$,
  then $(F, \cl_M(\cup \cX))$ is a modular pair. 
\end{lemma}
\begin{proof}
  Let $Z = \cl_M(\cup \cX)$. By Lemma~\ref{modulardistribpair}, it suffices to show 
  that $F \cap \cl_M(Z \cup G) = \cl_M((F \cap Z) \cup G)$ for every flat $G$ of $M | F$. 

  Let $G$ be a flat of $M | F$, and let $\cG = \{\cl_M(X \cup G) : X \in \cX\}$. 
  It is easy to check that, since $\cX$ is upwards-directed, so is $\cG$. 
  We also have $\cl_M(Z \cup G) = \cl_M(\cup_{X \in \cX} (X \cup G)) = \cl_M(\cup \cG)$, 
  so Lemmas~\ref{modulardistribpair} and~\ref{interclosuredirected} give 
  \begin{align*}
      F \cap \cl_M(Z \cup G) &= F \cap \cl_M(\cup \cG) \\
      &= \cl_M\br{\cup_{G \in \cG}(F \cap G)} \\ 
      &= \cl_M\br{\cup_{X \in \cX} (F \cap \cl_M(X \cup G))} \\ 
      &= \cl_M\br{\cup_{X \in \cX}\cl_M((F \cap X) \cup G)}\\
      &= \cl_M\br{\cup_{X \in \cX}((F \cap X) \cup G)}\\
      &= \cl_M\br{(F \cap \cup \cX) \cup G}
  \end{align*}
  This implies that $F \cap \cl_M(Z \cup G) \ss \cl_M((F \cap Z) \cup G)$, 
  and the reverse inclusion is immediate from the fact that $G \ss F$. 
\end{proof}

\begin{corollary}
  If $\cF$ is a directed family of modular flats of a finitary matroid $M$. 
  then $\cl_M(\cup \cF)$ is modular in $M$. 
\end{corollary}

The collection of finite-rank flats of a matroid $M$ is directed, 
and every flat is the closure of a union of rank-one flats; 
this implies the following. 

\begin{corollary}\label{modularoffinite}
  If $F$ is a flat of a finitary matroid $M$, and $(F,G)$ is a modular pair for 
  every finite-rank flat $G$ of $M$, then $F$ is modular in $M$. 
\end{corollary}

This will be enough to show that modular flats in finitary matroids are closed under arbitrary intersections.
We do not know a reason that the finitary hypothesis is necessary. 

\begin{lemma}\label{modularinter}
  If $\cF$ is a family of modular flats of a finitary matroid $M$, then $\cap \cF$ is modular in $M$. 
\end{lemma}
\begin{proof}
  By Corollary~\ref{modularoffinite}, it is enough to show that $(\cap \cF, K)$ is a modular pair for each 
  finite-rank flat $K$. 
  By Lemma~\ref{modulardistribpair}, it suffices to prove that every flat $G$ of $M | \cap \cF$ 
  satisfies $\cap \cF \cap \cl_M(K \cup G) = \cl_M((\cap \cF \cap K) \cup G)$. 
  
  Let $\cF_K$ be the collection of sets of the form $\cap\cF_0 \cap K$, where $\cF_0$ is a nonempty finite subset of $\cF$. 
  Clearly $\cap \cF_K = \cap \cF \cap K$; but also we see that the intersection of any two flats in $\cF_K$ is in $\cF_K$. 
  Since $\cF_K$ is a family of subflats of the finite-rank flat $K$, it follows that $\cF_K$ contains its intersection; 
  that is, there is some finite set $\cF_0 \ss \cF$ for which $\cap \cF_0 \cap K = \cap \cF_K = \cap \cF \cap K$. 
  
  Let $F_0 = \cap \cF_0$; by Lemma~\ref{modularfiniteinter}, the flat $F_0$ is modular in $M$.
  Lemma~\ref{modulardistribpair} gives
  \[\cap\cF \cap \cl_M(K \cup G) \ss F_0 \cap \cl_M(K \cup G) = \cl_M((F_0 \cap K) \cup G) = \cl_M((\cap \cF \cap K) \cup G).\]
  The inclusion $\cl_M((\cap \cF \cap K) \cup G) \ss \cap\cF \cap \cl_M(K \cup G)$ is easy; the lemma follows. 
\end{proof}

\section{Modular Matroids}

In this section, we consider matroids in which every flat is modular.
We need a few helper lemmas.

\begin{lemma}\label{modcircuit}
  Let $C$ be a circuit of a matroid $M$, and $(I,J)$ be a partition of $C$ with $|I| \ge 2$ and $|J| \ge 2$. 
  If $(\cl_M(I), \cl_M(J))$ is a modular pair in $M$, then there exists $e \in E(M) - C$ such that $I \cup \{e\}$ and $J \cup \{e\}$
  are both circuits. 
\end{lemma}
\begin{proof}
  Note that $I$ and $J$ are both proper subsets of $C$, so each is independent. It follows from Theorem~\ref{lcprop}
  and the fact that $J$ is a circuit of $M \dcon I$ that 
  $\sqcap_M(\cl_M(I), \cl_M(J)) = \sqcap_M(I,J) = n_{M \dcon I}(J) = 1$. 
  Since $(\cl_M(I), \cl_M(J))$ is a modular pair, it follows from Theorem~\ref{lcmod} that $r_M(\cl_M(I) \cap \cl_M(J)) = 1$. 
  Let $e$ be a nonloop in $\cl_M(I) \cap \cl_M(J)$. 

  For each proper subset $K$ of $I$, since $K \cup J$ is independent, we have $\sqcap_M(K,J) = 0$ 
  and so $\cl_M(K) \cap \cl_M(J) \ss \cl_M(\es)$, which implies that $e \notin \cl_M(K)$. 
  Therefore $e$ is spanned by $I$ but by no proper subset of $I$; since $|I| \ge 2$, it follows that $e \notin I$ and that $I \cup \{e\}$ is a circuit. 
  By the same argument, $J \cup \{e\}$ is also a circuit. 
\end{proof}

\begin{lemma}\label{modularaux}
  Let $M$ be a matroid in which every line is modular. Let $a_i : i \in \bN$ be distinct elements of $M$,
  and $A = \{a_i : i \in \bN\}$. Let $x,y$ be distinct elements of $M \del A$ such that $A \cup \{x,y\}$ is a circuit. 
  Then there exist elements $b_i : i \in \bN$ of $M \del (A \cup \{x\})$ such that
  \begin{enumerate}[(1)]
    \item\label{mab0} $b_0 = y$,
    \item\label{mabt} for each $k > 1$, the triple $\{b_{i-1}, a_{i-1}, b_i\}$ is a triangle of $M$,
    \item\label{mabc} for each $k \in \bN$, the set $\{x, b_k\} \cup \{a_i : i \ge k\}$ is a circuit. 
  \end{enumerate}
\end{lemma}
\begin{proof}
  Let $A_k = \{a_k : k \ge i\}$ for each $k \in \bN$. 
  It suffices to show that if $t \ge 0$ and elements $y = b_0, \dotsc, b_t$ have been defined to satisfy the above conditions 
  for all $k \le t$, 
  then there is some $b = b_{t+1}$ such that 
  $b \notin A \cup \{x\} \cup \{b_0, \dotsc, b_t\}$ and
  the sets $\{b_t, a_t, b\}$ and $\{x, b\} \cup A_{t+1}$ are both circuits; 
  this will allow us to define all terms in the sequence by recursion. 

  Since $\{b_t, a_t\} \cup \{x\} \cup A_{t+1}$ is a circuit by hypothesis, 
  and $\cl_M(\{b_t,a_t\})$ is modular in $M$, Lemma~\ref{modcircuit}
  gives that there exists $b \in E(M)$ for which $\{b,b_t,a_t\}$ and
  $\{x, b\} \cup A_{t+1}$ are both circuits; it remains to show that
  $b \notin A \cup \{x\} \cup \{b_0, \dotsc, b_t\}$. 
  
  If $b \in A \cup \{x\}$, then $\{x,b\} \cup A_{t+1}$ is a circuit properly contained in 
  $A \cup \{x,y\}$, a contradiction. If $b = b_\ell$ for some $\ell \le t$, then 
  $\{x, b\} \cup A_{t+1}$ is a circuit properly contained in the circuit $\{x,b\} \cup A_\ell$, 
  also a contradiction. 
\end{proof}

We now prove a stronger version of Theorem~\ref{modularfinitary}.

\begin{theorem}\label{modularlinefinitary}
  Every matroid in which every line is modular is finitary. 
\end{theorem}
\begin{proof}
  Let $M'$ be a matroid in which every line is modular, and let $C'$ be an infinite circuit of $M'$. Let $x,y \in C'$, 
  and $(a_i : i \in \bN)$ be distinct elements of $C' - \{x,y\}$. Let $A = \{a_i : i \in \bN\}$, 
  so $A \cup \{x,y\}$ is a circuit of $M = M \con (C' - \{x,y\} \cup A)$. 
  
  By Corollary~\ref{linemodularcontract}, every line of $M$ is modular. 
  Let $(b_i : i \in \bN)$ be elements given by Lemma~\ref{modularaux} for the $a_i$ and $x,y$. Let $B = \{b_i : i \in \bN\}$ and let $A_k = \{a_i : i \ge k\}$ for each $k \in \bN$.
  For each $i \in \bN$, let $C_i = \{x, b_i\} \cup A_i$, noting that $C_i$ is a circuit. 
  Since $y = b_0 \in \cl_M(B)$, and $a_i \in \cl_M\{b_i, b_{i+1}\} \in \cl_M(B)$
  for each $i$, we have $x \in \cl_M(\{y\} \cup A) \ss \cl_M(B)$. 
  There is therefore a circuit $C$ with $x \in C \ss B \cup \{x\}$. 
  
  Since $x$ is not a loop, we have $C \ne \{x\}$; 
  let $s \in \bN$ and $I \ss \{s+1, s+2, \dotsc, \}$ be such that $C = \{x, b_s\} \cup \{b_i : i \in I\}$. 
  Since all $a_i$ and $b_j$ are distinct and not equal to $x$,
  we have $b_i \in C_j$ if and only if $i = j$. 

  By the infinite circuit elimination axiom (see \ref{bdkpw}, Section 1.4) 
  applied to the circuit $C$, 
  the element $b_s$, and the elements $(b_i : i \in I)$ 
  of the circuits $(C_i : i \in I)$, 
  we see that $M$ has a circuit $C'$ of $M$ for which
  \[C' \ss C \cup \br{\cup_{i \in I} C_i} - \{b_i : i \in I\} = \br{\cup_{i \in I} (C_i- \{b_i\})} \cup \{x, b_s\} .\] 
  
  Since $C_i - \{b_i\} \ss \{x\} \cup A_i \ss \{x\} \cup A_{s+1}$ for each $i$,
  this gives 
  \[C' \ss A_{s+1} \cup \{x, b_s\} = C_s - \{a_s\}.\] 
  Therefore $C'$ is a proper subset of the circuit $C_s$, a contradiction. 
\end{proof}

\begin{theorem}
  Let $M$ be a loopless matroid. Then $M$ is modular if and only if every line of $M$
  intersects every hyperplane of $M$. 
\end{theorem}
\begin{proof}
  By Lemma~\ref{modularfinite},
  a line of $M$ is modular if and only if it intersects 
  every hyperplane, and a hyperplane is modular if and only if it intersects every line. 
  The forwards direction is immediate. 

  Conversely, suppose that every line of $M$ intersects every hyperplane,
  so every line and every hyperplane is modular. 
  By Theorem~\ref{modularlinefinitary}, we know that $M$ is finitary.
  Since every flat $F$ of $M$ is the intersection of the hyperplanes containing $F$, 
  Lemma~\ref{modularinter} implies that every flat of $M$ is modular, so $M$ is modular. 
\end{proof}

\section*{Acknowledgements}
We thank Nathan Bowler for being a frequent and patient source of information on infinite matroids, and 
for his comments on a draft version of this paper. We also thank 
Johannes Carmesin, Jim Geelen, Zach Walsh and Geoff Whittle for helpful conversations. 

\section*{References}

\newcounter{refs}

\begin{list}{[\arabic{refs}]}
{\usecounter{refs}\setlength{\leftmargin}{10mm}\setlength{\itemsep}{0mm}}

\item\label{bs09}
J. E. Bonin, W. R. Schmitt,
Splicing matroids,
Eur. J. Combin. 32 (2011), 722--744. 

\item\label{bc18}
N. Bowler, J. Carmesin,
An excluded minors method for infinite matroids,
J. Combin. Theory Ser. B 128 (2018), 104--113. 

\item\label{bg}
N. Bowler and S. Geschke, 
Self-dual uniform matroids on infinite sets, 
Proc. Amer. Math. Soc. 144 (2016), 459 -- 471

\item\label{bd11}
H. Bruhn, R. Diestel, 
Infinite matroids in graphs, 
Discrete Math. 311 (2011), 1461--1471. 

\item\label{bdkpw}
H. Bruhn, R. Diestel, M. Kriesell, R. Pendavingh, P. Wollan, 
Axioms for infinite matroids,
Adv. Math. 239 (2013), 18--46. 

\item\label{bw12}
H. Bruhn, P. Wollan, 
Finite connectivity in infinite matroids, 
Eur. J. Combin. 33 (2012), 1900--1912.  

\item\label{crapo65}
H. H. Crapo,
Single-element extensions of matroids, 
J. Res. Nat. Bur. Standards Sect. B 69B (1965), 55--65. 

\item\label{edmonds}
J. Edmonds, 
Matroid intersection,
Ann. Discr. Math 4 (1979), 39--49.

\item\label{hk96}
W. Hochst\"attler, S. Kromberg,
Adjoints and duals of matroids linearly representable over a skewfield,
Mathematica Scandanavica 78 (1996), 5--12. 

\item\label{ggw06}
J. Geelen, B. Gerards, G.Whittle,
Matroid $T$-connectivity, 
SIAM J. Discrete Math. 20 (2006), 588--596.

\item\label{gj25}
J. Gollin, A. Jo\'o,
Wild generalised truncation of infinite matroids,
arXiv:2504.05064 [math.CO]

\item\label{gk13}
J. Geelen, R. Kapadia, 
Representation of matroids with a modular plane, 
arXiv:1304.6448 [math.CO]

\item\label{gn}
J. Geelen, P. Nelson,
The structure of matroids with a spanning clique or projective geometry,
J. Combin. Theory Ser. B 127 (2017), 65-81.

\item\label{higgs69}
D.A. Higgs, 
Equicardinality of bases in $B$-matroids, 
Can. Math. Bull 12 (1969), 861--862. 

\item\label{higgs}
D.A. Higgs,
Matroids and duality,
Colloq. Math. 20 (1969), 215 -- 220. 

\item\label{k14}
R. Kapadia,
Matroids with a modular $4$-point line, 
SIAM J. Discrete Math. 28 (2014), 862--877. 

\item\label{lean4}
L.M. de Moura, S. Ullrich, 
The Lean 4 theorem prover and programming language. 
In: CADE (2021).

\item\label{o19}
J. Oxley, 
A matroid extension result,
SIAM J. Discrete Math. 33 (2019), 138--152. 

\item\label{mathlibpaper}
The mathlib Community,
The Lean mathematical library. In: Proceedings of the
9th ACM SIGPLAN International Conference on Certified Programs and Proofs,
CPP 2020, pp. 367–381. Association for Computing Machinery, New York (2020)

\item\label{mathlib}
The mathlib Community,
\smalltt{mathlib4},
https://github.com/leanprover-community/mathlib4,
(GitHub repository)

\item\label{matroidrepo}
P. Nelson,
\smalltt{Matroid},
https://github.com/apnelson1/Matroid,
(GitHub repository)

\item\label{oxley}
J. Oxley, 
Matroid Theory (second edition), 
Oxford University Press, 2011. 

\item\label{sachs}
D. Sachs,
Partition and modulated lattices,
Pacific J. Math. 11 (1961), 325--345.

\item\label{whitney}
H. Whitney, 
On the abstract properties of linear dependence,
Amer. J. Math. 57 (1935), 509--533. 

\end{list}

\newpage

\input{appendices.tex}

\end{document}

%% file: matroidpreamble.tex
\usepackage{amsmath,amssymb,enumerate,mathtools,xcolor}
\usepackage[pdftex,pagebackref=false, hidelinks]{hyperref} 
\theoremstyle{plain}
\newtheorem{theorem}{Theorem}[section]
\newtheorem{claim}{}[theorem]
\newtheorem{lemma}[theorem]{Lemma}
\newtheorem{problem}[theorem]{Problem}

\newtheorem{corollary}[theorem]{Corollary}

\newcounter{defn}
\theoremstyle{definition}
\newtheorem{definition}[defn]{Definition}

\newcommand{\bK}{\mathbb K}
\newcommand{\bZ}{\mathbb Z}
\newcommand{\bN}{\mathbb N}

\newcommand{\cB}{\mathcal{B}}
\newcommand{\cC}{\mathcal{C}}

\newcommand{\cF}{\mathcal{F}}

\newcommand{\cG}{\mathcal{G}}
\newcommand{\cL}{\mathcal{L}}

\newcommand{\cI}{\mathcal{I}}
\newcommand{\cH}{\mathcal{H}}
\newcommand{\cM}{\mathcal{M}}

\newcommand{\cU}{\mathcal{U}}

\newcommand{\cX}{\mathcal{X}}
\newcommand{\cY}{\mathcal{Y}}
\newcommand{\cZ}{\mathcal{Z}}

\newcommand{\dcon}{/ \!\! /}

\newcommand{\enn}{\mathbb{N}^{\infty}}

\DeclareMathOperator{\cl}{cl}

\DeclareMathOperator{\PG}{PG}

\DeclareMathOperator{\GF}{GF}

\newcommand{\del}{\!\setminus\!}
\newcommand{\con}{/}
\newcommand{\wh}{\widehat}

\newcommand{\abs}[1]{\left| #1 \right|}

\newcommand{\set}[1]{\left\{#1\right\}}
\newcommand{\br}[1]{\left(#1\right)}
\newcommand{\ab}[1]{\left\langle#1\right\rangle}

\newcommand{\es}{\varnothing}
\renewcommand{\le}{\leqslant}
\renewcommand{\ge}{\geqslant}
\renewcommand{\ss}{\subseteq}
\newcommand{\sps}{\supseteq}

\newcommand{\setof}[2]{\set{{#1} : {#2}}}

\newenvironment{subproof}[1][\proofname]{%
  \begin{proof}[Subproof:]%
}{%
  \end{proof}%
}

%% file: leanpreamble.tex
\usepackage[T1]{fontenc}
\usepackage[utf8]{inputenc}
\usepackage{lmodern}
\usepackage{amssymb}
\usepackage[frozencache]{minted}
\usepackage{newunicodechar}
\setminted{
  encoding=utf8,
  texcomments=true,
}
\newcommand{\smalltt}[1]{{\small\texttt{#1}}}
\setminted{fontsize=\footnotesize}
\newminted[leancode]{lean}{fontsize=\footnotesize}
\newunicodechar{✝}{\freeserif{✝}}
\newunicodechar{𝓞}{\ensuremath{\mathcal{O}}}
\newunicodechar{∀}{\ensuremath{\forall}}
\newunicodechar{∈}{\ensuremath{\in}}
\newunicodechar{∃}{\ensuremath{\exists}}
\newunicodechar{ℕ}{\ensuremath{\mathbb{N}}}
\newunicodechar{↔}{\ensuremath{\leftrightarrow}}
\newunicodechar{𝓕}{\ensuremath{\mathcal{F}}}
\newunicodechar{𝓒}{\ensuremath{\mathcal{C}}}
\newunicodechar{＼}{\ensuremath{\backslash}}
\newunicodechar{／}{\ensuremath{/}}
\newunicodechar{⊓}{\ensuremath{\sqcap}}
\newunicodechar{✶}{*}
\newunicodechar{∧}{\ensuremath{\wedge}}
\newunicodechar{∨}{\ensuremath{\vee}}
\newunicodechar{ᶜ}{\ensuremath{^c}}
\newunicodechar{₁}{\ensuremath{_\texttt{1}}}
\newunicodechar{₂}{\ensuremath{_\texttt{2}}}
\newunicodechar{α}{\ensuremath{\alpha}}
\newunicodechar{↦}{\ensuremath{\mapsto}}
\newunicodechar{∉}{\ensuremath{\notin}}
\newunicodechar{⊆}{\ensuremath{\subseteq}}
\newunicodechar{⊂}{\ensuremath{\subset}}
\newunicodechar{⊇}{\ensuremath{\supseteq}}
\newunicodechar{⊃}{\ensuremath{\supset}}
\newunicodechar{∩}{\ensuremath{\cap}}
\newunicodechar{∪}{\ensuremath{\cup}}
\newunicodechar{≤}{\ensuremath{\leq}}
\newunicodechar{≥}{\ensuremath{\geq}}
\newunicodechar{≠}{\ensuremath{\neq}}
\newunicodechar{ℤ}{\ensuremath{\mathbb{Z}}}
\newunicodechar{ℚ}{\ensuremath{\mathbb{Q}}}
\newunicodechar{ℝ}{\ensuremath{\mathbb{R}}}
\newunicodechar{ℂ}{\ensuremath{\mathbb{C}}}
\newunicodechar{λ}{\ensuremath{\lambda}}
\newunicodechar{β}{\ensuremath{\beta}}
\newunicodechar{γ}{\ensuremath{\gamma}}
\newunicodechar{ι}{\ensuremath{\gamma}}
\newunicodechar{⋃}{\ensuremath{\bigcup}}
\newunicodechar{⋂}{\ensuremath{\bigcap}}
\newunicodechar{₀}{\ensuremath{_\texttt{0}}}
\newunicodechar{∅}{\ensuremath{\varnothing}}
\newunicodechar{↾}{\ensuremath{|}}
\newunicodechar{⊤}{\ensuremath{\top}}
\newunicodechar{⊣}{\ensuremath{\dashv}}
\newunicodechar{≃}{\ensuremath{\simeq}}
\newunicodechar{≈}{\ensuremath{\approx}}
\newunicodechar{≡}{\ensuremath{\equiv}}
\newtheorem*{fdef}{Formal Definition}
\newtheorem*{fthm}{Formal Theorem}

%% file: appendices.tex
\section*{Appendix A : Formalization}

\subsection*{Formalized Proofs}
A proof assistant (such as \smalltt{lean4}) is a piece of software 
that can check an appropriately written proof for correctness. 
Preparing a proof in a form that can be checked is nontrivial, 
even when the proof is fully available in a human-readable form,
but the benefits are clear. To be convinced of the truth of 
a mathematical statement whose proof has been formalized, 
one needs only to verify that the \emph{statement}
of the theorem is formalized correctly, and to trust that the proof 
assistant is genuinely able to check proofs for correctness. 
The actual proof need not be read, either in code
form or in normal mathematical prose. 
The process of preparing and checking proofs is quickly gaining
interest, is becoming more and more ergonomic,
and can be performed with only modestly powerful consumer hardware. 

Although there are subtleties in the task of correctly formalizing
statements, and there are questions in principle about the trust
placed in a proof assistant itself, one only needs to trust \emph{one}
proof assistant; the authors believe that one can be orders of magnitude 
more confident in the correctness of a formalized proof than one
that has simply been peer-reviewed, especially if the proof is long 
and technical. 

Aside from issues of correctness, formalized statements in principle
allow for semantically searchable, centralized databases of mathematics, 
which is a boon to research for many reasons. The community-maintained 
\smalltt{mathlib4} [\ref{mathlib}]
library of formalized mathematics, which contains millions of lines of 
code and hundreds of thousands of formalized lemmas and theorems
in fields all across the discipline, is a prominent example. 

It is our opinion that, for these reasons and more,
formalization will become an essential component of the health of 
an area of mathematics going forward. 

\subsection*{Matroids}

This paper arose from a systematic effort to place matroid theory 
on a formally verifiable foundation using \smalltt{lean4}; 
this has resulted in several tens of thousands of lines of code.

At the time of writing, approximately 20\% of
this code appears in the \smalltt{Data/Matroid} subfolder of
\smalltt{mathlib4}, 
and the remainder in the second author's \smalltt{matroid} repository [\ref{matroidrepo}]. 
The total amount of formalized material in [\ref{matroidrepo}] is much larger
than what is covered in this article; for instance, it 
contains formalized versions of the matroid intersection theorem 
and Tutte's excluded minor theorem for finitary infinite matroids. 

It was decided early in the process to allow our formalized matroids to be infinite, 
and to maintain the infinite setting whenever possible. 
The original motivation for this was simply future-proofing and increased generality.
We quickly realized, however, that this called for mathematical novelty
that went beyond the matter of just producing formalized versions of existing results. 
There is no standard text on infinite matroids, and texts on finite matroids
tend to rely on finite reasoning methods.
Formalizing does not allow one to skip `obvious' statements, 
so to make infinite matroids work from the ground up,
our process thus grew to include developing new mathematics. 

Modularity, skewness and local connectivity, which are all useful concepts
for finite matroids, were parts of matroid theory where new ideas were especially required;
these ideas, and the results that came out of them,
are what were described in this article. 

We now aim to show the benefits of formalizing by  
providing the means for an interested reader to be convinced 
of the correctness of all our main results, even without checking the proofs. 

We will not go into detail on the actual process of formalizing a large 
body of mathematics. This is a difficult task; as well as mathematical
knowledge, it requires its own brand of craft and creativity, 
as well as `design' skills in that are rather analogous software engineering.
Formalizing matroid theory seems to have given rise to some insights 
and mathematical design patterns that are useful in other areas,
and we plan to discuss this all in future work. 

\subsection*{Using a proof assistant}
Since our code is written in \texttt{lean4}, we briefly describe the formalization process
using this language; other proof assistants are broadly similar. 
Our explanation is intended for a reader simply wanting to check that existing
results have correct formal proofs. For anything more, 
a website containing introductory information for those with a mathematics 
background can be found at {\small\url{https://leanprover-community.github.io/learn.html}}. This page also contains instructions for how to download the relevant
software, and `clone' a repository like [\ref{matroidrepo}], which
is the way to check our results on one's own device. 

Proofs in \smalltt{lean4} are intended to be read and written only in the context
of an integrated desktop environment (IDE) such as \textit{VS Code}. 
While the user creates or reads a proof, 
the software will provide an `infoview' that displays 
the mathematical assumptions and variables that are in play,
in human-readable form. 
The code that comprises the proof itself is a sequence of `tactics'
that can appear quite mathematically opaque to the uninitiated.

Mathematical \emph{statements} in \texttt{lean4}, on the other hand, are usually quite 
readable. They can be written to use mostly standard mathematical notation, 
with a familiar mix of logical connectives, quantifiers, subscripts 
and varying typefaces, all displaced in unicode. In all, they 
are typically much easier to parse than the raw \LaTeX \ code in which 
most of us are quite conversant. For instance, the following is the statement 
of the hard direction of Edmonds' matroid intersection theorem [\ref{edmonds}]:
for any matroids $M_1, M_2$ with $M_1$ finite, there is a common independent 
set $I$ of both the matroids for which $|I| = r_{M_1}(X) + r_{M_2}(E(M_2)-X)$. 
\footnote{The statement omits the standard assumptions that 
$E(M_1) = E(M_2)$ and that $M_2$ is finite; as it turns out, they aren't needed.}

\begin{leancode}
  theorem exists_common_ind (M₁ M₂ : Matroid α) [M₁.Finite] :
    ∃ I X, M₁.Indep I ∧ M₂.Indep I ∧ I.ncard = M₁.rk X + M₂.rk (M₂.E \ X)
\end{leancode}

And \emph{statements} being readable is what is important. A lean file containing errors
in the proofs will not be accepted by a proof assistant 
-- warnings and compilation errors will display in a desktop environment, 
and building the file via a command line will fail. So to be convinced of
correctness, it is enough to carefully read statements like the above appearing in 
a file, and to check that the file is accepted by the proof assistant. 

We collect the formalized definitions and statements relevant to this paper in two files in [\ref{matroidrepo}]: 
\smalltt{InfiniteModular/Definitions} and \smalltt{InfiniteModular/Theorems}. 
They are separated deliberately for ease of inspection and are both relatively short files, 
intended to be readable to a novice. 
However, they are by no means self-contained, 
and implicitly depend on the wider contents of the repository via references.
To run the software required to check these files, 
one must \href{https://docs.lean-lang.org/lean4/doc/quickstart.html}{download the lean4 software} along with the VS Code 
IDE, 
\href{https://docs.github.com/en/repositories/creating-and-managing-repositories/cloning-a-repository}{clone} the repository in [\ref{matroidrepo}], and open the relevant files in VS Code; the Lean infoview
will display when their correctness has been verified. 

\section*{Appendix B : Formal Definitions}
Our formalized theorem statements are contingent on many definitions, 
and we start by collecting them in one place. 
The \smalltt{InfiniteModular/Definitions} 
file~\footnote{\scriptsize \url{https://github.com/apnelson1/Matroid/blob/main/InfiniteModular/Definitions.lean}}
in [\ref{matroidrepo}] is specifically tailored to a reader of this paper, 
containing precisely the relevant definitions and no more, 
phrased in elementary terms with appropriate annotations. 
Since they are automatically checked for correctness when viewed in an IDE,
this is likely the best way to view these definitions,
but we also present them here with brief annotations. 

The normal way to define an object or predicate in Lean 4 is with the \smalltt{def}
keyword, and all the concepts we discuss are ultimately defined this way somewhere
in [\ref{matroidrepo}]. 
However, it is not technically convenient to move these definitions 
all to one file, or to duplicate them all using \smalltt{def}, so we go with 
a different option to convincingly argue that they are formulated correctly. 
The \smalltt{example} keyword allows one to formally 
state and prove an unnamed assertion, as below : 
\begin{leancode}
  example (a : ℕ) (a_less_than_two : a < 2) : 
    a + 2 < 4 := by 
      linarith
\end{leancode}
The first line contains the relevant variables and assumptions 
(in this case, a natural number $a$, and an assumption that $a < 2$,
which is given an arbitrary name), 
followed by a colon. The second line contains the statement being asserted 
(i.e. $a + 2 < 4$), followed by a \smalltt{:=} symbol, 
and then the line that follows is the proof.
(Here, the proof is a single invocation of a `linear arithmetic' tactic,
but this is not important for us.) 

This setup is enough to produce a formalized statement that a definition is
capturing what is intended. 
For instance, if a matroid $M$ were defined in terms of its bases, 
the following statement shows that the \smalltt{M.Indep} predicate correctly
corresponds to the notion of independence.  
\begin{leancode}
  example (M : Matroid α) (I : Set α) : 
    M.Indep I ↔ ∃ B : I ⊆ B ∧ M.IsBase B := 
    -- Proof here
\end{leancode}
Note that this is just an \textit{assertion} about the \smalltt{Indep} predicate, 
rather than its actual definition. 
Despite this, one can be sure,
solely by knowing that a proof of the above statement 
is accepted by Lean,
that independence is defined correctly,
without observing the actual definition of \smalltt{Indep}. 
(In this particular instance, the proof likely \emph{is}
simply the statement that this holds by the definition of independence, 
but even if independence were originally defined another way, 
the above reasoning would still be valid.)

The above example also shows some of the standard idioms for
discussing matroid theory in Lean. 
The variable \smalltt{(M : Matroid α)} introduces $M$
as a matroid whose elements belong to the `type' $\alpha$, 
which has itself been previously introduced as a global variable, 
and can be thought of as the set of all possible matroid elements. 
The variable \smalltt{(I : Set α)} introduces $I$ as a set of 
elements of the type $\alpha$. (The set $I$ could possibly contain 
elements outside the ground set of $M$, in which case 
the equivalence being asserted will still hold, with both sides being false). 

The statements \smalltt{M.Indep I} and \smalltt{M.IsBase B} 
state that $I$ and $B$ respectively satisfy the predicates of being 
independent in $M$, and being a base (basis) for $M$.
Predicates are often more convenient to work with set membership, 
so we prefer `has a property' statements like \smalltt{M.Indep I} over 
formalized analogues of set membership statements like $I \in \cI_M$. 
Finally, we note that the variable \smalltt{B} is not introduced 
with a type, though, like \smalltt{I}, it should have type \smalltt{Set α}. 
This is because Lean can infer the intended type of \smalltt{B}
via the subsequent statement \smalltt{I ⊆ B}.

This mechanism is enough to present formalized versions of (nearly) all 
new definitions in this paper. Some (such as the $\sqcap_M$ parameter 
for families of non-disjoint sets) are omitted, as we do need them in 
our main theorem statements. 

We also do not explicitly provide the formal
definitions for more basic matroidal notions already appearing in \smalltt{mathlib4}
such as \smalltt{Matroid}, \smalltt{Indep}, and \smalltt{Finitary}; 
these have already undergone the community review process for 
contributions, and we will thus assume that they are correct.  

In all the formalized statements below, we omit the proofs completely. 
The proofs all appear in the \smalltt{Definitions} file.
As might be expected, these proofs are mostly very short, essentially amounting to 
a formalized version of the statement that `this is the same definition from elsewhere'.

\begin{fdef}[Mutual Basis]
  A mutual basis for a set family $\ab{X_i : i \in \iota}$ 
  is an independent set $B$ intersecting each $X_i$ in a basis for $X_i$. 
  Here, the term \smalltt{X} is formally a function from the type \smalltt{ι}
  to the type of sets of \smalltt{α}, 
  and the function application $X_i$ is written as \smalltt{X i}.
  \begin{leancode}
example (M : Matroid α) (B : Set α) (X : ι → Set α) :
  M.IsMutualBasis B X ↔ M.Indep B ∧ ∀ i, M.IsBasis (X i ∩ B) (X i)
  \end{leancode}
\end{fdef}

\begin{fdef}[Modular Family]
  A modular family is a set family in $M$ that has a mutual basis. 
  \begin{leancode}
example (M : Matroid α) (B : Set α) (X : ι → Set α) :
  M.IsMutualBasis B X ↔ M.Indep B ∧ ∀ i, M.IsBasis (X i ∩ B) (X i)
  \end{leancode}
\end{fdef}

\begin{fdef}[Modular Pair]
  A modular pair is a pair with a mutual basis. 
  \begin{leancode}
example (M : Matroid α) (X Y : Set α) :
  M.IsModularPair X Y ↔ 
  ∃ B, M.Indep B ∧ M.IsBasis (X ∩ B) X ∧ M.IsBasis (Y ∩ B) Y
  \end{leancode}
\end{fdef}

\begin{fdef}[Modular Flat]
  A flat is modular if it forms a modular pair with every flat. 
  \begin{leancode}
example (M : Matroid α) (F : Set α) :
  M.IsModularFlat F ↔
  M.IsFlat F ∧ ∀ G, M.IsFlat G → M.IsModularPair F G
  \end{leancode}
\end{fdef}

\begin{fdef}[Modular Matroid]
  A matroid is modular if all its flats are modular. 
  \begin{leancode}
example (M : Matroid α) :
  M.Modular ↔ ∀ F, M.IsFlat F → M.IsModularFlat F
  \end{leancode}
\end{fdef}

\begin{fdef}[Modular Cut]
  This definition is different from the previous ones, 
  since \smalltt{ModularCut} is defined as a `packaged' object
  in its own right
  (in this case, a set of sets of $\alpha$, 
  together with a guarantee that they satisfy the properties
  of a modular cut) 
  rather than being a property of a set of sets. Because of this, 
  we need two separate statements to informally certify that the definition is correct. 
  
  Our first formalized
  statement says that given a collection $\cF$ of sets in $M$ satisfying
  the definition of a modular cut gives rise to a \smalltt{ModularCut} called $\cF'$.
  The object $\cF'$ can be thought of as $\cF$, together with a proof that $\cF$ is a modular cut. 
  \begin{leancode}
example (M : Matroid α) (𝓕 : Set (Set α))
  (forall_isFlat : ∀ F ∈ 𝓕, M.IsFlat F)
  (forall_superset : 
    ∀ F F', F ∈ 𝓕 → M.IsFlat F' → F ⊆ F'→ F' ∈ 𝓕)
  (forall_inter : 
    ∀ F F', F ∈ 𝓕 → F' ∈ 𝓕 → M.IsModularPair F F' → F ∩ F' ∈ 𝓕)
  (forall_inter_chain : 
    ∀ 𝓒 ⊆ 𝓕, 𝓒.Infinite → M.IsModularFamily (fun C : 𝓒 ↦ C.1)
      → IsChain (· ⊆ ·) 𝓒 → ⋂₀ 𝓒 ∈ 𝓕) :
  ∃ 𝓕' : M.ModularCut, ∀ F, F ∈ 𝓕 ↔ F ∈ 𝓕' 
  \end{leancode}
  We could be `cheating' here, since we could have just defined 
  \smalltt{ModularCut} to be an arbitrary unrestricted collection of sets.
  The second statement 
  shows that we are not; it states that a \smalltt{ModularCut} 
  always satisfies the rules it is meant to (in fact, it shows
  the stronger statement of closure under modular families.)
  \begin{leancode}
example (M : Matroid α) (𝓕 : M.ModularCut) :
  (∀ F ∈ 𝓕, M.IsFlat F)
  ∧ (∀ F F', F ∈ 𝓕 → M.IsFlat F' → F ⊆ F' → F' ∈ 𝓕)
  ∧ (∀ (F : ι → Set α), (∀ i, F i ∈ 𝓕) → M.IsModularFamily F → ⋂ i, F i ∈ 𝓕)
  \end{leancode}
\end{fdef}

\begin{fdef}[Extension by a Modular Cut]
  The statement shows that the extension of $M$ by $e$ 
  using the modular cut $\cF$, written
  \smalltt{M.extendBy e 𝓕}, has \smalltt{M} as its deletion, 
  and that its closure function relates correctly to $\cF$. 
  \begin{leancode}
example (M : Matroid α) (e : α) (e_not_mem : e ∉ M.E) (𝓕 : M.ModularCut) :
  (M.extendBy e 𝓕) ＼ {e} = M ∧ 
  ∀ F, M.IsFlat F → (e ∈ (M.extendBy e 𝓕).closure F ↔ F ∈ 𝓕)
  \end{leancode}
  Despite appearing to be only a definition, 
  this statement contains nontrivial mathematical content, 
  since it is treating \smalltt{M.extendBy e 𝓕} as a matroid, 
  which is only valid if this object has actually been proven to be a matroid. 
  This is one direction of the content of Theorem~\ref{finmodcutextension}.
\end{fdef}

\begin{fdef}[Projection by a Modular Cut]
  The projection of $M$ by a modular cut $\cF$ is the contraction of $e$
  from the extension of $M$. 
  \begin{leancode}
example (M : Matroid α) (e : α) (e_not_mem : e ∉ M.E) (𝓕 : M.ModularCut) :
  (M.extendBy e 𝓕) ／ {e} = M.projectBy 𝓕
  \end{leancode}
\end{fdef}

\begin{fdef}[Quotient] 
  If $E(M) = E(N)$, then $N$ is a quotient of $M$ if 
  $\cl_M(X) \ss \cl_N(X)$ for all $X \ss E(M)$.
  \begin{leancode}
example (M N : Matroid α) (ground_eq : M.E = N.E) :
  N ≤q M ↔ ∀ X ⊆ M.E, M.closure X ⊆ N.closure X :=f
  \end{leancode}
\end{fdef}

\begin{fdef}[Skew Family]
  An indexed family of sets is skew if it is modular,
  and the intersection of any pair is spanned by the empty set. 
  \begin{leancode}
example (M : Matroid α) (X : ι → Set α) :
  M.IsSkewFamily X ↔ 
    M.IsModularFamily X ∧ ∀ i j, i ≠ j → X i ∩ X j ⊆ M.closure ∅
  \end{leancode}
\end{fdef}

\begin{fdef}[Nullity]
  The nullity of a set is equal to $r^(M| X)$, 
  where the rank is viewed as a term in $\enn$. 
  \begin{leancode}
example (M : Matroid α) (X : Set α) :
  M.nullity X = (M ↾ X)✶.eRank
  \end{leancode}
\end{fdef}

\begin{fdef}[Connectivity for Set Families]
  The connectivity of a disjoint set family is the nullity of 
  a union of its bases. Note that this is an assertion about \emph{any} collection of 
  bases, so independence on the choice
  of bases is implicit in the truth of this statement.
  \begin{leancode}
example (M : Matroid α) (X I : ι → Set α)
  (disjoint : Pairwise (Disjoint on X))
  (basis : ∀ i, M.IsBasis (I i) (X i)) :
  M.multiConn X = M.nullity (⋃ i, I i)
  \end{leancode}
\end{fdef}

\begin{fdef}[Local Connectivity for Pairs]
  The local connectivity of a pair is the nullity of a union of bases
  plus the cardinality of an intersection of bases.
  \begin{leancode}
example (M : Matroid α) (I J X Y : Set α) (I_basis_X : M.IsBasis I X) 
  (J_basis_Y : M.IsBasis J Y) :
  M.eLocalConn X Y = (I ∩ J).encard + M.nullity (I ∪ J)
  \end{leancode}
\end{fdef}

\begin{fdef}[Projection by a set]
  The projection of a matroid $M$ by a set $X$ has the same ground set
  as $M$, and the same independent sets as $M \con X$. 
  \begin{leancode}
example (M : Matroid α) (X : Set α) :
  (M.project X).E = M.E ∧ ∀ I, (M.project X).Indep I ↔ (M ／ X).Indep I
  \end{leancode}
\end{fdef}

\begin{fdef}[Discrepancy]
  If $N$ is a finitary quotient of $M$, then 
  the discrepancy of $X$ is the cardinality of the set difference of a 
  basis pair for $X$. The independence of the choice of basis pair is implicit 
  in the truth of the statement. 
  \begin{leancode}
example (M N : Matroid α) (finitary : N.Finitary) (quotient : N ≤q M) 
  (X I J : Set α) (I_basis : N.IsBasis I X) (J_basis : M.IsBasis J X) 
  (I_subset_J : I ⊆ J) :
  quotient.discrepancy X = (J \ I).encard
  \end{leancode}
\end{fdef}

\section*{Appendix C : Formalized Theorems}

Our main theorems all appear and can be checked for correctness 
in the \smalltt{InfiniteModular/Theorems}
file
\footnote{\scriptsize \url{https://github.com/apnelson1/Matroid/blob/main/InfiniteModular/Theorems.lean}} 
of [\ref{matroidrepo}], 
which is again specifically designed to be inspected by a reader of this paper, 
with explicitness and simplicity in mind. The statements below omit 
the proofs entirely, and the file is designed to make the proofs appear very short. 
(In fact, they are simply short references to longer proofs appearing 
elsewhere in the repository.)

We do not include all results with short proofs and long statements
(such as the equivalences in Theorem~\ref{quotprojectequiv})
or those are stated in this paper only for expository reasons,
though nearly all our results appear in [\ref{matroidrepo}] in some form.

Like with our definitions, we also use \smalltt{example} to state theorems. 

\begin{fthm}[Theorem~\ref{finmodcutextension}]
  An extension is uniquely determined by its modular cut. 
  The assumptions state that $P$ is some extension of $M$ by $e$, 
  for which the flats of $P \del \{e\}$ spanning $e$ are those 
  \begin{leancode}
example (M P : Matroid α) (e : α) (e_mem : e ∈ P.E) (𝓕 : M.ModularCut)
  (P_del_eq : P ＼ {e} = M) 
  (P_flat_iff : ∀ F, M.IsFlat F → (e ∈ P.closure F ↔ F ∈ 𝓕)) :
  P = M.extendBy e 𝓕
  \end{leancode}
  As mentioned earlier, the remainder of this theorem's content 
  is implicitly encoded in the fact that \smalltt{M.extendBy e 𝓕}
  is a matroid. 
\end{fthm}

\begin{fthm}[Theorem~\ref{quotientisprojectintro}]
  If $N$ or $M^*$ is finitary, and $N$ is a quotient of $M$ with finite discrepancy,
  then $N$ is a projection of $M$. 
  \begin{leancode}
example (N M : Matroid α) (finitary : Finitary N ∨ Finitary M✶) 
  (quot : N ≤q M) (hfin : quot.discrepancy M.E < ⊤) 
  (h_inf : M.Eᶜ.Infinite) :
  ∃ (P : Matroid α) (X : Set α), X ⊆ P.E ∧ P ／ X = N ∧ P ＼ X = M
  \end{leancode}
  The term $\top$ denotes the infinity element in $\enn$, and
  the assumption \smalltt{h\_inf} simply states that there are infinitely many 
  non-elements of the ground set; we need these elements to be available 
  for the set $X$ to exist. This is required by Lean's type theory, 
  but is of course mathematically benign. 
\end{fthm}

\begin{fthm}[Theorem~\ref{bigprojofseq}]
  If $M_0, \dotsc, M_n$ are matroids where each is a projection of previous one, 
  then $M_n$ is an $n$-element projection of $M_0$. 
  \begin{leancode}
example {n : ℕ} (M : Fin (n+1) → Matroid α) 
  (𝓕 : (i : Fin n) → (M i.castSucc).ModularCut)
  (projection : ∀ i, M i.succ = (M i.castSucc).projectBy (𝓕 i))
  {X : Finset α} (hX : X.card = n) (hdj : Disjoint (M 0).E X) :
  ∃ (P : Matroid α), 
    (X : Set α) ⊆ P.E ∧ P ＼ X = M 0 ∧ P ／ X = M (Fin.last n) 
  \end{leancode}
  Statements involving finite indexing can be a little technical, 
  due to the rigidity of Lean's type theory. 
  Here, \smalltt{Fin n} denotes the set $\{0, \dotsc, n-1\}$, 
  and \smalltt{castSucc} and \smalltt{succ} are respectively the 
  `inclusion' and `successor'
  functions from \smalltt{Fin n} to \smalltt{Fin (n+1)},
  so \smalltt{castSucc 0 = 0} and \smalltt{succ 0 = 1}. 
  Thus, $\mathcal{\cF}$ is a sequence $\cF_0, \dotsc, \cF_{n-1}$ 
  where each $\cF_i$ is a modular cut of $M_i$, 
  and \smalltt{projection} is the assertion that each $M_{i+1}$ is the projection
  of $M_i$ by $\cF_i$. The last few assumptions just state that there 
  is a finite set of the appropriate size outside the ground set 
  with which the projection can be performed. 
\end{fthm}

\begin{fthm}[Theorem~\ref{splice}]
If $C$ and $C$ are disjoint sets, not both infinite,
and $M$ and $N$ are matroids with $M ／ C = N ＼ D$,
then there is a matroid $P$ with $P ＼ D = M$ and $P ／ C = N$.
\begin{leancode}
example (M N : Matroid α) (C D : Set α) (hfin : C.Finite ∨ D.Finite) 
  (disjoint : Disjoint C D)
  (con_eq_del : M ／ C = N ＼ D) : ∃ P, P ＼ D = M ∧ P ／ C = N :=
  exists_splice_of_contract_eq_delete' hfin disjoint con_eq_del
\end{leancode}
\end{fthm}

\begin{fthm}[Theorem~\ref{skewequivintro}]
  The equivalence of the first and last conditions in the theorem: a disjoint 
  family is skew if and only if it is crossed by no circuit. 
  \begin{leancode}
example (M : Matroid α) (X : ι → Set α) (subset_ground : ∀ i, X i ⊆ M.E)
  (disjoint : Pairwise (Disjoint on X)) :
  M.IsSkewFamily X ↔ ∀ C, M.IsCircuit C → C ⊆ ⋃ i, X i → ∃ i, C ⊆ X i
  \end{leancode}
\end{fthm}

\begin{fthm}[Theorem~\ref{gutscutintro}]
  Given a collection $X_i$ of sets with union $E(M)$, 
  a flat $F$ belongs to the guts modular cut of $\ab{X_i}$ 
  if and only if $\ab{X_i}$ is skew in $M \dcon F$. 
  \begin{leancode}
example (M : Matroid α) (X : ι → Set α) (union : ⋃ i, X i = M.E) 
  (F : Set α) (flat : M.IsFlat F) :
  F ∈ M.gutsModularCut X union ↔ (M.project F).IsSkewFamily X
  \end{leancode}
  The real content of Theorem~\ref{gutscutintro} is not in this equivalence per se, 
  but instead encoded in the fact that \smalltt{M.gutsModularCut X union} is actually a valid
  \smalltt{ModularCut} (i.e. its members satisfy the modular cut axioms.)
  It we had not verified that this is the case, the statement above 
  would be rejected by Lean as ill-formed, rather than false. 
\end{fthm}

\begin{fthm}[Theorem~\ref{lambdaminproj}(ii)]
The dual connectivity of a partition $\ab{X_i}$ of $E(M)$
is equal to the minimum size of a set $A$ such that there exists a matroid $P$
for which $P \del A = M$, and $\ab{X_i}$ is skew in $P \con A$. 
(Since $\lambda_{M^*}(\ab{X_i}) = \lambda_{M}(\ab{X_i})$ for every partition 
$\ab{X_i}$, the dual statement in (i) is trivially equivalent.)

As with previous statements, making statements about non-elements of the ground set
together with infima is little complicated; to avoid this, 
we give formal statements of the lower and upper bounds separately. 

The first states that if $A$ is a finite set, disjoint from the ground set of $A$,
with size equal to the dual connectivity of a partition $\ab{X_i}$ of $E(M)$,
then $M$ has a projection by $A$ in which $X$ is skew. -/
\begin{leancode}
example (M : Matroid α) (X : ι → Set α)
  (union_eq_ground : ⋃ i, X i = M.E) (disjoint : Pairwise (Disjoint on X))
  (A : Set α) (hAfin : A.Finite) (hA : A.encard = M✶.multiConn X)
  (disjoint_ground : Disjoint A M.E) :
  ∃ (P : Matroid α), A ⊆ P.E ∧ P ＼ A = M ∧ (P ／ A).IsSkewFamily X
\end{leancode}

The second states that, if $A$ is any set in a matroid $P$ for which $P \del A = M$
and $\ab{X_i}$ is skew in $P \con A$, 
then $A$ has cardinality at least $\lambda_{M^*}(\ab{X_i})$.
\begin{leancode}
example (P M : Matroid α) (X : ι → Set α) (union_eq_ground : ⋃ i, X i = M.E)
  (disjoint : Pairwise (Disjoint on X)) (A : Set α) (delete_eq : P ＼ A = M)
  (contract_skew : (P ／ A).IsSkewFamily X) :
  A.encard ≥ M✶.multiConn X
\end{leancode}
\end{fthm}

\begin{fthm}[Theorem~\ref{modularfinitary}]
  Every modular matroid is finitary. 
  \begin{leancode}
example (M : Matroid α) (h : M.Modular) : M.Finitary
  \end{leancode}
\end{fthm}

\begin{fthm}[Theorem~\ref{modularlinehyperplane}]
  A loopless matroid is modular if and only if every line intersects every hyperplane.
  \begin{leancode}
example (M : Matroid α) (hM : M.Loopless) :
  M.Modular ↔ ∀ L H, M.IsLine L → M.IsHyperplane H → (L ∩ H).Nonempty
  \end{leancode}
\end{fthm}